\documentclass[12pt]{amsart}

\usepackage{tikz-cd}
\usepackage{graphicx}
\usepackage{amssymb}
\usepackage{tikz}
\usetikzlibrary{matrix,arrows.meta,bending}
\usepackage{color}
\usepackage[numbers]{natbib}
\usetikzlibrary{arrows.meta,positioning,fit,calc}
\usepackage{hyperref} 
\usepackage[a4paper, left=3cm, right=3cm, top=3cm, bottom=3cm, footskip=15mm]{geometry}

\newtheorem{theorem}{Theorem}[section]

\newtheorem{lemma}[theorem]{Lemma}
\newtheorem{proposition}[theorem]{Proposition}
\newtheorem{corollary}[theorem]{Corollary}

\theoremstyle{definition}
\newtheorem{definition}[theorem]{Definition}

\theoremstyle{remark}

\numberwithin{equation}{section}

\makeindex

\usepackage{fancyhdr}
\usepackage{calc} 

\AtBeginDocument{%
  \setlength{\textwidth}{\dimexpr\paperwidth - 6cm\relax}%
  \setlength{\oddsidemargin}{\dimexpr 3cm - 1in\relax}%
  \setlength{\evensidemargin}{\dimexpr 3cm - 1in\relax}%
  \setlength{\marginparwidth}{2.5cm}%
}

\begin{document}

\title{Multivariate expansivity theory and Pierce-Birkhoff conjecture}

\author{T. Agama}
\address{Department of Mathematics, African Institute for Mathematical science, Ghana
}
\email{theophilus@aims.edu.gh/emperordagama@yahoo.com}


\date{\today}



\footnote{
\par
.}%
.

\begin{abstract}
Motivated by the Pierce-Birkhoff conjecture, we launch an extension program for single variable expansivity theory. We study this notion for tuples of polynomials in the ring $\mathbb{R}[x_1,x_2,\ldots,x_n]$. As an application, we develop some class of inequalities to study the Pierce-Birkhoff conjecture
\end{abstract}


\maketitle

\begingroup
  \setlength{\parskip}{6pt} 
  \tableofcontents
\endgroup

\section{Introduction}

The Pierce--Birkhoff conjecture is a classical problem in real algebraic geometry and the theory of lattice-ordered rings. In its modern form, it asks whether every continuous piecewise-polynomial function on $\mathbb{R}^n$ can be written as a finite supremum of finite infima of polynomial functions; equivalently, whether the ring of piecewise-polynomial functions is generated from the polynomial ring by iterating the lattice operations $\sup$ and $\inf$ \cite{BirkhoffPierce1956,HenriksenIsbell1962,Mahé1984,Delzell1989}. This problem has guided a substantial literature connecting semialgebraic geometry, the real spectrum, separating ideals, and local representation theory \cite{Delzell1989,Marshall1992,Wagner2010,LucasSchaubSpivakovsky2012}. The present paper is motivated by that circle of ideas, but it develops a different language: instead of working directly with piecewise-polynomial functions, we introduce a formal calculus of directional expansions on tuples of polynomials and study the discrete invariants that arise from their iterates.\\

The central object of the paper is a multivariate expansion operator built from three ingredients: a directional differentiation map, a linear combination operator on tuples, and a coordinate packaging map. Applied repeatedly, this operator produces a hierarchy of expansions in specific directions, mixed directions, and hybrid configurations. The resulting theory is organized around invariants that measure how an expansion behaves under iteration. Among these are the \emph{totient} and \emph{residue} of an expansion, the \emph{Doppler intensity} induced by one expansion on another, the \emph{destabilization stage}, the notions of \emph{diagonalization} and \emph{exactness}, and the \emph{sub-expansion}, \emph{index}, \emph{dominating number}, \emph{kernel}, and \emph{singularity} associated to an expansion. Although the terminology is new, the guiding principle is simple: repeated directional expansion should be tracked by a family of numerical and structural invariants that record when information is lost, when it stabilizes, and when one expansion can be recovered from another.\\

A key theme of the paper is that multivariate and single-variable expansivity theory are not unrelated formalisms, but different manifestations of the same underlying mechanism. In the multivariate setting, one studies tuples of polynomials in $\mathbb{R}[x_1,\dots,x_n]$ and the effect of iterated expansion in one or more directions. In the single-variable setting, the same ideas collapse to a simpler iteration theory in which expansion chains admit a rank, a degree, a limit, and a local number. The present paper develops a bridge between these two viewpoints by showing how specialization in the remaining variables can be used to compare the multivariate and one-dimensional theories. This bridge is particularly important for the later analytic part of the paper, where the notions of normalization, unionization, kernels, singular points, and analyticity are introduced in a way that parallels classical local and global behaviors in real algebraic geometry.\\

The paper is also designed to produce a structured combinatorial calculus for chains of sub-expansions. Once a sub-expansion relation is fixed, one can define an index measuring the number of steps required to pass from one expansion to another, and then study how these indices behave along chains. This leads to domination inequalities, additive relations for chains, and comparison principles between indices, totient, and local numbers. These results are not merely formal: they are intended to provide a way of organizing the complexity of an expansion and of comparing the behavior of different directions within the same tuple.\\

The relevance of this framework to the Pierce--Birkhoff conjecture is conceptual rather than direct. The conjecture itself concerns the representation of piecewise-polynomial functions, but such representation problems are often controlled by local and directional information. By constructing a theory in which expansions can be compared direction by direction, and by attaching to them discrete invariants that measure extinction, stabilization, and analytic behavior, the paper aims to supply a new algebraic language for studying piecewise-polynomial phenomena. In particular, the theory is meant to clarify how multivariate polynomial data can be decomposed into directional components and how those components can be organized into chains, kernels, singularities, and exact sequences of expansions.\\

\subsection{Organization of the paper:} The paper is organized as follows. Section~2 introduces multivariate expansions in mixed and specified directions and proves the first structural properties of the theory, including linearity, commutativity, and basic integral-type inequalities attached to the expansion formalism. Sections~3--11 develop the invariants associated with iterative behavior: the totient and residue, the Doppler effect, destabilization, diagonalization, hybrid expansions, exactness, and sub-expansion chains. Section~11 studies dominating expansions and the index calculus for chains of sub-expansions, and then passes to the analytic side by introducing kernels, singular points, normalization, unionization, and local numbers. Finally, the last part of the paper compares the multivariate and single-variable theories, explains the specialization procedure, and presents applications to additive number theory and to questions motivated by the Pierce--Birkhoff conjecture.

\subsection{Notations}
Throughout this paper, we maintain the usual standard notion $\mathcal{S}$ for all tuples whose entries belong to the ring $\mathbb{R}[x_1, x_2,\ldots, x_n]$. Occasionally, we might choose to index these tuples by $\mathcal{S}_j$ over the natural numbers $\mathbb{N}$ when there are two or more and we want to keep them distinct from each other. The tuples $\mathcal{S}_0=(0,0,\ldots,0)$ and $\mathcal{S}_e=(1,1,\ldots,1)$ are still reserved for the null and unit tuples, respectively. In addition to the above requirements, any polynomial tuple will be assumed to contain exactly $n\geq 2$ entries, and two tuples under the operation of addition or subtraction will be assumed to contain the same number of entries.
\bigskip

\section{Expansion in mixed and specified directions}

In this section, we introduce the notion of an expansion in a \emph{mixed} and \emph{specific} directions. We launch the following extension program:

\begin{definition}
Let $\mathcal{F}:=\{\mathcal{S}_i\}_{i=1}^{\infty}$ be a collection of tuples of polynomials $f_k\in\mathbb{R}[x_1,x_2,\ldots,x_n]$. By an expansion on $\mathcal{S}\in \mathcal{F}:=\{\mathcal{S}_i\}_{i=1}^{\infty}$ in the direction $x_i$ for $1\leq i\leq n$, we mean the composite map
\begin{align}
(\gamma^{-1}\circ\beta\circ\gamma\circ \nabla)_{[x_i]}:\mathcal{F}\longrightarrow \mathcal{F}\nonumber
\end{align}
where 
\begin{align}
\gamma(\mathcal{S})=\begin{pmatrix}f_1\\f_2\\ \vdots \\f_n
\end{pmatrix} \quad \mathrm{and} \quad \beta(\gamma(\mathcal{S}))=\begin{pmatrix}0 &  1 &\cdots 1\\1 & 0 & \cdots 1\\\vdots & \vdots & \cdots \vdots\\ 1 & 1 & \cdots 0\end{pmatrix}\begin{pmatrix}f_1\\f_2\\ \vdots \\f_n
\end{pmatrix}\nonumber 
\end{align}
with 
\begin{align}
\nabla_{[x_i]}(\mathcal{S})=\bigg(\frac{\partial{f_1}}{\partial{x_i}},\frac{\partial{f_2}}{\partial{x_i}},\ldots,\frac{\partial{f_n}}{\partial{x_i}}\bigg).\nonumber
\end{align}
The value of the $l^{th}$ expansion at $a$ of $x_i$ is \begin{align}
(\gamma^{-1}\circ \beta \circ \gamma \circ \nabla)^{l}_{[x_i](a)}(\mathcal{S})\nonumber
\end{align}
where $(\gamma^{-1}\circ\beta\circ\gamma\circ\nabla)^{l}_{[x_i](a)}(\mathcal{S})$ is a tuple of polynomials in $\mathbb{R}[x_1,\ldots,x_{i-1},x_{i+1},\ldots,x_n]$. Similarly, by an expansion in the mixed direction $\otimes_{i=1}^{l}[x_{\sigma(i)}]$, we mean 
\begin{align}
(\gamma^{-1}\circ\beta\circ\gamma\circ\nabla)_{{\otimes}_{i=1}^{l}[x_{\sigma(i)}]}(\mathcal{S})=(\gamma^{-1}\circ\beta\circ\gamma\circ\nabla)_{{\otimes}_{i=2}^{l}[x_{\sigma(i)}]}\circ(\gamma^{-1}\circ\beta\circ\gamma\circ\nabla)_{[x_{\sigma(1)}]}(\mathcal{S})\nonumber
\end{align}
for any permutation $\sigma:\{1,2,\ldots,l\}\longrightarrow \{1,2,\ldots,l\}$. The value of this expansion on a given value $a_i$ of $x_{\sigma(i)}$ for all $i\in [\sigma(1),\sigma(l)]$ \begin{align}
(\gamma^{-1}\circ \beta \circ \gamma \circ \nabla)_{{\otimes}_{i=1}^{l}[x_{\sigma(i)}](a_i)}(\mathcal{S})\nonumber
\end{align}
where $(\gamma^{-1}\circ\beta\circ\gamma\circ\nabla)_{{\otimes}_{i=1}^{l}[x_{\sigma(i)}](a_i)}(\mathcal{S})$ is a tuple of real numbers $\mathbb{R}$.
\end{definition}
\bigskip

\begin{proposition}\label{linearity}
A multivariate expansion is linear.
\end{proposition}

\begin{proof}
It suffices to show that each of the operators $\nabla_{[x_i]}:\{\mathcal{S}_i\}_{i=1}^{\infty}\longrightarrow \{\mathcal{S}_i\}_{i=1}^{\infty}$ for a fixed direction $[x_i]$, $\gamma :\{\mathcal{S}_i\}_{i=1}^{\infty}\longrightarrow \{\mathcal{S}_i\}_{i=1}^{\infty}$ and $\beta \circ \gamma :\{\mathcal{S}_i\}_{i=1}^{\infty}\longrightarrow \{\mathcal{S}_i\}_{i=1}^{\infty}$ is linear, since the map $ \gamma:\{\mathcal{S}_i\}_{i=1}^{\infty}\longrightarrow \{\mathcal{S}_i\}_{i=1}^{\infty}$ is bijective. Let $\mathcal{S}_a=(f_1,f_2,\ldots, f_n), \mathcal{S}_b=(g_1,g_2,\ldots, g_n) \in \mathcal{F}=\{\mathcal{S}_i\}_{i=1}^{\infty}$ and $\lambda, \mu \in \mathbb{R}$. We get 
\begin{align}
\nabla_{[x_i]}(\lambda \mathcal{S}_a+\mu \mathcal{S}_b)&=\nabla(\lambda (f_1,f_2,\ldots,f_n)+\mu (g_1,g_2,\ldots, g_n))\nonumber \\&=\nabla_{[x_i]}((\lambda f_1, \lambda f_2,\ldots,\lambda f_n)+(\mu g_1,\mu g_2,\ldots,\mu g_n))\nonumber \\&=\nabla_{[x_i]}((\lambda f_1+\mu g_1,\lambda f_2+\mu g_2,\ldots \lambda f_n+\mu g_n))\nonumber \\&=(\frac{\partial{(\lambda f_1+\mu g_1)}}{\partial{x_i}}, \frac{\partial{(\lambda f_2+\mu g_2)}}{\partial{x_i}},\ldots, \frac{\partial{(\lambda f_n+\mu g_n)}}{\partial{x_i}})\nonumber \\&=(\lambda \frac{\partial{f_1}}{\partial{x_i}}+\mu \frac{\partial{g_1}}{\partial{x_i}},\lambda \frac{\partial{f_2}}{\partial{x_i}}+\mu \frac{\partial{g_2}}{\partial{x_i}},\ldots, \lambda \frac{\partial{f_n}}{\partial{x_i}}+\mu \frac{\partial{g_n}}{\partial{x_i}})\nonumber \\&=(\lambda \frac{\partial{f_1}}{\partial{x_i}},\lambda \frac{\partial{f_2}}{\partial{x_i}},\ldots,\lambda \frac{\partial{f_n}}{\partial{x_i}})+(\mu \frac{\partial{g_1}}{\partial{x_i}},\mu \frac{\partial{g_2}}{\partial{x_i}},\ldots,\mu \frac{\partial{g_n}}{\partial{x_i}})\nonumber \\&=\lambda (\frac{\partial{f_1}}{\partial{x_i}},\frac{\partial{f_2}}{\partial{x_i}},\ldots, \frac{\partial{f_n}}{\partial{x_i}})+\mu (\frac{\partial{g_1}}{\partial{x_i}},\frac{\partial{g_2}}{\partial{x_i}},\ldots, \frac{\partial{g_n}}{\partial{x_i}})\nonumber \\&=\lambda \nabla_{[x_i]}(\mathcal{S}_a)+\mu \nabla_{[x_i]}(\mathcal{S}_b).\nonumber
\end{align}
Similarly, we get 
\begin{align}
\gamma(\lambda \mathcal{S}_a+\mu \mathcal{S}_b)&=\begin{pmatrix}\lambda f_1+\mu g_1\\\lambda f_2+\mu g_2\\\vdots \\\lambda f_n+\mu g_n\end{pmatrix}\nonumber \\&=\begin{pmatrix}\lambda f_1\\\lambda f_2\\ \vdots \\\lambda f_n  \end{pmatrix}+\begin{pmatrix}\mu g_1 \\\mu g_2 \\ \vdots \\\mu g_n\end{pmatrix}\nonumber \\&=\lambda \gamma(\mathcal{S}_a)+\mu \gamma(\mathcal{S}_b).\nonumber
\end{align}
Similarly 
\begin{align}
\beta\circ\gamma(\lambda \mathcal{S}_a+\mu \mathcal{S}_b)&=\begin{pmatrix}0 & 1 &\cdots 1\\1 & 0 & 1 \cdots 1\\ \vdots & \vdots & \cdots \vdots\\1 & 1 & \cdots 0 \end{pmatrix}
\begin{pmatrix}\lambda f_1+\mu g_1\\\lambda f_2+\mu g_2\\ \vdots \\\lambda f_n+\mu g_n\end{pmatrix}\nonumber \\&=\begin{pmatrix}0 & 1 &\cdots 1\\1 & 0 & 1 \cdots 1\\ \vdots & \vdots & \cdots \vdots\\1 & 1 & \cdots 0 \end{pmatrix}\bigg\{ \begin{pmatrix}\lambda f_1\\\lambda f_2\\ \vdots \\\lambda f_n\end{pmatrix}+\begin{pmatrix}\mu g_1\\ \mu g_2\\ \vdots \\ \mu g_n  \end{pmatrix}\bigg\}\nonumber \\&=\lambda \begin{pmatrix}0 & 1 &\cdots 1\\1 & 0 & 1 \cdots 1\\ \vdots & \vdots & \cdots \vdots\\1 & 1 & \cdots 0 \end{pmatrix}\begin{pmatrix}f_1\\ f_2\\ \vdots \\f_n\end{pmatrix}+\mu \begin{pmatrix}0 & 1 &\cdots 1\\1 & 0 & 1 \cdots 1\\ \vdots & \vdots & \cdots \vdots\\1 & 1 & \cdots 0 \end{pmatrix}
\begin{pmatrix}g_1\\g_2\\ \vdots \\g_n  \end{pmatrix}\nonumber \\&=\lambda (\beta \circ \gamma)(\mathcal{S}_a)+\mu (\beta \circ \gamma)(\mathcal{S}_b).\nonumber
\end{align}
This proves the linearity of the expansion.
\end{proof}
\bigskip

Here, we verify the commutative property of an expansion. This reinforces the fact that there is no need to give precedence to the direction of an expansion. In essence, it gives some flexibility to the way and manner in which an expansion could be applied on tuples of polynomials.

\begin{proposition}
An expansion is commutative.
\end{proposition}

\begin{proof}
Consider $\mathcal{F}:=\{\mathcal{S}_i\}_{i=1}^{\infty}$ the collection of tuples in the ring $\mathbb{R}[x_1,x_2,\ldots,x_n]$. It suffices to show that for any $\mathcal{S}\in \mathcal{F}$, we have
\begin{align}
(\gamma^{-1}\circ\beta\circ\gamma\circ\nabla)_{[x_i]\otimes [x_j]}(\mathcal{S})=(\gamma^{-1}\circ\beta\circ\gamma\circ\nabla)_{[x_j]\otimes [x_i]}(\mathcal{S}).\nonumber 
\end{align} 
We can write
\begin{align}
(\gamma^{-1}\circ\beta\circ\gamma\circ\nabla)_{[x_i]}(\mathcal{S})&=\bigg(\bigg(\sum \limits_{\substack{t\in [1,n]\\t\neq 1}}\sum \limits_{k=t}\frac{\partial{f_k}}{\partial{x_i}}\bigg),\ldots,\bigg(\sum \limits_{\substack{t\in [1,n]\\t\neq n}}\sum \limits_{k=t}\frac{\partial{f_k}}{\partial{x_i}}\bigg)\bigg)\nonumber 
\end{align}
and make the assignment 
\begin{align}
\mathcal{S}_{g_k}&=(g_{k1},g_{k2},\ldots,g_{kn})\nonumber \\&=\bigg(\bigg(\sum \limits_{\substack{t\in [1,n]\\t\neq 1}}\sum \limits_{k=t}\frac{\partial{f_k}}{\partial{x_i}}\bigg),\ldots,\bigg(\sum \limits_{\substack{t\in [1,n]\\t\neq n}}\sum \limits_{k=t}\frac{\partial{f_k}}{\partial{x_i}}\bigg)\bigg)\nonumber 
\end{align}
for $g_{ki}\in \mathbb{R}[x_1,x_2,\ldots,x_n]$. We apply the second expansion on $\mathcal{S}_{g_k}$ and get 
\begin{align}
(\gamma^{-1}\circ\beta\circ\gamma\circ\nabla)_{[x_j]}(\mathcal{S}_{g_k})&=\bigg(\sum \limits_{\substack{s\in [1,n]\\s\neq 1}}\sum \limits_{k=s}\frac{\partial{g_{ks}}}{\partial{x_j}},\ldots,\sum \limits_{\substack{s\in [1,n]\\s\neq n}}\sum \limits_{k=s}\frac{\partial{g_{ks}}}{\partial{x_j}}\bigg)\nonumber
\end{align} 
so that by combining the two expansions in both directions, we have 
\begin{align}
(\gamma^{-1}\circ\beta\circ\gamma\circ\nabla)_{[x_i]\otimes [x_j]}(\mathcal{S})&=(\gamma^{-1}\circ\beta\circ\gamma\circ \nabla)_{[x_j]}(\mathcal{S}_{g_k})\nonumber \\&=\bigg(\bigg(\sum \limits_{\substack{s\in [1,n]\\s\neq 1}}\sum \limits_{k=s}\frac{\partial{g_{ks}}}{\partial{x_j}}\bigg),\ldots,\bigg(\sum \limits_{\substack{s\in [1,n]\\s\neq n}}\sum \limits_{k=s}\frac{\partial{g_{ks}}}{\partial{x_j}}\bigg)\bigg)\nonumber \\&=\bigg(\bigg(\sum \limits_{\substack{s\in [1,n]\\s\neq 1}}\sum \limits_{k=s}\sum \limits_{\substack{t\in [1,n]\\t\neq 1}}\sum \limits_{k=t}\frac{\partial^2{f_{k}}}{\partial{x_j} \partial{x_i}}\bigg),\nonumber \\&\ldots,\bigg(\sum \limits_{\substack{s\in [1,n]\\s\neq n}}\sum \limits_{k=s}\sum \limits_{\substack{t\in [1,n]\\t\neq n}}\sum \limits_{k=t}\frac{\partial^2{f_{k}}}{\partial{x_j} \partial{x_i}}\bigg)\bigg)\nonumber
\end{align}
by using the linearity of the operator $\frac{\partial}{\partial{x_i}}$. Applying the expansion in the opposite direction and using the linearity of the operator 
\begin{align}
\frac{\partial}{\partial{x_i}}\nonumber
\end{align} 
we have 
\begin{align}
(\gamma^{-1}\circ\beta\circ\gamma\circ\nabla)_{[x_j]\otimes [x_i]}(\mathcal{S})&=\bigg(\bigg(\sum \limits_{\substack{s\in [1,n]\\s\neq 1}}\sum \limits_{k=s}\sum \limits_{\substack{t\in [1,n]\\t\neq 1}}\sum \limits_{k=t}\frac{\partial^2{f_{k}}}{\partial{x_j} \partial{x_i}}\bigg),\nonumber \\&\ldots,\bigg(\sum \limits_{\substack{s\in [1,n]\\s\neq n}}\sum \limits_{k=s}\sum \limits_{\substack{t\in [1,n]\\t\neq n}}\sum \limits_{k=t}\frac{\partial^2{f_{k}}}{\partial{x_j} \partial{x_i}}\bigg)\bigg)\nonumber
\end{align}
by exploiting the condition 
\begin{align}
\frac{\partial^2}{\partial{x_i} \partial{x_j}}=\frac{\partial^2}{\partial{x_j} \partial{x_i}}\nonumber
\end{align} 
for each polynomial $g_i,f_i \in \mathbb{R}[x_1,x_2,\ldots,x_n]$. Comparing the result of both expansions in reverse directions, the claim follows immediately.
\end{proof}

\subsection{The area of an expansion}

In this section, we introduce the notion of the \emph{area} of a multivariate expansion. This is an extension of the area of an expansion under the single variable theory.

\begin{definition}
Let $\mathcal{F}:=\{\mathcal{S}_i\}_{i=1}^{\infty}$ be a collection of tuples of the ring $\mathbb{C}[x_1,x_2,\ldots, x_l]$. By the \emph{area} induced by the expansion
\begin{align}
(\gamma^{-1}\circ\beta\circ\gamma\circ\nabla)_{{\otimes}_{i=1}^{l}[x_{\sigma(i)}]}(\mathcal{S}),\nonumber
\end{align}
denoted by $\mathcal{A}_{\vec{a},\vec{b}}(\mathcal{S})$, from the point $\vec{a}$ to the point $\vec{b}$ we mean
\begin{align}
\mathcal{A}_{\vec{a},\vec{b}}(\mathcal{S}):&=\overrightarrow{O\Delta_{\vec{a}}^{\vec{b}}\left [(\gamma^{-1}\circ \beta \circ \gamma \circ \nabla)_{{\otimes}_{i=1}^{l}[x_{\sigma(i)}]}(\mathcal{S})\right]}\cdot \overrightarrow{O\mathcal{S}_e}\nonumber \\&=\sum \limits_{j=1}^{n}\int \limits_{a_{\sigma(l)}}^{b_{\sigma(l)}}\int \limits_{a_{\sigma(l-1)}}^{b_{\sigma(l-1)}}\cdots \int \limits_{a_{\sigma(1)}}^{b_{\sigma(1)}}g_jdx_{\sigma(1)}dx_{\sigma(2)}\cdots dx_{\sigma(l)}\nonumber 
\end{align}
where 
\begin{align}
\Delta_{\vec{a}}^{\vec{b}}\left [(\gamma^{-1}\circ \beta \circ \gamma \circ \nabla)_{{\otimes}_{i=1}^{l}[x_{\sigma(i)}]}(\mathcal{S})\right]&=\bigg(\int \limits_{a_{\sigma(l)}}^{b_{\sigma(l)}}\int \limits_{a_{\sigma(l-1)}}^{b_{\sigma(l-1)}}\cdots \int \limits_{a_{\sigma(1)}}^{b_{\sigma(1)}}g_{1}dx_{\sigma(1)}dx_{\sigma(2)}\cdots dx_{\sigma(l)},\ldots,\nonumber \\&\int \limits_{a_{\sigma(l)}}^{b_{\sigma(l)}} \int \limits_{a_{\sigma(l-1)}}^{b_{\sigma(l-1)}}\cdots \int \limits_{a_{\sigma(1)}}^{b_{\sigma(1)}}g_ndx_{\sigma(1)}dx_{\sigma(2)}\ldots dx_{\sigma(l)}\bigg)\nonumber
\end{align}
and $\mathcal{S}_e=(1,1,\ldots,1)$ is the unit tuple and $g_i\in \mathbb{C}[x_1,x_2,\ldots,x_l]$ for $1\leq i\leq n$.
\end{definition}
\bigskip

\begin{proposition}\label{arealinearity}
The area of the expansion between points in space is linear.
\end{proposition}

\begin{proof}
Let $\omega,\mu \in \mathbb{R}$ and $\mathcal{S}_1,\mathcal{S}_2$ be a tuple of polynomials whose entry belongs to the ring $\mathbb{C}[x_1,x_2,\ldots,x_n]$ with $\mathcal{S}_1=(f_1,f_2,\ldots,f_n)$ and $\mathcal{S}_2=(g_1,g_2,\ldots,g_n)$. Let $\vec{a},\vec{b}\in \mathbb{R}^{l}$. We can write
\begin{align}
    \mathcal{A}_{\vec{a},\vec{b}}(\omega \mathcal{S}_1+\mu \mathcal{S}_2)&=\overrightarrow{O\Delta_{\vec{a}}^{\vec{b}}\left [(\gamma^{-1}\circ \beta \circ \gamma \circ \nabla)_{{\otimes}_{i=1}^{l}[x_{\sigma(i)}]}(\omega \mathcal{S}_1+\mu \mathcal{S}_2)\right]}\cdot \overrightarrow{O\mathcal{S}_e}\nonumber \\&=\sum \limits_{j=1}^{n}\int \limits_{a_{\sigma(l)}}^{b_{\sigma(l)}}\int \limits_{a_{\sigma(l-1)}}^{b_{\sigma(l-1)}}\cdots \int \limits_{a_{\sigma(1)}}^{b_{\sigma(1)}}(\omega f_j+\mu g_j)dx_{\sigma(1)}dx_{\sigma(2)}\cdots dx_{\sigma(l)}\nonumber \\&=\omega \sum \limits_{j=1}^{n}\int \limits_{a_{\sigma(l)}}^{b_{\sigma(l)}}\int \limits_{a_{\sigma(l-1)}}^{b_{\sigma(l-1)}}\cdots \int \limits_{a_{\sigma(1)}}^{b_{\sigma(1)}}f_jdx_{\sigma(1)}dx_{\sigma(2)}\cdots dx_{\sigma(l)}\nonumber \\&+\mu \sum \limits_{j=1}^{n}\int \limits_{a_{\sigma(l)}}^{b_{\sigma(l)}}\int \limits_{a_{\sigma(l-1)}}^{b_{\sigma(l-1)}}\cdots \int \limits_{a_{\sigma(1)}}^{b_{\sigma(1)}}g_jdx_{\sigma(1)}dx_{\sigma(2)}\cdots dx_{\sigma(l)}\nonumber \\&=\omega \mathcal{A}_{\vec{a},\vec{b}}(\mathcal{S}_1)+\mu \mathcal{A}_{\vec{a},\vec{b}}(\mathcal{S}_2).\nonumber
\end{align}
This proves that the area of an expansion between points in space is a linear map.
\end{proof}
\bigskip

The following is an application to an inequality controlling the 2-norm of any $l$-th fold integration by the $l$-th fold integration of the 2-norm of the corresponding integrand.

\begin{theorem}\label{general inequality}
Let $g_j\in \mathbb{R}[x_1,x_2,\ldots,x_l]$ for $1\leq j\leq n$. There exists some constant $C:=C(l)>0$ such that the inequality
\begin{align}
    \int \limits_{a_{\sigma(l)}}^{b_{\sigma(l)}}\int \limits_{a_{\sigma(l-1)}}^{b_{\sigma(l-1)}}\cdots \int \limits_{a_{\sigma(1)}}^{b_{\sigma(1)}}\sqrt{\bigg(\sum \limits_{j=1}^{n}g^2_j\bigg)}dx_{\sigma(1)}dx_{\sigma(2)}\cdots dx_{\sigma(l)}\nonumber \\ \geq C(l)\sqrt{\sum \limits_{j=1}^{n}\bigg(\int \limits_{a_{\sigma(l)}}^{b_{\sigma(l)}}\int \limits_{a_{\sigma(l-1)}}^{b_{\sigma(l-1)}}\cdots \int \limits_{a_{\sigma(1)}}^{b_{\sigma(1)}}g_jdx_{\sigma(1)}dx_{\sigma(2)}\cdots dx_{\sigma(l)}\nonumber\bigg)^2}.\nonumber
\end{align}
holds for $b_{i}>a_{i}$ for each $1\leq i\leq l$.
\end{theorem}

\begin{proof}
Using the notion of the area of an expansion, we obtain an equivalent expression
\begin{align}
   \left |\mathcal{A}_{\vec{a},\vec{b}}(\mathcal{S})\right |:&=\left |\overrightarrow{O\Delta_{\vec{a}}^{\vec{b}}\left [(\gamma^{-1}\circ \beta \circ \gamma \circ \nabla)_{{\otimes}_{i=1}^{l}[x_{\sigma(i)}]}(\mathcal{S})\right]}\cdot \overrightarrow{O\mathcal{S}_e}\right |\nonumber \\&=\sqrt{n}|\cos (\alpha)|\sqrt{\sum \limits_{j=1}^{n}\bigg(\int \limits_{a_{\sigma(l)}}^{b_{\sigma(l)}}\int \limits_{a_{\sigma(l-1)}}^{b_{\sigma(l-1)}}\cdots \int \limits_{a_{\sigma(1)}}^{b_{\sigma(1)}}g_jdx_{\sigma(1)}dx_{\sigma(2)}\cdots dx_{\sigma(l)}\nonumber\bigg)^2}
\end{align}
On the other hand, we observe that by changing the order of summation, we can write using the notion of the area of an expansion the following 
\begin{align}
    \left |\mathcal{A}_{\vec{a},\vec{b}}(\mathcal{S})\right |:&=\left |\overrightarrow{O\Delta_{\vec{a}}^{\vec{b}}\left [(\gamma^{-1}\circ \beta \circ \gamma \circ \nabla)_{{\otimes}_{i=1}^{l}[x_{\sigma(i)}]}(\mathcal{S})\right]}\cdot \overrightarrow{O\mathcal{S}_e}\right |\nonumber \\&=\left |\sum \limits_{j=1}^{n}\int \limits_{a_{\sigma(l)}}^{b_{\sigma(l)}}\int \limits_{a_{\sigma(l-1)}}^{b_{\sigma(l-1)}}\cdots \int \limits_{a_{\sigma(1)}}^{b_{\sigma(1)}}g_jdx_{\sigma(1)}dx_{\sigma(2)}\cdots dx_{\sigma(l)}\right |\nonumber \\&=\left |\int \limits_{a_{\sigma(l)}}^{b_{\sigma(l)}}\int \limits_{a_{\sigma(l-1)}}^{b_{\sigma(l-1)}}\cdots \int \limits_{a_{\sigma(1)}}^{b_{\sigma(1)}}\sum \limits_{j=1}^{n}g_jdx_{\sigma(1)}dx_{\sigma(2)}\cdots dx_{\sigma(l)}\right |\nonumber \\&=\left |\int \limits_{a_{\sigma(l)}}^{b_{\sigma(l)}}\int \limits_{a_{\sigma(l-1)}}^{b_{\sigma(l-1)}}\cdots \int \limits_{a_{\sigma(1)}}^{b_{\sigma(1)}}\overrightarrow{O(\gamma^{-1}\circ \beta \circ \gamma \circ \nabla)_{{\otimes}_{i=1}^{l}[x_{\sigma(i)}]}(\mathcal{S})}\cdot \overrightarrow{O\mathcal{S}_e}dx_{\sigma(1)}dx_{\sigma(2)}\cdots dx_{\sigma(l)}\right |\nonumber \\& \leq  \int \limits_{a_{\sigma(l)}}^{b_{\sigma(l)}}\int \limits_{a_{\sigma(l-1)}}^{b_{\sigma(l-1)}}\cdots \int \limits_{a_{\sigma(1)}}^{b_{\sigma(1)}}||\overrightarrow{O(\gamma^{-1}\circ \beta \circ \gamma \circ \nabla)_{{\otimes}_{i=1}^{l}[x_{\sigma(i)}]}(\mathcal{S})}||||\overrightarrow{O\mathcal{S}_e}||dx_{\sigma(1)}dx_{\sigma(2)}\cdots dx_{\sigma(l)}\nonumber \\&=\sqrt{n}\int \limits_{a_{\sigma(l)}}^{b_{\sigma(l)}}\int \limits_{a_{\sigma(l-1)}}^{b_{\sigma(l-1)}}\cdots \int \limits_{a_{\sigma(1)}}^{b_{\sigma(1)}}\sqrt{\bigg(\sum \limits_{j=1}^{n}g^2_j\bigg)}dx_{\sigma(1)}dx_{\sigma(2)}\cdots dx_{\sigma(l)}.\nonumber
\end{align}
The inequality follows by combining this upper bound with the area of expansion.
\end{proof}
\bigskip

It is important to note that Theorem \ref{general inequality} is a generalization of the usual integral inequality for bounded functions $h$; in particular, polynomials. That is, if we take $\bigg(\sum \limits_{j=1}^{n}g^2_j\bigg)=h^2$, then 
\begin{align}
    \sqrt{\bigg(\sum \limits_{j=1}^{n}g^2_j\bigg)}=h\nonumber
\end{align}
and the result is the usual integral inequality with $h$ now on the finite supports $[a_1,b_1], [a_2,b_2],\ldots,[a_l,b_l]$. We deduce an inequality relating the minimum gap between the limits of integration to their corresponding  $l$-th fold integration of the function.

\begin{corollary}
Let $g_j\in \mathbb{R}[x_1,x_2,\ldots,x_l]$ for $1\leq j\leq n$ such that $g_j\neq 1$ with 
\begin{align}
    \sum \limits_{j=1}^{n}g^2_j\leq 1\nonumber
\end{align}
on $\cup_{i=1}^{l}[a_{i},b_i]$. There exist some constant $C:=C(l)>0$ such that the inequality
\begin{align}
    \frac{1}{C(l)}\prod \limits_{i=1}^{l}|b_{\sigma(i)
    }-a_{\sigma(i)}| \geq \sqrt{\sum \limits_{j=1}^{n}\bigg(\int \limits_{a_{\sigma(l)}}^{b_{\sigma(l)}}\int \limits_{a_{\sigma(l-1)}}^{b_{\sigma(l-1)}}\cdots \int \limits_{a_{\sigma(1)}}^{b_{\sigma(1)}}g_jdx_{\sigma(1)}dx_{\sigma(2)}\cdots dx_{\sigma(l)}\nonumber\bigg)^2}.\nonumber
\end{align}
holds.
\end{corollary}
\bigskip

There is no obvious reason why an analogous inequality should not hold for continuous functions on the intervals $\cup_{i=1}^{l}[a_i,b_i]$ by replacing the space of real multivariate polynomials $\mathbb{R}[x_1,x_2,\ldots,x_n]$ with the general space of real multivariate functions $\mathbb{F}[x_1,x_2,\ldots,x_l]$ continuous on the interval $\cup_{i=1}^{l}[a_i,b_i]$.

\subsection{The volume of an expansion}

In this section, we introduce the concept of the \emph{volume} induced by $n$ points in space.

\begin{definition}
Let $\mathcal{F}:=\{\mathcal{S}_i\}_{i=1}^{\infty}$ be a collection of tuples of polynomials with entries $f_k\in \mathbb{C}[x_1,x_2,\ldots, x_n]$. By the \emph{volume} induced by the expansion
\begin{align}
(\gamma^{-1}\circ\beta\circ\gamma\circ\nabla)_{{\otimes}_{i=1}^{l}[x_{\sigma(i)}]}(\mathcal{S}),\nonumber
\end{align}
denoted by $\mathcal{V}{\vec{a}_1,\vec{a}_2,\ldots,\vec{a}_n}(\mathcal{S})$, at the linearly independent \emph{spots} $\vec{a}_1,\vec{a}_2,\ldots,\vec{a}_n$ such that $\vec{a}_i\neq \vec{O}$ for $1\leq i\leq n$, we mean the sum
\begin{align}
\mathcal{V}_{\vec{a}_1,\vec{a}_2\ldots \vec{a}_n}(\mathcal{S}):&=\sum \limits_{\substack{s,t\in [1,n]\\s\leq t\\k\neq s,t\\k\in [1,n] }}\overrightarrow{O\Delta_{\vec{a}_s}^{\vec{a}_t}\left [(\gamma^{-1}\circ \beta \circ \gamma \circ \nabla)_{{\otimes}_{i=1}^{l}[x_{\sigma(i)}]}(\mathcal{S})\right]}\cdot \overrightarrow{O\mathcal{S}_e}\nonumber \\& \times \bigg | \bigg |\vec{a}_{s} \diamond \vec{a}_{t}\diamond \cdots \diamond \vec{a}_k \diamond \vec{a}_v\bigg |\bigg |\nonumber
\end{align}
where $\vec{a}_{s} \diamond \vec{a}_{t}\diamond \cdots \diamond \vec{a}_k \diamond \vec{a}_{v}$ is the cross product of any of $n-1$ spots including the spots $\vec{a}_s,\vec{a}_t$ and $\sigma:[1,l]\longrightarrow [1,l]$ is a permutation.
\end{definition}

\begin{proposition}
The volume of an expansion between spots is a linear operator.
\end{proposition}

\begin{proof}
This is an easy consequence of Proposition \ref{arealinearity}.
\end{proof}

\begin{theorem}\label{averageinequality}
Let $\vec{a}_{1},\vec{a}_{2}\ldots,\vec{a_n}$ be any linearly independent vectors in the space $\mathbb{R}^{l}$ and $\sigma:\{1,2,\ldots,l\}\longrightarrow \{1,2,\ldots,l\}$ be any permutation. If for each $g_k\in \mathbb{R}[x_1,x_2,\ldots,x_l]$ with $1\leq k\leq n$ and 
\begin{align}
    \int \limits_{a_{i_{\sigma(l)}}}^{a_{j_{\sigma(l)}}}\int \limits_{a_{i_{\sigma(l-1)}}}^{a_{j_{\sigma(l-1)}}}\cdots \int \limits_{a_{i_{\sigma(1)}}}^{a_{j_{\sigma(1)}}}g_kdx_{\sigma(1)}dx_{\sigma(2)}\cdots dx_{\sigma(l)}>0\nonumber
\end{align}
then there exists a constant $C>0$ such that 
\begin{align}
    \sum \limits_{\substack{i,j\in [1,n]\\a_{i_{\sigma(s)}}<a_{j_{\sigma(s)}}\\s\in [1,l]\\v\neq i,j\\v\in [1,n] }}\bigg | \bigg |\vec{a}_{i} \diamond \vec{a}_{j}\diamond \cdots \diamond \vec{a}_v\bigg |\bigg |\sum \limits_{k=1}^{n}\int \limits_{a_{i_{\sigma(l)}}}^{a_{j_{\sigma(l)}}}\int \limits_{a_{i_{\sigma(l-1)}}}^{a_{j_{\sigma(l-1)}}}\cdots \int \limits_{a_{i_{\sigma(1)}}}^{a_{j_{\sigma(1)}}}g_kdx_{\sigma(1)}dx_{\sigma(2)}\cdots dx_{\sigma(l)}\nonumber\\ \leq 2C\times \binom{n}{2}\times \sqrt{n}\times \nonumber \\ \int \limits_{a_{i_{\sigma(l)}}}^{a_{j_{\sigma(l)}}}\int \limits_{a_{i_{\sigma(l-1)}}}^{a_{j_{\sigma(l-1)}}}\cdots \int \limits_{a_{i_{\sigma(1)}}}^{a_{j_{\sigma(1)}}}\sqrt{\bigg(\sum \limits_{k=1}^{n}(\mathrm{max}(g_k))^2\bigg)}dx_{\sigma(1)}dx_{\sigma(2)}\cdots dx_{\sigma(l)}.\nonumber
\end{align}
\end{theorem}

\begin{proof}
We let 
\begin{align}
    \int \limits_{a_{i_{\sigma(l)}}}^{a_{j_{\sigma(l)}}}\int \limits_{a_{i_{\sigma(l-1)}}}^{a_{j_{\sigma(l-1)}}}\cdots \int \limits_{a_{i_{\sigma(1)}}}^{a_{j_{\sigma(1)}}}g_kdx_{\sigma(1)}dx_{\sigma(2)}\cdots dx_{\sigma(l)}\nonumber
\end{align}
be the $k^{th}$ entry of the vector $\overrightarrow{O\Delta_{\vec{a}_s}^{\vec{a}_t}\left [(\gamma^{-1}\circ \beta \circ \gamma \circ \nabla)_{{\otimes}_{i=1}^{l}[x_{\sigma(i)}]}(\mathcal{S})\right]}$ for $1\leq k \leq n$. Using the definition of the volume induced by an expansion at $n$ spots yields the expression at the left side. On the other hand, we observe that each outer sum is determined by their spots and in each of these, we maintain two distinct spots twice so that we get $2\times \binom{n}{2}$ for the number of such possible distinct sums. Under the main assumption that the maximum function taken is then absorbed by the $l$-fold integral, we deduce via the interpolation 
\begin{align}
    \left|\overrightarrow{O\Delta_{\vec{a}_i}^{\vec{a}_j}\left [(\gamma^{-1}\circ\beta\circ\gamma\circ \nabla)_{{\otimes}_{i=1}^{l}[x_{\sigma(i)}]}(\mathcal{S})\right]}\cdot \overrightarrow{O\mathcal{S}_e}\right|&=\left|\sum \limits_{k=1}^{n}\int \limits_{a_{i_{\sigma(l)}}}^{a_{j_{\sigma(l)}}}\cdots \int \limits_{a_{i_{\sigma(1)}}}^{a_{j_{\sigma(1)}}}g_kdx_{\sigma(1)}dx_{\sigma(2)}\cdots dx_{\sigma(l)}\right|\nonumber \\&=\left |\int \limits_{a_{i_{\sigma(l)}}}^{a_{j_{\sigma(l)}}}\cdots \int \limits_{a_{i_{\sigma(1)}}}^{a_{j_{\sigma(1)}}}\sum \limits_{k=1}^{n}g_k dx_{\sigma(1)}dx_{\sigma(2)}\cdots dx_{\sigma(l)}\right|\nonumber \\&=\bigg|\int \limits_{a_{i_{\sigma(l)}}}^{a_{j_{\sigma(l)}}}\cdots \int \limits_{a_{i_{\sigma(1)}}}^{a_{j_{\sigma(1)}}}\overrightarrow{O\left[(\gamma^{-1}\circ\beta\circ\gamma\circ \nabla)_{{\otimes}_{i=1}^{l}[x_{\sigma(i)}]}(\mathcal{S})\right]}\nonumber \\& \cdot \overrightarrow{O\mathcal{S}_e} dx_{\sigma(1)}dx_{\sigma(2)}\cdots dx_{\sigma(l)}\bigg |.\nonumber
\end{align}
\end{proof}
\bigskip

We deduce from this result another inequality that controls the average of any $l$-fold integral by the $l$-fold integral of the $2$-norm of the corresponding integrand.

\begin{corollary}
Let $\vec{a}_{1},\vec{a}_{2}\ldots,\vec{a_n}$ be linearly independent vectors in the space $\mathbb{R}^{l}$ and $\sigma:\{1,2,\ldots,l\}\longrightarrow \{1,2,\ldots,l\}$ be any permutation. Assume that for each $g_k\in \mathbb{R}[x_1,x_2,\ldots,x_l]$ such that $1\leq k\leq n$ with the cross products $\bigg|\bigg|\vec{a}_{i}\diamond \vec{a}_{j}\diamond \cdots \diamond \vec{a}_v\bigg|\bigg|\geq K$ for $K\in \mathbb{R}$ for all $1\leq i,j,\ldots,v\leq n$. There exist some constant $C>0$ such that if 
\begin{align}
    \int \limits_{a_{i_{\sigma(l)}}}^{a_{j_{\sigma(l)}}}\int \limits_{a_{i_{\sigma(l-1)}}}^{a_{j_{\sigma(l-1)}}}\cdots \int \limits_{a_{i_{\sigma(1)}}}^{a_{j_{\sigma(1)}}}g_kdx_{\sigma(1)}dx_{\sigma(2)}\cdots dx_{\sigma(l)}>0\nonumber
\end{align}
then 
\begin{align}
    \sum \limits_{k=1}^{n}\int \limits_{a_{i_{\sigma(l)}}}^{a_{j_{\sigma(l)}}}\int \limits_{a_{i_{\sigma(l-1)}}}^{a_{j_{\sigma(l-1)}}}\cdots \int \limits_{a_{i_{\sigma(1)}}}^{a_{j_{\sigma(1)}}}g_kdx_{\sigma(1)}dx_{\sigma(2)}\cdots dx_{\sigma(l)}&\leq C\times \frac{1}{K}\sqrt{n}\times \nonumber \\&\int \limits_{a_{i_{\sigma(l)}}}^{a_{j_{\sigma(l)}}}\cdots \int \limits_{a_{i_{\sigma(1)}}}^{a_{j_{\sigma(1)}}}\sqrt{\bigg(\sum \limits_{k=1}^{n}(\mathrm{max}(g_k))^2\bigg)}\nonumber \\&\times dx_{\sigma(1)}\cdots dx_{\sigma(l)}\nonumber
\end{align}
\end{corollary}

\bigskip

\subsection{The totient and residue of an expansion}

In this section, we introduce the notion of the \emph{residue} and the \emph{totient} of an expansion. These two notions are analogous to the notion of the rank and the degree of an expansion under the single variable theory.

\begin{definition}\label{totient}
Let $\mathcal{F}=\{\mathcal{S}_i\}_{i=1}^{\infty}$ be a collection of tuples of polynomials in the ring $\mathbb{R}[x_1,x_2,\ldots,x_n]$. We say that the expansion $(\gamma^{-1}\circ\beta\circ\gamma\circ\nabla)_{[x_i]}(\mathcal{S})$ is \emph{free} with \emph{totient} $k$, denoted by $\Phi[(\gamma^{-1}\circ\beta\circ\gamma\circ\nabla)_{[x_i]}(\mathcal{S})]$, if 
\begin{align}
(\gamma^{-1}\circ\beta\circ\gamma\circ\nabla)^k_{[x_i]}(\mathcal{S})=\mathcal{S}_0\nonumber
\end{align}
where $k>0$ is the smallest such number. We call the expansion $(\gamma^{-1}\circ\beta\circ\gamma\circ\nabla)^{k-1}_{[x_i]}(\mathcal{S})$ the \emph{residue} of the expansion. We denote the residue by $\Theta[(\gamma^{-1}\circ\beta\circ\gamma\circ\nabla)^{k-1}_{[x_i]}(\mathcal{S})]$. Similarly, by the \emph{totient} of the mixed expansion 
$$
(\gamma^{-1}\circ\beta\circ\gamma\circ\nabla)_{{\otimes}_{i=1}^{l}[x_{\sigma(i)}]}(\mathcal{S})
$$ 
we mean the smallest value of $k$ such that 
\begin{align}
(\gamma^{-1}\circ\beta\circ\gamma\circ \nabla)^k_{{\otimes}_{i=1}^{l}[x_{\sigma(i)}]}(\mathcal{S})=\mathcal{S}_0.\nonumber
\end{align}
We denote the totient of the mixed expansion by
\begin{align}
\Phi[(\gamma^{-1}\circ\beta\circ\gamma\circ\nabla)_{{\otimes}_{i=1}^{l}[x_{\sigma(i)}]}(\mathcal{S})].\nonumber
\end{align}
\end{definition}
\bigskip

\begin{proposition}
Let $\mathcal{F}=\{\mathcal{S}_i\}_{i=1}^{\infty}$ be a collection of tuples of polynomials in the ring $\mathbb{R}[x_1,x_2,\ldots,x_n]$. If the expansions 
$$
(\gamma^{-1}\circ\beta\circ\gamma\circ\nabla)_{[x_i]}(\mathcal{S}_k)
$$ 
and 
$$
(\gamma^{-1}\circ\beta\circ\gamma\circ\nabla)_{[x_i]}(\mathcal{S}_l)
$$ 
are free with totient $s$ and $t$, respectively, and such that there exists no $\alpha \in \mathbb{R}$ such that $\mathcal{S}_k=\alpha \mathcal{S}_l$, then the expansion
\begin{align}
(\gamma^{-1}\circ\beta\circ\gamma\circ\nabla)_{[x_i]}(\mathcal{S}_k+\mathcal{S}_l)\nonumber
\end{align}
is also free with totient $\mathrm{max}\{s,t\}$.
\end{proposition}

\begin{proof}
Suppose that the expansions 
$$
(\gamma^{-1}\circ \beta \circ \gamma \circ \nabla)_{[x_i]}(\mathcal{S}_k)
$$ 
and 
$$
(\gamma^{-1}\circ \beta \circ \gamma \circ \nabla)_{[x_i]}(\mathcal{S}_l)
$$ 
are free with totient $s$ and $t$, respectively. It implies that 
\begin{align}
(\gamma^{-1}\circ\beta\circ\gamma\circ\nabla)^s_{[x_i]}(\mathcal{S}_k)=\mathcal{S}_0\nonumber
\end{align}
with  
$$
(\gamma^{-1}\circ\beta\circ\gamma\circ\nabla)^{s-m}_{[x_i]}(\mathcal{S}_k)\neq \mathcal{S}_0
$$ 
for all $m\leq s$ and 
\begin{align}
(\gamma^{-1}\circ\beta\circ\gamma\circ\nabla)^{t}_{[x_i]}(\mathcal{S}_l)=\mathcal{S}_0\nonumber
\end{align}
with 
$$
(\gamma^{-1}\circ\beta\circ\gamma\circ\nabla)^{t-m}_{[x_i]}(\mathcal{S}_l)\neq \mathcal{S}_0
$$ 
for all $m\leq t$. We apply $\mathrm{max}\{s,t\}$ copies of the expansion maps to the tuple $\mathcal{S}_k+\mathcal{S}_l$. By the linearity of an expansion, we get
\begin{align}
(\gamma^{-1}\circ \beta \circ \gamma \circ \nabla)^{\mathrm{max}\{s,t\}}_{[x_i]}(\mathcal{S}_k+\mathcal{S}_l)&=(\gamma^{-1}\circ \beta \circ \gamma \circ \nabla)^{\mathrm{max}\{s,t\}}_{[x_i]}(\mathcal{S}_k)\nonumber \\&+(\gamma^{-1}\circ \beta \circ \gamma \circ \nabla)^{\mathrm{max}\{s,t\}}_{[x_i]}(\mathcal{S}_l)\nonumber \\&=\mathcal{S}_0\nonumber
\end{align}
since $s,t\leq \mathrm{max}\{s,t\}$. Observe that for any positive integers $r$ with $1\leq r\leq \mathrm{max}\{s,t\}$ the linearity of the expansion map implies
\begin{align}
(\gamma^{-1}\circ \beta \circ \gamma \circ \nabla)^{\mathrm{max}\{s,t\}-r}_{[x_i]}(\mathcal{S}_k+\mathcal{S}_l)&=(\gamma^{-1}\circ \beta \circ \gamma \circ \nabla)^{\mathrm{max}\{s,t\}-r}_{[x_i]}(\mathcal{S}_k)\nonumber \\&+(\gamma^{-1}\circ \beta \circ \gamma \circ \nabla)^{\mathrm{max}\{s,t\}-r}_{[x_i]}(\mathcal{S}_l)\nonumber \\&\neq \mathcal{S}_0\nonumber
\end{align}
since at least one of the inequality $\mathrm{max}\{s,t\}-r<s$ or $\mathrm{max}\{s,t\}-r<t$ must hold. Thus, $\mathrm{max}\{s,t\}$ is the totient of the expansion $(\gamma^{-1}\circ \beta \circ \gamma \circ \nabla)_{[x_i]}(\mathcal{S}_k+\mathcal{S}_l)$. This completes the proof of the proposition.
\end{proof}
\bigskip

Here, we expose an important relationship between the totient of the mixed expansion and expansion in specific directions. One could view this result as a sub-additivity property of the totient of an expansion.

\begin{theorem}\label{subadditive totient}
Let $\mathcal{F}=\{\mathcal{S}_i\}_{i=1}^{\infty}$ be a collection of tuples of polynomials in the ring $\mathbb{R}[x_1,x_2,\ldots,x_n]$. We have
\begin{align}
\Phi[(\gamma^{-1}\circ \beta \circ \gamma \circ \nabla)_{\otimes_{i=1}^{l}[x_{\sigma(i)}]}(\mathcal{S})]\leq \frac{1}{l}\sum \limits_{i=1}^{l}\Phi[(\gamma^{-1}\circ \beta \circ \gamma \circ \nabla)_{[x_{\sigma(i)}]}(\mathcal{S})]+\mathcal{K}\nonumber
\end{align}
where $\mathcal{K}(l)=\mathcal{K}>0$.
\end{theorem}
\bigskip

We remark that the inequality allows us to control the totient of a mixed expansion by the average of the totient of expansions in specific directions involved in the mixed expansion. We relegate the proof of this to subsequent sections, where we develop the relevant tools to establish this inequality. We will develop  stronger version of the inequality that contains enough information. 
\bigskip

\section{The Doppler effect induced by an expansion}

In this section, we introduce the notion of the \emph{Doppler} effect induced by an expansion. This phenomena is mostly induced by expansion on several other expansions in a specific direction.

\begin{definition}\label{dropler effect}
Let $\mathcal{F}=\{\mathcal{S}_i\}_{i=1}^{\infty}$ be a collection of tuples of polynomials in the ring $\mathbb{R}[x_1,x_2,\ldots,x_n]$. The expansion 
$$
(\gamma^{-1}\circ\beta\circ\gamma\circ\nabla)_{\otimes_{i=1}^{l}[x_{\sigma(i)}]}(\mathcal{S})
$$ 
is said to induce a \emph{Doppler} effect with \emph{intensity} $k$, denoted by 
$$
\mathcal{I}[(\gamma^{-1}\circ\beta\circ\gamma\circ\nabla)_{[x_j]}(\mathcal{S})]=k,
$$ 
on the expansion 
$$
(\gamma^{-1}\circ\beta\circ\gamma\circ\nabla)_{[x_j]}(\mathcal{S})
$$ 
if 
\begin{align}
(\gamma^{-1}\circ\beta\circ\gamma\circ\nabla)^{k}_{[x_j]}\circ (\gamma^{-1}\circ\beta\circ\gamma\circ\nabla)_{\otimes_{i=1}^{l}[x_{\sigma(i)}]}(\mathcal{S})\nonumber 
\end{align}
is \emph{free} with 
$$
k<\Phi[(\gamma^{-1}\circ\beta\circ\gamma\circ\nabla)_{[x_j]}(\mathcal{S})]
$$ 
and $k$ is the smallest such number. In other words, we say that the expansion admits a \emph{Doppler effect} from the \emph{source} 
$$
(\gamma^{-1}\circ\beta\circ\gamma\circ\nabla)_{\otimes_{i=1}^{l}[x_{\sigma(i)}]}(\mathcal{S})
$$ 
with intensity $k$. The energy
$$
\mathbb{E}[(\gamma^{-1}\circ\beta\circ\gamma\circ\nabla)_{[x_j]}](\mathcal{S})
$$ 
saved by the expansion under the Doppler effect is
\begin{align}
\mathbb{E}[(\gamma^{-1}\circ\beta\circ\gamma\circ\nabla)_{[x_j]}](\mathcal{S})=\Phi[(\gamma^{-1}\circ\beta\circ\gamma\circ \nabla)_{[x_j]}(\mathcal{S})]-\mathcal{I}[(\gamma^{-1}\circ \beta \circ \gamma \circ \nabla)_{[x_j]}(\mathcal{S})].\nonumber
\end{align}
We call this equation the energy-Doppler effect-intensity equation.
\end{definition}
\bigskip

\begin{proposition}\label{dropler effect additive property}
Let $\mathcal{F}=\{\mathcal{S}_i\}_{i=1}^{\infty}$ be a collection of tuples of polynomials in the ring $\mathbb{R}[x_1,x_2,\ldots,x_n]$. If the expansions $(\gamma^{-1}\circ\beta\circ\gamma\circ\nabla)_{[x_s]}(\mathcal{S})$ and $(\gamma^{-1}\circ\beta\circ\gamma\circ \nabla)_{[x_t]}(\mathcal{S})$ each admits a Doppler effect from the same source with intensities $k_1$ and $k_2$, respectively, then the expansion 
\begin{align}
\bigg [(\gamma^{-1}\circ\beta\circ\gamma\circ\nabla)_{[x_{s}]}+(\gamma^{-1}\circ\beta\circ\gamma\circ\nabla)_{[x_{t}]}\bigg ](\mathcal{S})\nonumber 
\end{align}
also admits a Doppler effect from the same source with intensity $\mathrm{max}\{k_1,k_2\}$.
\end{proposition}

\begin{proof}
Suppose that the expansions  
$$
(\gamma^{-1}\circ\beta\circ\gamma\circ\nabla)_{[x_s]}(\mathcal{S})
$$ 
and 
$$
(\gamma^{-1}\circ\beta\circ\gamma\circ \nabla)_{[x_t]}(\mathcal{S})
$$ 
each admits a Doppler effect from the same source with intensities $k_1$ and $k_2$, respectively. Let $(\gamma^{-1}\circ\beta\circ\gamma\circ\nabla)_{\otimes_{i=1}^{l}[x_{\sigma(i)}]}(\mathcal{S})$ be their source. By definition \ref{dropler effect}, we have
\begin{align}
(\gamma^{-1}\circ\beta\circ\gamma\circ\nabla)^{k_1}_{[x_s]}\circ (\gamma^{-1}\circ\beta\circ\gamma\circ\nabla)_{\otimes_{i=1}^{l}[x_{\sigma(i)}]}(\mathcal{S})=\mathcal{S}_0\label{intensity 1}
\end{align}
and 
\begin{align}
(\gamma^{-1}\circ\beta\circ\gamma\circ\nabla)^{k_2}_{[x_t]}\circ (\gamma^{-1}\circ\beta\circ\gamma\circ\nabla)_{\otimes_{i=1}^{l}[x_{\sigma(i)}]}(\mathcal{S})=\mathcal{S}_0.\label{intensity 2}
\end{align}
Consider the expansion map $\bigg [(\gamma^{-1}\circ\beta\circ\gamma\circ\nabla)_{[x_{s}]}+(\gamma^{-1}\circ\beta\circ\gamma\circ\nabla)_{[x_{t}]}\bigg ]$ and apply $\mathrm{max}\{k_1,k_2\}$ copies to the source $(\gamma^{-1}\circ\beta\circ\gamma\circ\nabla)_{\otimes_{i=1}^{l}[x_{\sigma(i)}]}(\mathcal{S})$. Using \eqref{intensity 1} and \eqref{intensity 2}, we deduce
\begin{align}
\bigg [(\gamma^{-1}\circ\beta\circ\gamma\circ\nabla)_{[x_{s}]}+(\gamma^{-1}\circ\beta\circ\gamma\circ\nabla)_{[x_{t}]}\bigg ]^{\mathrm{max}\{k_1,k_2\}}\circ (\gamma^{-1}\circ\beta\circ\gamma\circ\nabla)_{\otimes_{i=1}^{l}[x_{\sigma(i)}]}(\mathcal{S})=\mathcal{S}_0.\nonumber
\end{align}
We get
\begin{align}
\bigg [(\gamma^{-1}\circ \beta \circ \gamma \circ \nabla)_{[x_{s}]}+(\gamma^{-1}\circ \beta \circ \gamma \circ \nabla)_{[x_{t}]}\bigg ]^{\mathrm{max}\{k_1,k_2\}-r} \circ (\gamma^{-1} \circ \beta \circ \gamma \circ \nabla)_{\otimes_{i=1}^{l}[x_{\sigma(i)}]}(\mathcal{S})\neq \mathcal{S}_0
\end{align}
for any $r\geq 1$, using the linearity of a multivariate expansion and the fact that at least one of the inequality must hold
\begin{align}
\mathrm{max}\{k_1,k_2\}-r\leq k_1 \quad \mathrm{or} \quad \mathrm{max}\{k_1,k_2\}-r\leq k_2.\nonumber
\end{align}
Thus, $\mathrm{max}\{k_1,k_2\}$ is the intensity of the Doppler effect induced on the concatenations of the expansions under the same source.
\end{proof}
\bigskip

One could ask whether an analog of this result exists for expansions with concatenated directions. While a general answer to this question may seem very baffling, we can obtain a variant under additional requirements that ensure the expansion in one direction does not wear off and interfere with the direction of the other. We make this assertion precise in the following proposition.

\begin{proposition}\label{mixed expansion dropler effect property}
Let $\mathcal{F}=\{\mathcal{S}_i\}_{i=1}^{\infty}$ be a collection of tuples of polynomials in the ring $\mathbb{R}[x_1,x_2,\ldots,x_n]$. Let the expansions 
$$
(\gamma^{-1}\circ\beta\circ\gamma\circ\nabla)_{[x_s]}(\mathcal{S})
$$ 
and 
$$
(\gamma^{-1}\circ\beta\circ\gamma\circ \nabla)_{[x_t]}(\mathcal{S})
$$ 
each admits a Doppler effect with intensities $k_1$ and $k_2$, respectively, from the source 
$$
(\gamma^{-1}\circ\beta\circ\gamma\circ\nabla)_{\otimes_{i=1}^{l}[x_{\sigma(i)}]}(\mathcal{S}).
$$ 
If the expansions 
$$
(\gamma^{-1}\circ\beta\circ\gamma\circ\nabla)_{[x_s]}(\mathcal{S})
$$ 
and 
$$
(\gamma^{-1}\circ\beta\circ\gamma\circ\nabla)_{[x_t]}(\mathcal{S})
$$ 
admits no Doppler effect from the sources
\begin{align}
(\gamma^{-1}\circ\beta\circ\gamma\circ\nabla)_{[x_t]}\circ (\gamma^{-1}\circ\beta\circ\gamma\circ\nabla)_{\otimes_{i=1}^{l}[x_{\sigma(i)}]}(\mathcal{S})\nonumber 
\end{align}
and 
\begin{align}
(\gamma^{-1}\circ\beta\circ\gamma\circ\nabla)_{[x_s]}\circ (\gamma^{-1}\circ\beta\circ\gamma\circ\nabla)_{\otimes_{i=1}^{l}[x_{\sigma(i)}]}(\mathcal{S}),\nonumber 
\end{align} 
respectively, then the mixed expansion 
\begin{align}
(\gamma^{-1}\circ\beta\circ\gamma\circ\nabla)_{[x_s]\otimes [x_t]}(\mathcal{S})\nonumber
\end{align}
also admits a Doppler effect from the same source with intensity $\mathrm{min}\{k_1,k_2\}$.
\end{proposition}
\bigskip

\begin{proof}
Suppose that the expansions 
$$
(\gamma^{-1}\circ\beta\circ\gamma\circ\nabla)_{[x_s]}(\mathcal{S})
$$ 
and 
$$
(\gamma^{-1}\circ\beta\circ\gamma\circ\nabla)_{[x_t]}(\mathcal{S})
$$ 
each admits a Doppler effect from the same source with intensities $k_1$ and $k_2$, respectively. Let 
$$
(\gamma^{-1}\circ\beta\circ\gamma\circ\nabla)_{\otimes_{i=1}^{l}[x_{\sigma(i)}]}(\mathcal{S})
$$ 
be their source. By definition \ref{dropler effect}, we have
\begin{align}
(\gamma^{-1}\circ\beta\circ\gamma\circ\nabla)^{k_1}_{[x_s]}\circ (\gamma^{-1}\circ\beta\circ\gamma\circ\nabla)_{\otimes_{i=1}^{l}[x_{\sigma(i)}]}(\mathcal{S})=\mathcal{S}_0\label{intensity mixed 1}
\end{align}
and 
\begin{align}
(\gamma^{-1}\circ\beta\circ\gamma\circ\nabla)^{k_2}_{[x_t]}\circ (\gamma^{-1}\circ\beta\circ\gamma\circ\nabla)_{\otimes_{i=1}^{l}[x_{\sigma(i)}]}(\mathcal{S})=\mathcal{S}_0.\label{intensity mixed 2}
\end{align}
We apply $\mathrm{min}\{k_1,k_2\}$ copies of the mixed expansion operator $(\gamma^{-1}\circ\beta\circ\gamma\circ \nabla)_{[x_s]\otimes [x_t]}$ to the source $(\gamma^{-1}\circ\beta\circ\gamma\circ\nabla)_{\otimes_{i=1}^{l}[x_{\sigma(i)}]}(\mathcal{S})$
and get 
\begin{align}
(\gamma^{-1}\circ\beta\circ\gamma\circ\nabla)^{\mathrm{min}\{k_1,k_2\}}_{[x_s]\otimes [x_t]}\circ(\gamma^{-1}\circ\beta\circ\gamma\circ\nabla)_{\otimes_{i=1}^{l}[x_{\sigma(i)}]}(\mathcal{S})&=\mathcal{S}_0\nonumber
\end{align}
using the commutative property of the expansion operator and \eqref{intensity mixed 1} and \eqref{intensity mixed 2}. Again by commutativity, we get 
\begin{align}
(\gamma^{-1}\circ\beta\circ\gamma\circ\nabla)^{\mathrm{min}\{k_1,k_2\}-r}_{[x_s]\otimes [x_t]}\circ(\gamma^{-1}\circ\beta\circ\gamma\circ\nabla)_{\otimes_{i=1}^{l}[x_{\sigma(i)}]}(\mathcal{S})&\neq \mathcal{S}_0\nonumber
\end{align}
for any $\mathrm{min}\{k_1,k_2\}\geq r\geq 1$, since $\mathrm{min}\{k_1,k_2\}-r<k_1$ with $\mathrm{min}\{k_1,k_2\}-r<k_2$ and $k_1,k_2$ are the intensities of the Doppler effects and the expansions 
$$
(\gamma^{-1}\circ\beta\circ\gamma\circ\nabla)_{[x_s]}(\mathcal{S})
$$ 
and 
$$
(\gamma^{-1}\circ\beta\circ\gamma\circ\nabla)_{[x_t]}(\mathcal{S})
$$ 
admit no Doppler effect from the sources
\begin{align}
(\gamma^{-1}\circ\beta\circ\gamma\circ\nabla)_{[x_t]}\circ (\gamma^{-1}\circ\beta\circ\gamma\circ\nabla)_{\otimes_{i=1}^{l}[x_{\sigma(i)}]}(\mathcal{S})\nonumber 
\end{align}
and 
\begin{align}
(\gamma^{-1}\circ\beta\circ\gamma\circ\nabla)_{[x_s]}\circ (\gamma^{-1}\circ\beta\circ\gamma\circ\nabla)_{\otimes_{i=1}^{l}[x_{\sigma(i)}]}(\mathcal{S})\nonumber 
\end{align} 
respectively. This proves that 
$$
\mathrm{min}\{k_1,k_2\}=\mathcal{I}[(\gamma^{-1}\circ \beta \circ \gamma \circ \nabla)_{[x_s]\otimes [x_t]}(\mathcal{S})]
$$
is the intensity of the Doppler effect induced on the mixed expansion.
\end{proof}
\bigskip

\begin{proposition}\label{dropler effect totient}
Let $\mathcal{F}=\{\mathcal{S}_i\}_{i=1}^{\infty}$ be a collection of tuples of polynomials in the ring $\mathbb{R}[x_1,x_2,\ldots,x_n]$. Let 
$$
(\gamma^{-1}\circ\beta\circ\gamma\circ\nabla)_{[x_s]}(\mathcal{S})
$$ 
and 
$$
(\gamma^{-1}\circ\beta\circ\gamma\circ\nabla)_{[x_t]}(\mathcal{S})
$$ 
be expansions with totient $k_1$ and $k_2$, respectively. If the expansions 
$$
(\gamma^{-1}\circ\beta\circ\gamma\circ \nabla)_{[x_s]}(\mathcal{S})
$$ 
and 
$$
(\gamma^{-1}\circ\beta\circ\gamma\circ\nabla)_{[x_t]}(\mathcal{S})
$$ 
admits no Doppler effect from the source
\begin{align}
(\gamma^{-1}\circ\beta\circ\gamma\circ\nabla)^{u}_{[x_{t}]}(\mathcal{S})\nonumber 
\end{align}
for $u<k_2$ and 
\begin{align}
(\gamma^{-1}\circ\beta\circ\gamma\circ\nabla)^{v}_{[x_{s}]}(\mathcal{S})\nonumber 
\end{align}
for $v<k_1$, respectively, then 
\begin{align}
\Phi[(\gamma^{-1}\circ\beta\circ\gamma\circ \nabla)_{[x_{t}]\otimes [x_{s}]}(\mathcal{S})]=\mathrm{min}\{k_1,k_2\}.\nonumber
\end{align}
\end{proposition}
\bigskip

\begin{proof}
Suppose that 
$$
(\gamma^{-1}\circ\beta\circ\gamma\circ\nabla)_{[x_s]}(\mathcal{S})
$$ 
and 
$$
(\gamma^{-1}\circ\beta\circ\gamma\circ\nabla)_{[x_t]}(\mathcal{S})
$$ 
are expansions with totient $k_1$ and $k_2$, respectively. It implies that 
\begin{align}
(\gamma^{-1}\circ\beta\circ\gamma\circ\nabla)^{k_1}_{[x_s]}(\mathcal{S})=\mathcal{S}_0\label{mixed totient 1}
\end{align}
and
\begin{align}
(\gamma^{-1}\circ\beta\circ\gamma\circ\nabla)^{k_2}_{[x_t]}(\mathcal{S})=\mathcal{S}_0\label{mixed totient 2}
\end{align}
where $k_1,k_2$ are the smallest such number. Using commutativity of the expansion operator with \eqref{mixed totient 1} and \eqref{mixed totient 2}, we get
\begin{align}
(\gamma^{-1}\circ\beta\circ\gamma\circ\nabla)^{\mathrm{min}\{k_1,k_2\}}_{[x_{t}]\otimes [x_{s}]}(\mathcal{S})&=(\gamma^{-1}\circ\beta\circ\gamma\circ\nabla)^{\mathrm{min}\{k_1,k_2\}}_{[x_{s}]}\circ (\gamma^{-1}\circ\beta\circ\gamma\circ\nabla)^{\mathrm{min}\{k_1,k_2\}}_{[x_{t}]}(\mathcal{S})\nonumber \\&=\mathcal{S}_0.\nonumber
\end{align}
Under the assumption that the expansions $(\gamma^{-1}\circ\beta\circ\gamma\circ\nabla)_{[x_s]}(\mathcal{S})$ and $(\gamma^{-1}\circ\beta\circ\gamma\circ\nabla)_{[x_t]}(\mathcal{S})$ admits no Doppler effect from the source
\begin{align}
(\gamma^{-1}\circ\beta\circ\gamma\circ\nabla)^{v}_{[x_{t}]}(\mathcal{S})\nonumber 
\end{align}
for $v <k_2$ and 
\begin{align}
(\gamma^{-1}\circ\beta\circ\gamma\circ\nabla)^{u}_{[x_{s}]}(\mathcal{S})\nonumber 
\end{align}
for $u<k_1$, respectively, we find that 
$$
\Phi[(\gamma^{-1}\circ\beta\circ\gamma\circ\nabla)_{[x_s]}(\mathcal{S})]
$$ 
and 
$$
\Phi[(\gamma^{-1}\circ\beta\circ\gamma\circ\nabla)_{[x_t]}(\mathcal{S})]
$$ 
are the smallest numbers, respectively, such that 
\begin{align}
(\gamma^{-1}\circ\beta\circ\gamma\circ\nabla)^{\Phi[(\gamma^{-1}\circ\beta\circ\gamma\circ\nabla)_{[x_s]}(\mathcal{S})]}_{[x_s]}\circ (\gamma^{-1}\circ\beta\circ\gamma\circ\nabla)^{v}_{[x_{t}]}(\mathcal{S})=\mathcal{S}_0\nonumber
\end{align}
and
\begin{align}
(\gamma^{-1}\circ\beta\circ\gamma\circ\nabla)^{\Phi[(\gamma^{-1}\circ\beta\circ\gamma\circ\nabla)_{[x_t]}(\mathcal{S})]}_{[x_t]}\circ (\gamma^{-1}\circ \beta \circ \gamma \circ \nabla)^{u}_{[x_{s}]}(\mathcal{S})=\mathcal{S}_0\nonumber
\end{align}
so that for any $\mathrm{min}\{k_1,k_2\}\geq c\geq 1$ then 
\begin{align}
(\gamma^{-1}\circ\beta\circ\gamma\circ\nabla)^{\mathrm{min}\{k_1,k_2\}-c}_{[x_{t}]\otimes [x_{s}]}(\mathcal{S})&\neq \mathcal{S}_0\nonumber
\end{align}
using the linearity of an expansion operator. This proves that 
$$
\mathrm{min}\{k_1,k_2\}=\Phi[(\gamma^{-1}\circ\beta\circ\gamma\circ\nabla)_{[x_{t}]\otimes [x_{s}]}(\mathcal{S})]
$$ 
is the totient of the mixed expansion.
\end{proof}
\bigskip

\section{Destabilization of an expansion}

In this section, we introduce the notion of \emph{destabilization} induced by an expansion.

\begin{definition}
Let $\mathcal{F}=\{\mathcal{S}_i\}_{i=1}^{\infty}$ be a collection of tuples of polynomials in the ring $\mathbb{R}[x_1,x_2,\ldots,x_n]$. We say that the expansion $(\gamma^{-1}\circ\beta\circ\gamma\circ\nabla)_{[x_{i}]}(\mathcal{S})$ is said to undergo natural \emph{destabilization} if $(\gamma^{-1}\circ\beta\circ\gamma\circ\nabla)^{0}_{[x_{i}](0)}(\mathcal{S})\neq \mathcal{S}_0$. We say that it undergoes \emph{destabilization} at the stage $k\geq 1$ if 
$$
(\gamma^{-1}\circ\beta\circ\gamma\circ\nabla)^{j}_{[x_{i}](0)}(\mathcal{S})=\mathcal{S}_0
$$ 
for all $1\leq j\leq k-1$ and 
$$
(\gamma^{-1}\circ\beta\circ\gamma\circ\nabla)^{k}_{[x_{i}](0)}(\mathcal{S})\neq \mathcal{S}_0.
$$ 
In other words, we say that the expansion $(\gamma^{-1}\circ\beta\circ\gamma\circ\nabla)_{[x_{i}]}(\mathcal{S})$ admits a destabilization at the stage $k\geq 1$. We say that it is \emph{strongly} destabilized if the vector
\begin{align}
\overrightarrow{O(\gamma^{-1}\circ\beta\circ\gamma\circ\nabla)^{k}_{[x_{i}](0)}(\mathcal{S})}\nonumber
\end{align}
has no zero entry.
\end{definition}
\bigskip

Here, we show that destabilization stage must be reached in an expansion. The following result confines this stage to a certain range.

\begin{proposition}
Let $\mathcal{F}=\{\mathcal{S}_i\}_{i=1}^{\infty}$ be a collection of tuples of polynomials in the ring $\mathbb{R}[x_1,x_2,\ldots,x_n]$. The stage of destabilization $k$ of the expansion $(\gamma^{-1}\circ\beta\circ\gamma\circ \nabla)_{[x_{i}]}(\mathcal{S})$ satisfies the inequality
\begin{align}
0\leq k<\Phi[(\gamma^{-1}\circ\beta\circ\gamma\circ\nabla)_{[x_{i}]}(\mathcal{S})].\nonumber
\end{align}
\end{proposition}

\begin{proof}
If the expansion $(\gamma^{-1}\circ\beta\circ\gamma\circ\nabla)_{[x_{i}]}(\mathcal{S})$ admits a natural destabilization, then the stage $k=0$. Thus, we may assume that the expansion $(\gamma^{-1}\circ\beta\circ\gamma\circ\nabla)_{[x_{i}]}(\mathcal{S})$ does not admit a natural destabilization. We suppose to the contrary that the stage of destabilization of some expansion $(\gamma^{-1}\circ\beta\circ\gamma\circ\nabla)_{[x_{i}]}(\mathcal{S}_m)$ satisfies 
$$
k\geq \Phi[(\gamma^{-1}\circ\beta\circ\gamma\circ \nabla)_{[x_{i}]}(\mathcal{S}_m)]
$$ 
so that 
\begin{align}
(\gamma^{-1}\circ\beta\circ\gamma\circ\nabla)^{\Phi[(\gamma^{-1}\circ\beta\circ\gamma\circ\nabla)_{[x_{i}]}(\mathcal{S}_m)]-1}_{[x_{i}](0)}(\mathcal{S}_m)=\mathcal{S}_0. \nonumber
\end{align}
This is a contradiction, since the expansion 
\begin{align}
(\gamma^{-1}\circ\beta\circ\gamma\circ\nabla)^{\Phi[(\gamma^{-1}\circ\beta\circ\gamma\circ\nabla)_{[x_{i}]}(\mathcal{S}_m)]-1}_{[x_{i}]}(\mathcal{S}_m)\nonumber
\end{align}
is the residue of the expansion $(\gamma^{-1}\circ\beta\circ\gamma\circ\nabla)_{[x_{i}]}(\mathcal{S}_m)$ and thus has no direction of form $[x_i]$. 
\end{proof}
\bigskip

\begin{theorem}\label{stage strong destabilize}
Let $\mathcal{F}=\{\mathcal{S}_i\}_{i=1}^{\infty}$ be a collection of tuples of polynomials in the ring $\mathbb{R}[x_1,x_2,\ldots,x_n]$. For all directions $[x_j]$ with $1\leq j\leq n$ every expansion $(\gamma^{-1}\circ\beta\circ\gamma\circ\nabla)_{[x_{j}]}(\mathcal{S})$ is strongly destabilized at the stage 
$$
\Phi[(\gamma^{-1}\circ\beta\circ\gamma\circ\nabla)_{[x_{j}]}(\mathcal{S})]-1.
$$
\end{theorem}

\begin{proof}
Let $(\gamma^{-1}\circ\beta\circ\gamma\circ\nabla)_{[x_{j}]}(\mathcal{S})$ be any expansion in an arbitrary direction $[x_j]$. By definition \ref{totient} the expansion 
\begin{align}
(\gamma^{-1}\circ\beta\circ\gamma\circ\nabla)^{\Phi[(\gamma^{-1}\circ\beta\circ\gamma\circ\nabla)_{[x_{j}]}(\mathcal{S})]-1}_{[x_{j}]}(\mathcal{S})\nonumber
\end{align}
is the residue of the expansion $(\gamma^{-1}\circ\beta\circ\gamma\circ\nabla)_{[x_{j}]}(\mathcal{S})$. We suppose to the contrary that the vector 
\begin{align}
\overrightarrow{O(\gamma^{-1}\circ\beta\circ\gamma\circ\nabla)^{\Phi[(\gamma^{-1}\circ\beta\circ\gamma\circ\nabla)_{[x_{j}]}(\mathcal{S})]-1}_{[x_{j}](0)}(\mathcal{S})}\nonumber
\end{align}
has at least a zero entry. It follows that the expansion $(\gamma^{-1}\circ\beta\circ\gamma\circ\nabla)^{\Phi[(\gamma^{-1}\circ\beta\circ\gamma\circ\nabla)_{[x_{j}]}(\mathcal{S})]-1}_{[x_{j}]}(\mathcal{S})$ contains the direction $[x_j]$ and hence
\begin{align}
(\gamma^{-1}\circ\beta\circ\gamma\circ\nabla)^{\Phi[(\gamma^{-1}\circ\beta\circ\gamma\circ\nabla)_{[x_{j}]}(\mathcal{S})]}_{[x_{j}]}(\mathcal{S})\neq \mathcal{S}_0\nonumber
\end{align}
which contradicts the fact that 
$$
\Phi[(\gamma^{-1}\circ\beta\circ\gamma\circ\nabla)_{[x_{j}]}(\mathcal{S})]
$$ 
is the totient of the expansion $(\gamma^{-1}\circ\beta\circ\gamma\circ\nabla)_{[x_{j}]}(\mathcal{S})$. This completes the proof.
\end{proof}
\bigskip

Here, we relate the notion of the Doppler effect induced by a mixed expansion on expansions in a specific direction to the notion of destabilization. We show that these two notions are somewhat related.
\bigskip

\section{Diagonalization and sub-expansion of an expansion}

In this section, we introduce the notion of \emph{diagonalization} of an expansion and sub-expansion of an expansion.

\begin{definition}\label{diagonalization}
Let $\mathcal{F}=\{\mathcal{S}_i\}_{i=1}^{\infty}$ be a collection of tuples of polynomials in the ring $\mathbb{R}[x_1,x_2,\ldots,x_n]$. We say that the mixed expansion $(\gamma^{-1}\circ\beta\circ\gamma\circ\nabla)_{\otimes_{i=1}^{l}[x_{\sigma(i)}]}(\mathcal{S})$ is \emph{diagonalizable} in the direction $[x_j]$ ($1\leq j\leq n$) at the spot $\mathcal{S}_r\in \mathcal{F}$ of order $k$ with $\mathcal{S}-\mathcal{S}_r$ not a tuple of $\mathbb{R}$ if 
\begin{align}
(\gamma^{-1}\circ\beta\circ\gamma\circ\nabla)_{\otimes_{i=1}^{l}[x_{\sigma(i)}]}(\mathcal{S})=(\gamma^{-1}\circ\beta\circ\gamma\circ \nabla)^{k}_{[x_{j}]}(\mathcal{S}_r).\nonumber
\end{align}
We call the expansion 
$$
(\gamma^{-1}\circ\beta\circ\gamma\circ\nabla)_{[x_{j}]}(\mathcal{S}_r)
$$ 
the \emph{diagonal} of the mixed expansion $(\gamma^{-1}\circ\beta\circ\gamma\circ\nabla)_{\otimes_{i=1}^{l}[x_{\sigma(i)}]}(\mathcal{S})$ of \emph{order} $k\geq 1$. We denote by $\mathcal{O}[(\gamma^{-1}\circ\beta\circ\gamma\circ\nabla)_{[x_{j}]}(\mathcal{S}_r)]$ the order of the diagonal.
\end{definition}
\bigskip

\begin{proposition}\label{diagonalsum}
Let $\mathcal{F}:=\{\mathcal{S}_i\}_{i=1}^{\infty}$ be a collection of tuples of polynomials in the ring $\mathbb{R}[x_1,x_2,\ldots,x_n]$. Assume that the expansions $(\gamma^{-1}\circ\beta\circ\gamma\circ\nabla)_{\otimes_{i=1}^{l}[x_{\sigma(i)}]}(\mathcal{S}_t)$ and $(\gamma^{-1}\circ\beta\circ\gamma\circ\nabla)_{\otimes_{i=1}^{l}[x_{\sigma(i)}]}(\mathcal{S}_r)$ are both diagonalizable in the fixed direction $[x_i]$ at the spot $\mathcal{S}_a$ with order $u$ and $\mathcal{S}_k$ with order $v$, respectively. If $u>v$ (resp. $v>u$) then
\begin{align}
(\gamma^{-1}\circ\beta\circ\gamma\circ\nabla)_{\otimes_{i=1}^{l}[x_{\sigma(i)}]}(\mathcal{S}_t+\mathcal{S}_r)\nonumber 
\end{align}
is also diagonalizable at the spot $(\gamma^{-1}\circ\beta\circ\gamma\circ\nabla)^{u-v}_{[x_i]}(\mathcal{S}_a)+\mathcal{S}_k$ with order $v$, respectively
\begin{align}
(\gamma^{-1}\circ\beta\circ\gamma\circ\nabla)^{v-u}_{[x_i]}(\mathcal{S}_k)+\mathcal{S}_a\nonumber
\end{align}
with order $u$.
\end{proposition}

\begin{proof}
Suppose that the expansions $(\gamma^{-1}\circ\beta\circ\gamma\circ\nabla)_{\otimes_{i=1}^{l}[x_{\sigma(i)}]}(\mathcal{S}_t)$ and $(\gamma^{-1}\circ\beta\circ\gamma\circ\nabla)_{\otimes_{i=1}^{l}[x_{\sigma(i)}]}(\mathcal{S}_r)$ are both diagonalizable in the fixed direction $[x_i]$ at the spots $\mathcal{S}_a$ of order $u$ and $\mathcal{S}_k$ of order $v$, respectively. By definition \ref{diagonalization}, we get 
\begin{align}
(\gamma^{-1}\circ\beta\circ\gamma\circ\nabla)_{\otimes_{i=1}^{l}[x_{\sigma(i)}]}(\mathcal{S}_t)=(\gamma^{-1}\circ\beta\circ\gamma\circ\nabla)^{u}_{[x_i]}(\mathcal{S}_a)\nonumber
\end{align}
and 
\begin{align}
(\gamma^{-1}\circ\beta\circ\gamma\circ\nabla)_{\otimes_{i=1}^{l}[x_{\sigma(i)}]}(\mathcal{S}_r)=(\gamma^{-1}\circ\beta\circ\gamma\circ\nabla)^{v}_{[x_i]}(\mathcal{S}_k).\nonumber
\end{align}
Concatenating the two mixed expansion and using the linearity of an expansion with $u>v$, we have
\begin{align}
(\gamma^{-1}\circ\beta\circ\gamma\circ\nabla)_{\otimes_{i=1}^{l}[x_{\sigma(i)}]}(\mathcal{S}_t+\mathcal{S}_r)&=(\gamma^{-1}\circ\beta\circ\gamma\circ\nabla)_{\otimes_{i=1}^{l}[x_{\sigma(i)}]}(\mathcal{S}_t)+(\gamma^{-1}\circ\beta\circ \nonumber \\& \gamma\circ\nabla)_{\otimes_{i=1}^{l}[x_{\sigma(i)}]}(\mathcal{S}_r)\nonumber \\&=(\gamma^{-1}\circ\beta\circ\gamma\circ\nabla)^{u}_{[x_i]}(\mathcal{S}_a)+(\gamma^{-1}\circ\beta\circ\gamma\circ\nabla)^{v}_{[x_i]}(\mathcal{S}_k).\nonumber
\end{align}
Under the assumption $u>v$, we deduce
\begin{align}
(\gamma^{-1}\circ\beta\circ\gamma\circ\nabla)_{\otimes_{i=1}^{l}[x_{\sigma(i)}]}(\mathcal{S}_t+\mathcal{S}_r)&=(\gamma^{-1}\circ\beta\circ\gamma\circ\nabla)^{v}_{[x_i]}\bigg((\gamma^{-1}\circ\beta\circ\gamma\circ\nabla)^{u-v}_{[x_i]}(\mathcal{S}_a)+\mathcal{S}_k\bigg).\nonumber
\end{align}
We establish the claim by choosing the spot
\begin{align}
\mathcal{S}_f=(\gamma^{-1}\circ\beta\circ\gamma\circ\nabla)^{u-v}_{[x_i]}(\mathcal{S}_a)+\mathcal{S}_k.\nonumber
\end{align}
\end{proof}
\bigskip

The notion of the sub-expansion of an expansion in the \emph{single variable theory} extends to this framework.

\begin{definition}\label{subexpansion}
Let $\mathcal{F}=\{\mathcal{S}_i\}_{i=1}^{\infty}$ be a collection of tuples of polynomials in the ring $\mathbb{R}[x_1,x_2,\ldots,x_n]$. We say that the expansion $(\gamma^{-1}\circ\beta\circ\gamma\circ\nabla)^k_{[x_{j}]}(\mathcal{S}_z)$ is a \emph{sub-expansion} of the expansion $(\gamma^{-1}\circ\beta\circ\gamma\circ\nabla)^l_{[x_{j}]}(\mathcal{S}_t)$, denoted $(\gamma^{-1}\circ\beta\circ\gamma\circ\nabla)^k_{[x_{j}]}(\mathcal{S}_z)\leq (\gamma^{-1}\circ\beta\circ\gamma\circ\nabla)^l_{[x_{j}]}(\mathcal{S}_t)$ if there exist some $0\leq m$ such that 
\begin{align}
(\gamma^{-1}\circ\beta\circ\gamma\circ\nabla)^k_{[x_{j}]}(\mathcal{S}_z)=(\gamma^{-1}\circ\beta\circ\gamma\circ\nabla)^{k+m}_{[x_{j}]}(\mathcal{S}_t).\nonumber
\end{align}
We say that the sub-expansion is proper if $m+k=l$. We denote this proper sub-expansion by 
$$
(\gamma^{-1}\circ\beta\circ\gamma\circ\nabla)^k_{[x_{j}]}(\mathcal{S}_z)<(\gamma^{-1}\circ\beta\circ\gamma\circ\nabla)^l_{[x_{j}]}(\mathcal{S}_t).
$$ 
On the other hand, we say that the sub-expansion is \emph{ancient} if $m+k>l$.
\end{definition}
\bigskip

Here, we relate the notion of the \emph{sub-expansion} of an expansion to the notion of \emph{diagonalization} of a mixed expansion.

\begin{proposition}\label{diagsubexpansion1}
Let $\mathcal{F}=\{\mathcal{S}_i\}_{i=1}^{\infty}$ be a collection of tuples of polynomials in the ring $\mathbb{R}[x_1,x_2,\ldots,x_n]$. If the mixed expansion $(\gamma^{-1}\circ\beta\circ\gamma\circ\nabla)_{\otimes_{i=1}^{l}[x_{\sigma(i)}]}(\mathcal{S})$ is \emph{diagonalizable} in the direction $[x_j]$ ($1\leq j\leq n$) at the \emph{spots} $\mathcal{S}_{t},\mathcal{S}_{r}\in \mathcal{F}$ such that $\mathcal{S}_t-\mathcal{S}_r$ is not a tuple of $\mathbb{R}$ of orders $k_t$ and $k_r$, respectively, and $k_r>k_t$ then
\begin{align}
(\gamma^{-1}\circ\beta\circ\gamma\circ\nabla)^{k_t}_{[x_{j}]}(\mathcal{S}_t)\leq (\gamma^{-1}\circ\beta\circ\gamma\circ\nabla)^{k_r}_{[x_{j}]}(\mathcal{S}_r).\nonumber
\end{align}
\end{proposition}

\begin{proof}
Let $\mathcal{F}=\{\mathcal{S}_i\}_{i=1}^{\infty}$ be a collection of tuples of polynomials in the ring $\mathbb{R}[x_1,x_2,\ldots,x_n]$ and the mixed expansion $(\gamma^{-1}\circ\beta\circ\gamma\circ\nabla)_{\otimes_{i=1}^{l}[x_{\sigma(i)}]}(\mathcal{S})$ be diagonalizable in the direction $[x_j]$ ($1\leq j\leq n$) at the spots $\mathcal{S}_{t},\mathcal{S}_{r}\in \mathcal{F}$ such that $\mathcal{S}_t-\mathcal{S}_r$ is not a tuple of $\mathbb{R}$ with orders $k_t$ and $k_r$ so that 
\begin{align}
(\gamma^{-1}\circ\beta\circ\gamma\circ\nabla)_{\otimes_{i=1}^{l}[x_{\sigma(i)}]}(\mathcal{S})=(\gamma^{-1}\circ\beta\circ\gamma\circ\nabla)^{k_t}_{[x_{j}]}(\mathcal{S}_t)\label{disub1}
\end{align}
and 
\begin{align}
(\gamma^{-1}\circ\beta\circ\gamma\circ\nabla)_{\otimes_{i=1}^{l}[x_{\sigma(i)}]}(\mathcal{S})=(\gamma^{-1}\circ\beta\circ\gamma\circ\nabla)^{k_r}_{[x_{j}]}(\mathcal{S}_r).\label{disub2}
\end{align}
Combining \eqref{disub1} and \eqref{disub2}, we deduce
\begin{align}
(\gamma^{-1}\circ\beta\circ\gamma\circ\nabla)^{k_t}_{[x_{j}]}(\mathcal{S}_t)=(\gamma^{-1}\circ\beta\circ\gamma\circ\nabla)^{k_r}_{[x_{j}]}(\mathcal{S}_r)\nonumber
\end{align}
since $\mathcal{S}_t-\mathcal{S}_r$ is not a tuple of $\mathbb{R}$. Since $k_r>k_t$, it follows that there exist some $m\geq 1$ such that 
\begin{align}
(\gamma^{-1}\circ\beta\circ\gamma\circ\nabla)^{k_t+m}_{[x_{j}]}(\mathcal{S}_r)=(\gamma^{-1}\circ\beta\circ\gamma\circ\nabla)^{k_t}_{[x_{j}]}(\mathcal{S}_t)\nonumber
\end{align}
so that 
\begin{align}
(\gamma^{-1}\circ\beta\circ\gamma\circ\nabla)^{k_t}_{[x_{j}]}(\mathcal{S}_t)\leq (\gamma^{-1}\circ\beta\circ\gamma\circ\nabla)^{k_r}_{[x_{j}]}(\mathcal{S}_r).\nonumber
\end{align}
\end{proof}
\bigskip

The converse of Proposition \ref{diagsubexpansion1} may not necessarily hold because the sub-expansion may be ancient. We can establish the converse when we allow the sub-expansion to be a proper sub-expansion. 

\begin{proposition}\label{diagsubexpansion2}
Let $\mathcal{F}=\{\mathcal{S}_i\}_{i=1}^{\infty}$ be a collection of tuples of polynomials in the ring $\mathbb{R}[x_1,x_2,\ldots,x_n]$. If the expansion $(\gamma^{-1}\circ\beta\circ\gamma\circ\nabla)_{[x_{i}]}(\mathcal{S}_t)$ is a diagonal of order $k$ of the mixed expansion $(\gamma^{-1}\circ\beta\circ\gamma\circ\nabla)_{\otimes_{i=1}^{l}[x_{\sigma(i)}]}(\mathcal{S})$ and 
\begin{align}
(\gamma^{-1}\circ\beta\circ\gamma\circ\nabla)^l_{[x_{i}]}(\mathcal{S}_r)<(\gamma^{-1}\circ\beta\circ\gamma\circ\nabla)^k_{[x_{i}]}(\mathcal{S}_t) \nonumber
\end{align}
then the expansion $(\gamma^{-1}\circ\beta\circ\gamma\circ\nabla)_{[x_{i}]}(\mathcal{S}_r)$ is also a diagonal of order $l$ of the same mixed expansion.
\end{proposition}

\begin{proof}
We suppose that the expansion $(\gamma^{-1}\circ\beta\circ\gamma\circ\nabla)_{[x_{i}]}(\mathcal{S}_t)$ is the diagonal of order $k$ of the mixed expansion $(\gamma^{-1}\circ\beta\circ\gamma\circ\nabla)_{\otimes_{i=1}^{l}[x_{\sigma(i)}]}(\mathcal{S})$. By definition \ref{diagonalization}, we get
\begin{align}
(\gamma^{-1}\circ\beta\circ\gamma\circ\nabla)_{\otimes_{i=1}^{l}[x_{\sigma(i)}]}(\mathcal{S})&=(\gamma^{-1}\circ\beta\circ\gamma\circ\nabla)^k_{[x_{i}]}(\mathcal{S}_t).\nonumber
\end{align}
Since
\begin{align}
(\gamma^{-1}\circ\beta\circ\gamma\circ\nabla)^l_{[x_{i}]}(\mathcal{S}_r)<(\gamma^{-1}\circ\beta\circ\gamma\circ\nabla)^k_{[x_{i}]}(\mathcal{S}_t) \nonumber
\end{align}
the definition \ref{subexpansion} implies
\begin{align}
(\gamma^{-1}\circ\beta\circ\gamma\circ\nabla)^l_{[x_{i}]}(\mathcal{S}_r)=(\gamma^{-1}\circ\beta\circ\gamma\circ\nabla)^{l+m}_{[x_{i}]}(\mathcal{S}_t).\nonumber
\end{align}
for some $0\leq m$ with $l+m=k$ so that 
\begin{align}
(\gamma^{-1}\circ\beta\circ\gamma\circ\nabla)^l_{[x_{i}]}(\mathcal{S}_r)=(\gamma^{-1}\circ\beta\circ\gamma\circ\nabla)^k_{[x_{i}]}(\mathcal{S}_t).\nonumber
\end{align}
We establish the claim since $(\gamma^{-1}\circ\beta\circ\gamma\circ \nabla)_{[x_{i}]}(\mathcal{S}_t)$ is a diagonal of order $k$ of the mixed expansion $(\gamma^{-1}\circ\beta\circ\gamma\circ\nabla)_{\otimes_{i=1}^{l}[x_{\sigma(i)}]}(\mathcal{S})$.
\end{proof}
\bigskip

The notion of the totient, the Doppler effect and the diagonalization of an expansion may seem to be quite separate disparate notions of the theory but the following Proposition indicates a subtle connection among these three.

\begin{proposition}\label{connection}
Let $\mathcal{F}=\{\mathcal{S}_i\}_{i=1}^{\infty}$ be a collection of tuples of polynomials in the ring $\mathbb{R}[x_1,x_2,\ldots,x_n]$. If the mixed expansion $(\gamma^{-1}\circ\beta\circ\gamma\circ\nabla)_{\otimes_{i=1}^{l}[x_{\sigma(i)}]}(\mathcal{S})$ induces a Doppler effect with intensity $k$ on the expansion $(\gamma^{-1}\circ\beta\circ\gamma\circ\nabla)_{[x_j]}(\mathcal{S})$ and is diagonalizable in the direction $[x_j]$ at the spot $\mathcal{S}_t$ of order $s$, then the expansion 
\begin{align}
(\gamma^{-1}\circ\beta\circ\gamma\circ\nabla)_{[x_j]}(\mathcal{S}_t)\nonumber
\end{align}
is free with totient 
\begin{align}
\Phi[(\gamma^{-1}\circ\beta\circ\gamma\circ\nabla)_{[x_j]}(\mathcal{S}_t)]=k+s.\nonumber
\end{align}
\end{proposition}

\begin{proof}
We suppose that the mixed expansion $(\gamma^{-1}\circ\beta\circ\gamma\circ\nabla)_{\otimes_{i=1}^{l}[x_{\sigma(i)}]}(\mathcal{S})$ induces a Doppler effect with intensity $k$ on the expansion $(\gamma^{-1}\circ\beta\circ\gamma\circ\nabla)_{[x_j]}(\mathcal{S})$. By definition \ref{dropler effect}, we have
\begin{align}
(\gamma^{-1}\circ\beta\circ\gamma\circ\nabla)^{k}_{[x_j]}\circ (\gamma^{-1}\circ\beta\circ\gamma\circ\nabla)_{\otimes_{i=1}^{l}[x_{\sigma(i)}]}(\mathcal{S})=\mathcal{S}_0\nonumber
\end{align}
with 
$$
k<\Phi[(\gamma^{-1}\circ\beta\circ\gamma\circ\nabla)_{[x_j]}(\mathcal{S})]
$$ 
and $k$ is the smallest such number. Under the assumption the mixed expansion $(\gamma^{-1}\circ\beta\circ\gamma\circ\nabla)_{\otimes_{i=1}^{l}[x_{\sigma(i)}]}(\mathcal{S})$ is diagonalizable in the direction $[x_j]$ at the spot $\mathcal{S}_t$ with order $s$, we get 
\begin{align}
(\gamma^{-1}\circ\beta\circ\gamma\circ\nabla)_{\otimes_{i=1}^{l}[x_{\sigma(i)}]}(\mathcal{S})=(\gamma^{-1}\circ\beta\circ\gamma\circ\nabla)^s_{[x_j]}(\mathcal{S}_t)\nonumber
\end{align}
so that we have 
\begin{align}
(\gamma^{-1}\circ\beta\circ\gamma\circ\nabla)^{k+s}_{[x_j]}(\mathcal{S}_t)=\mathcal{S}_0.\nonumber
\end{align}
Now, we suppose that there exist some $r\leq k+s$ such that 
\begin{align}
(\gamma^{-1}\circ\beta\circ\gamma\circ\nabla)^{k+s-r}_{[x_j]}(\mathcal{S}_t)=\mathcal{S}_0.\nonumber
\end{align}
We have 
\begin{align}
(\gamma^{-1}\circ\beta\circ\gamma\circ\nabla)^{k-r}_{[x_j]}\circ (\gamma^{-1}\circ\beta\circ\gamma\circ\nabla)_{\otimes_{i=1}^{l}[x_{\sigma(i)}]}(\mathcal{S})=\mathcal{S}_0.\nonumber
\end{align}
This is a contradiction, since $k=\mathcal{I}[(\gamma^{-1}\circ\beta\circ\gamma\circ\nabla)_{[x_j]}(\mathcal{S})]$ is the intensity of the doppler effect and is the smallest such number. It follows that 
\begin{align}
\Phi[(\gamma^{-1}\circ\beta\circ\gamma\circ\nabla)_{[x_j]}(\mathcal{S}_t)]&=\mathcal{I}[(\gamma^{-1}\circ\beta\circ\gamma\circ\nabla)_{[x_j]}(\mathcal{S})]+\mathcal{O}[(\gamma^{-1}\circ\beta\circ\gamma\circ\nabla)_{[x_j]}(\mathcal{S}_t)]\nonumber \\&=k+s\nonumber
\end{align}
and the claim follows immediately.
\end{proof}
\bigskip

\begin{lemma}\label{admit dropler}
Let $\mathcal{F}=\{\mathcal{S}_i\}_{i=1}^{\infty}$ be a collection of tuples of polynomials in the ring $\mathbb{R}[x_1,x_2,\ldots,x_n]$. The expansion $(\gamma^{-1}\circ\beta\circ\gamma\circ\nabla)_{[x_{\sigma(i)}]}(\mathcal{S})$ for all $1\leq i\leq l$ admits the Doppler effect from the source
\begin{align}
(\gamma^{-1}\circ\beta\circ\gamma\circ\nabla)_{\otimes_{i=1}^{l}[x_{\sigma(l)}]}(\mathcal{S}).\nonumber
\end{align}
\end{lemma}

\begin{proof}
We consider an arbitrary expansion $(\gamma^{-1}\circ\beta\circ\gamma\circ\nabla)_{[x_{\sigma(j)}]}(\mathcal{S})$ for all $1\leq j\leq l$. Using the commutative property of an expansion, we can write 
\begin{align}
(\gamma^{-1}\circ\beta\circ\gamma\circ\nabla)_{\otimes_{i=1}^{l}[x_{\sigma(i)}]}(\mathcal{S})&=(\gamma^{-1}\circ\beta\circ\gamma\circ\nabla)_{[x_{\sigma(j)}]} \nonumber \\& \circ (\gamma^{-1}\circ\beta\circ\gamma\circ\nabla)_{\otimes_{\substack{i=1\\i\neq j}}^{l}[x_{\sigma(i)}]}(\mathcal{S}).\nonumber
\end{align}
We deduce
\begin{align}
(\gamma^{-1}\circ\beta\circ\gamma\circ\nabla)^{\Phi[(\gamma^{-1}\circ\beta\circ\gamma\circ\nabla)_{[x_{\sigma(j)}]}(\mathcal{S})]-1}_{[x_{\sigma(j)}]}\circ (\gamma^{-1}\circ\beta\circ\gamma\circ\nabla)_{\otimes_{i=1}^{l}[x_{\sigma(i)}]}(\mathcal{S})=\mathcal{S}_0.\nonumber
\end{align}
It follows that there exists some smallest number $k\leq \Phi[(\gamma^{-1}\circ\beta\circ\gamma\circ\nabla)_{[x_{\sigma(j)}]}(\mathcal{S})]-1<\Phi[(\gamma^{-1}\circ\beta\circ\gamma\circ\nabla)_{[x_{\sigma(j)}]}(\mathcal{S})]$ such that 
\begin{align}
(\gamma^{-1}\circ\beta\circ\gamma\circ\nabla)^{k}_{[x_{\sigma(j)}]}\circ (\gamma^{-1}\circ\beta\circ\gamma\circ\nabla)_{\otimes_{i=1}^{l}[x_{\sigma(i)}]}(\mathcal{S})=\mathcal{S}_0.\nonumber
\end{align}
This proves the claim that each expansion of the form $(\gamma^{-1}\circ\beta\circ\gamma\circ\nabla)_{[x_{\sigma(j)}]}(\mathcal{S})$ for all $1\leq j\leq l$ admits a Doppler effect from the source 
\begin{align}
(\gamma^{-1}\circ\beta\circ\gamma\circ\nabla)_{\otimes_{i=1}^{l}[x_{\sigma(i)}]}(\mathcal{S}).\nonumber
\end{align}
\end{proof}
\bigskip

Here, we show that the notion of diagonalization exists for mixed expansion in each direction involved in the mixed expansion. The proof is quite iterative in nature and will be used in the sequel.

\begin{proposition}\label{exist diagonalization}
Let $\mathcal{F}=\{\mathcal{S}_i\}_{i=1}^{\infty}$ be a collection of tuples of polynomials in the ring $\mathbb{R}[x_1,x_2,\ldots,x_n]$. The mixed expansion 
\begin{align}
(\gamma^{-1}\circ\beta\circ\gamma\circ\nabla)_{\otimes_{i=1}^{l}[x_{\sigma(i)}]}(\mathcal{S})\nonumber
\end{align}
is diagonalizable in each direction $[x_{\sigma(i)}]$ for $1\leq i\leq l$.
\end{proposition}

\begin{proof}
We consider the mixed expansion 
\begin{align}
(\gamma^{-1}\circ\beta\circ\gamma\circ\nabla)_{\otimes_{i=1}^{l}[x_{\sigma(i)}]}(\mathcal{S})\nonumber
\end{align}
and let $[x_{\sigma(j)}]$ for $1 \leq j\leq l$ be our targeted direction. Using the commutative property of an expansion, we have
\begin{align}
(\gamma^{-1}\circ\beta\circ\gamma\circ\nabla)_{\otimes_{i=1}^{l}[x_{\sigma(i)}]}(\mathcal{S})&=(\gamma^{-1}\circ\beta\circ\gamma\circ\nabla)_{[x_{\sigma(j)}]}\circ (\gamma^{-1}\circ\beta\circ\gamma\circ\nabla)_{\otimes_{\substack{i=1\\i\neq j}}^{l}[x_{\sigma(i)}]}(\mathcal{S}).\nonumber
\end{align}
We consider the residual mixed expansion
\begin{align}
(\gamma^{-1}\circ\beta\circ\gamma\circ\nabla)_{\otimes_{\substack{i=1\\i\neq j}}^{l}[x_{\sigma(i)}]}(\mathcal{S})&=(\gamma^{-1}\circ\beta\circ\gamma\circ\nabla)_{\otimes_{\substack{i=1\\i\neq j}}^{l}[x_{\sigma(i)}]}\nonumber \\& \circ (\gamma^{-1}\circ\beta\circ\gamma\circ\nabla)_{[x_{\sigma(1)}]}(\mathcal{S}).\nonumber
\end{align}
If there exists some tuple $\mathcal{S}_a\in \mathcal{F}$ such that 
\begin{align}
(\gamma^{-1}\circ\beta\circ\gamma\circ\nabla)_{[x_{\sigma(1)}]}(\mathcal{S})&=(\gamma^{-1}\circ\beta\circ\gamma\circ\nabla)_{[x_{\sigma(j)}]}(\mathcal{S}_a)\nonumber
\end{align}
then we make a substitution and obtain two copies of the expansion operator $(\gamma^{-1}\circ\beta\circ\gamma\circ\nabla)_{[x_{\sigma(j)}]}$ by using the commutative property of an expansion. Otherwise, we choose 
\begin{align}
\mathcal{S}_b=(\gamma^{-1}\circ\beta\circ\gamma\circ\nabla)_{[x_{\sigma(1)}]}(\mathcal{S})\nonumber
\end{align}
and apply the remaining operators on it. Iterating in this manner, we will obtain 
\begin{align}
(\gamma^{-1}\circ\beta\circ\gamma\circ\nabla)_{\otimes_{i=1}^{l}[x_{\sigma(i)}]}(\mathcal{S})&=(\gamma^{-1}\circ\beta\circ\gamma\circ\nabla)^k_{[x_{\sigma(j)}]}(\mathcal{S}_t)\nonumber
\end{align}
for $k\geq 1$ and for some $\mathcal{S}_t\in \mathcal{F}$. This completes the proof of the proposition.
\end{proof}
\bigskip

We are now ready to prove the inequality announced at the outset of the paper. We gather the tools developed in the previous section to obtain a stronger version of the inequality.

\begin{theorem}\label{mixed specific inequality}
Let $\mathcal{F}=\{\mathcal{S}_i\}_{i=1}^{\infty}$ be a collection of tuples of polynomials in the ring $\mathbb{R}[x_1,x_2,\ldots,x_n]$. We have the inequality
\begin{align}
\Phi[(\gamma^{-1}\circ\beta\circ\gamma\circ\nabla)_{\otimes_{i=1}^{l}[x_{\sigma(i)}]}(\mathcal{S})]&<\frac{1}{l}\sum \limits_{i=1}^{l}\Phi[(\gamma^{-1}\circ\beta\circ\gamma\circ\nabla)_{[x_{\sigma(i)}]}(\mathcal{S})]\nonumber \\&+\frac{1}{l}\sum \limits_{\substack{1\leq i\leq l\\\mathcal{S}_r\in \mathrm{Diag}[(\gamma^{-1}\circ\beta\circ\gamma\circ\nabla)_{\otimes_{i=1}^{l}[x_{\sigma(i)}]}(\mathcal{S})]\\\mathcal{S}_r=(\gamma^{-1}\circ\beta\circ\gamma\circ\nabla)_{[x_{\sigma(i)}]}(\mathcal{S}_t)}}\mathcal{O}[(\gamma^{-1}\circ\beta\circ\gamma\circ\nabla)_{[x_{\sigma(i)}]}(\mathcal{S}_t)]\nonumber
\end{align}
where 
$$
\mathrm{Diag}[(\gamma^{-1}\circ\beta\circ\gamma\circ\nabla)_{\otimes_{i=1}^{l}[x_{\sigma(i)}]}(\mathcal{S})]
$$ 
is the set of all diagonals of the expansion 
\begin{align}
(\gamma^{-1}\circ\beta\circ\gamma\circ\nabla)_{\otimes_{i=1}^{l}[x_{\sigma(i)}]}(\mathcal{S}).\nonumber
\end{align}
\end{theorem}

\begin{proof}
We consider the mixed expansion 
\begin{align}
(\gamma^{-1}\circ\beta\circ\gamma\circ\nabla)_{\otimes_{i=1}^{l}[x_{\sigma(i)}]}(\mathcal{S}).\nonumber
\end{align}
Using Proposition \ref{exist diagonalization}, we find that for each direction $[x_{\sigma(i)}]$ for $1\leq i\leq l$ there exist some spot $\mathcal{S}_t$ and a number $k\geq 1$ such that we can write
\begin{align}
(\gamma^{-1}\circ\beta\circ\gamma\circ\nabla)_{\otimes_{i=1}^{l}[x_{\sigma(i)}]}(\mathcal{S})&=(\gamma^{-1}\circ\beta\circ\gamma\circ\nabla)^k_{[x_{\sigma(i)}]}(\mathcal{S}_t).\nonumber
\end{align}
Again, using the lemma \ref{admit dropler}, we find that each of the expansions $(\gamma^{-1}\circ\beta\circ\gamma\circ\nabla)^k_{[x_{\sigma(i)}]}(\mathcal{S})$ admits a Doppler effect from the source 
\begin{align}
(\gamma^{-1}\circ\beta\circ\gamma\circ\nabla)_{\otimes_{i=1}^{l}[x_{\sigma(i)}]}(\mathcal{S}).\nonumber
\end{align}
We deduce for each direction $[x_{\sigma(i)}]$ the relation
\begin{align}
\Phi[(\gamma^{-1}\circ\beta\circ\gamma\circ\nabla)_{[x_i]}(\mathcal{S}_t)]&=\mathcal{I}[(\gamma^{-1}\circ\beta\circ\gamma\circ\nabla)_{[x_i]}(\mathcal{S})]+\mathcal{O}[(\gamma^{-1}\circ\beta\circ\gamma\circ\nabla)_{[x_i]}(\mathcal{S}_t)].\nonumber
\end{align}
Using the definition \ref{dropler effect}, we obtain further the inequality
\begin{align}
\Phi[(\gamma^{-1}\circ\beta\circ\gamma\circ\nabla)_{[x_i]}(\mathcal{S}_t)]&<\Phi[(\gamma^{-1}\circ\beta\circ\gamma\circ\nabla)_{[x_i]}(\mathcal{S})]+\mathcal{O}[(\gamma^{-1}\circ\beta\circ\gamma\circ\nabla)_{[x_i]}(\mathcal{S}_t)].\nonumber
\end{align}
We deduce
\begin{align}
\Phi[(\gamma^{-1}\circ\beta\circ\gamma\circ\nabla)_{{\otimes}_{i=1}^{l}[x_{\sigma(i)}]}(\mathcal{S})]&=\Phi[(\gamma^{-1}\circ\beta\circ\gamma\circ\nabla)^k_{[x_i]}(\mathcal{S}_t)]\nonumber \\&\leq \Phi[(\gamma^{-1}\circ\beta\circ\gamma\circ\nabla)_{[x_i]}(\mathcal{S}_t)]\nonumber 
\end{align}
so that we have the refined inequality
\begin{align}
\Phi[(\gamma^{-1}\circ\beta\circ\gamma\circ\nabla)_{{\otimes}_{i=1}^{l}[x_{\sigma(i)}]}(\mathcal{S})]<\Phi[(\gamma^{-1}\circ\beta\circ\gamma\circ\nabla)_{[x_i]}(\mathcal{S})]+\mathcal{O}[(\gamma^{-1}\circ\beta\circ\gamma\circ\nabla)_{[x_i]}(\mathcal{S}_t)].\nonumber
\end{align}
Since there are $l$ directions under consideration, we add $l$ such chains of the inequality and obtain
\begin{align}
l\Phi[(\gamma^{-1}\circ\beta\circ\gamma\circ \nabla)_{\otimes_{i=1}^{l}[x_{\sigma(i)}]}(\mathcal{S})]&<\sum \limits_{i=1}^{l}\Phi[(\gamma^{-1}\circ\beta\circ\gamma\circ\nabla)_{[x_{\sigma(i)}]}(\mathcal{S})]\nonumber \\&+\sum \limits_{\substack{1\leq i\leq l\\\mathcal{S}_r\in \mathrm{Diag}[(\gamma^{-1}\circ\beta\circ\gamma\circ\nabla)_{\otimes_{i=1}^{l}[x_{\sigma(i)}]}(\mathcal{S})]\\\mathcal{S}_r=(\gamma^{-1}\circ\beta\circ\gamma\circ\nabla)_{[x_{\sigma(i)}]}(\mathcal{S}_t)}}\mathcal{O}[(\gamma^{-1}\circ\beta\circ\gamma\circ\nabla)_{[x_{\sigma(i)}]}(\mathcal{S}_t)].\nonumber
\end{align}
This completes the proof of the theorem.
\end{proof}
\bigskip

\begin{corollary}\label{actual inequality}
Let $\mathcal{F}:=\{\mathcal{S}_i\}_{i=1}^{\infty}$ be a collection of tuples of polynomials in the ring $\mathbb{R}[x_1,x_2,\ldots,x_n]$. If $\mathcal{O}[(\gamma^{-1}\circ\beta\circ\gamma\circ\nabla)_{[x_{\sigma(i)}]}(\mathcal{S}_t)]=1$ for each 
\begin{align}
(\gamma^{-1}\circ\beta\circ\gamma\circ\nabla)_{[x_{\sigma(i)}]}(\mathcal{S}_t) \in \mathrm{Diag}[(\gamma^{-1}\circ\beta\circ\gamma\circ\nabla)_{\otimes_{i=1}^{l}[x_{\sigma(i)}]}(\mathcal{S})]\nonumber
\end{align}
then 
\begin{align}
\Phi[(\gamma^{-1}\circ\beta\circ\gamma\circ\nabla)_{\otimes_{i=1}^{l}[x_{\sigma(i)}]}(\mathcal{S})]&<\frac{1}{l}\sum \limits_{i=1}^{l}\Phi[(\gamma^{-1}\circ\beta\circ\gamma\circ\nabla)_{[x_{\sigma(i)}]}(\mathcal{S})]+1.\nonumber
\end{align}
\end{corollary}

\begin{proof}
This is a consequence of the inequality in Theorem \ref{mixed specific inequality} by taking
\begin{align}
\mathcal{O}[(\gamma^{-1}\circ\beta\circ\gamma\circ\nabla)_{[x_{\sigma(i)}]}(\mathcal{S}_t)]=1\nonumber
\end{align}
for each 
\begin{align}
(\gamma^{-1}\circ\beta\circ\gamma\circ\nabla)_{[x_{\sigma(i)}]}(\mathcal{S}_t) \in \mathrm{Diag}[(\gamma^{-1}\circ\beta\circ\gamma\circ\nabla)_{\otimes_{i=1}^{l}[x_{\sigma(i)}]}(\mathcal{S})].\nonumber
\end{align}
\end{proof}
\bigskip

Using the energy-Doppler effect-intensity equation in definition \ref{dropler effect}
\begin{align}
\mathbb{E}[(\gamma^{-1}\circ\beta\circ\gamma\circ\nabla)_{[x_j]}](\mathcal{S})=\Phi[(\gamma^{-1}\circ\beta\circ\gamma\circ\nabla)_{[x_j]}(\mathcal{S})]-\mathcal{I}[(\gamma^{-1}\circ\beta\circ\gamma\circ\nabla)_{[x_j]}(\mathcal{S})]\nonumber
\end{align}
and Theorem \ref{mixed specific inequality}, we obtain a refined inequality 
\begin{align}
\Phi[(\gamma^{-1}\circ\beta\circ\gamma\circ\nabla)_{\otimes_{i=1}^{l}[x_{\sigma(i)}]}(\mathcal{S})]&<\frac{1}{l}\sum \limits_{i=1}^{l}\mathbb{E}[(\gamma^{-1}\circ\beta\circ\gamma\circ\nabla)_{[x_{\sigma(i)}]}(\mathcal{S})]\nonumber \\&+\frac{1}{l}\sum \limits_{i=1}^{l}\mathcal{I}[(\gamma^{-1}\circ\beta\circ\gamma\circ\nabla)_{[x_{\sigma(i)}]}(\mathcal{S})]\nonumber \\&+\frac{1}{l}\sum \limits_{\substack{1\leq i\leq l\\\mathcal{S}_r\in \mathrm{Diag}[(\gamma^{-1}\circ\beta\circ\gamma\circ\nabla)_{\otimes_{i=1}^{l}[x_{\sigma(i)}]}(\mathcal{S})]\\\mathcal{S}_r=(\gamma^{-1}\circ\beta\circ\gamma\circ\nabla)_{[x_{\sigma(i)}]}(\mathcal{S}_t)}}\mathcal{O}[(\gamma^{-1}\circ\beta\circ\gamma\circ\nabla)_{[x_{\sigma(i)}]}(\mathcal{S}_t)].\nonumber
\end{align}
We call this inequality the totient-energy-Doppler effect intensity inequality.
\bigskip

\section{Hybrid expansions}

In this section, we introduce and study the notion of \emph{hybrid} expansions.

\begin{definition}\label{hybrid}
Let $\mathcal{F}=\{\mathcal{S}_i\}_{i=1}^{\infty}$ be a collection of tuples of polynomials in the ring $\mathbb{R}[x_1,x_2,\ldots,x_n]$. We say that the expansions $(\gamma^{-1}\circ\beta\circ\gamma\circ\nabla)^k_{[x_{i}]}(\mathcal{S}_a)$ and $(\gamma^{-1}\circ\beta\circ\gamma\circ\nabla)^t_{[x_{j}]}(\mathcal{S}_b)$ with $i\neq j$ are \emph{hybrid} if 
\begin{align}
(\gamma^{-1}\circ\beta\circ\gamma\circ\nabla)^k_{[x_{i}]}(\mathcal{S}_a)=(\gamma^{-1}\circ\beta\circ\gamma\circ\nabla)^t_{[x_{j}]}(\mathcal{S}_b). \nonumber
\end{align}
We denote this relationship by
\begin{align}
(\gamma^{-1}\circ\beta\circ\gamma\circ\nabla)^k_{[x_{i}]}(\mathcal{S}_a)\Join (\gamma^{-1}\circ\beta\circ\gamma\circ\nabla)^t_{[x_{j}]}(\mathcal{S}_b). \nonumber
\end{align}
\end{definition}
\bigskip

\begin{proposition}\label{diaghybrid}
Let $\mathcal{F}=\{\mathcal{S}_i\}_{i=1}^{\infty}$ be a collection of tuples of polynomials in the ring $\mathbb{R}[x_1,x_2,\ldots,x_n]$. If the mixed expansion $(\gamma^{-1}\circ\beta \circ\gamma\circ\nabla)_{\otimes_{i=1}^{l}[x_{\sigma(i)}]}(\mathcal{S})$ is diagonalizable at the spot $\mathcal{S}_a$ of order $k$ in the direction $[x_i]$ and 
\begin{align}
(\gamma^{-1}\circ\beta\circ\gamma\circ\nabla)^k_{[x_{i}]}(\mathcal{S}_a)\Join (\gamma^{-1}\circ\beta\circ\gamma\circ\nabla)^t_{[x_{j}]}(\mathcal{S}_b) \nonumber
\end{align}
then the mixed expansion is also diagonalizable at the spot $\mathcal{S}_b$ of order $t$ in the direction $[x_j]$.
\end{proposition}

\begin{proof}
Suppose that the mixed expansion $(\gamma^{-1}\circ\beta\circ\gamma\circ\nabla)_{\otimes_{i=1}^{l}[x_{\sigma(i)}]}(\mathcal{S})$ is diagonalizable at the spot $\mathcal{S}_a$ of order $k$ in the direction $[x_i]$. By the definition \ref{diagonalization}, we have
\begin{align}
(\gamma^{-1}\circ\beta\circ\gamma\circ\nabla)_{\otimes_{i=1}^{l}[x_{\sigma(i)}]}(\mathcal{S})=(\gamma^{-1}\circ\beta\circ\gamma\circ\nabla)^k_{[x_{i}]}(\mathcal{S}_a).\nonumber
\end{align}
Under the assumption that the expansions are hybrid, we get
\begin{align}
(\gamma^{-1}\circ\beta\circ\gamma\circ\nabla)_{\otimes_{i=1}^{l}[x_{\sigma(i)}]}(\mathcal{S})=(\gamma^{-1}\circ\beta\circ\gamma\circ\nabla)^t_{[x_{j}]}(\mathcal{S}_b)\nonumber
\end{align}
and the claim follows immediately.
\end{proof}
\bigskip

\begin{proposition}\label{hybrid totient control}
Let $\mathcal{F}=\{\mathcal{S}_i\}_{i=1}^{\infty}$ be a collection of tuples of polynomial in the ring $\mathbb{R}[x_1,x_2,\ldots,x_n]$. Let $(\gamma^{-1}\circ\beta\circ\gamma\circ\nabla)_{[x_{i}]}(\mathcal{S}_a)$ be a diagonal of the mixed expansion 
\begin{align}
(\gamma^{-1}\circ\beta\circ\gamma\circ\nabla)_{\otimes_{i=1}^{l}[x_{\sigma(i)}]}(\mathcal{S})\nonumber
\end{align}
of order $k\geq 1$. If 
\begin{align}
(\gamma^{-1}\circ\beta\circ\gamma\circ\nabla)^k_{[x_{i}]}(\mathcal{S}_a)\Join (\gamma^{-1}\circ\beta\circ\gamma\circ\nabla)^t_{[x_{j}]}(\mathcal{S}_b) \nonumber
\end{align}
then 
\begin{align}
\Phi[(\gamma^{-1}\circ\beta\circ\gamma\circ\nabla)^t_{[x_{j}]}(\mathcal{S}_b)]&<\mathrm{max}\{\Phi[(\gamma^{-1}\circ\beta\circ\gamma\circ\nabla)_{[x_{\sigma(i)}]}(\mathcal{S})]\}_{i=1}^{l}\nonumber \\&+\mathrm{max}\{\mathcal{O}[(\gamma^{-1}\circ\beta\circ\gamma\circ\nabla)_{[x_{\sigma(i)}]}(\mathcal{S}_t)]\}_{\substack{i=1\\\mathcal{S}_t\in \mathrm{Diag}[(\gamma^{-1}\circ\beta\circ\gamma\circ\nabla)_{\otimes_{i=1}^{l}[x_{\sigma(i)}]}(\mathcal{S})]}}^{l}.\nonumber
\end{align}
\end{proposition}

\begin{proof}
Suppose that the expansion $(\gamma^{-1}\circ\beta\circ\gamma\circ\nabla)_{[x_{i}]}(\mathcal{S}_a)$ is a diagonal of the mixed expansion 
\begin{align}
(\gamma^{-1}\circ\beta\circ\gamma\circ\nabla)_{\otimes_{i=1}^{l}[x_{\sigma(i)}]}(\mathcal{S})\nonumber
\end{align}
of order $k\geq 1$. We deduce
\begin{align}
(\gamma^{-1}\circ\beta\circ\gamma\circ\nabla)_{\otimes_{i=1}^{l}[x_{\sigma(i)}]}(\mathcal{S})&=(\gamma^{-1}\circ\beta\circ\gamma\circ\nabla)^k_{[x_{i}]}(\mathcal{S}_a).\nonumber
\end{align}
Since 
\begin{align}
(\gamma^{-1}\circ\beta\circ\gamma\circ\nabla)^k_{[x_{i}]}(\mathcal{S}_a)\Join (\gamma^{-1}\circ\beta\circ\gamma\circ\nabla)^t_{[x_{j}]}(\mathcal{S}_b) \nonumber
\end{align}
we can write
\begin{align}
(\gamma^{-1}\circ\beta\circ\gamma\circ\nabla)_{\otimes_{i=1}^{l}[x_{\sigma(i)}]}(\mathcal{S})&=(\gamma^{-1}\circ\beta\circ\gamma\circ\nabla)^t_{[x_{j}]}(\mathcal{S}_b) \nonumber
\end{align}
so that by using Theorem \ref{mixed specific inequality}, we obtain 
\begin{align}
\Phi[(\gamma^{-1}\circ\beta\circ\gamma\circ\nabla)^t_{[x_{j}]}(\mathcal{S}_b)]&=\Phi[(\gamma^{-1}\circ\beta\circ\gamma\circ\nabla)_{\otimes_{i=1}^{l}[x_{\sigma(i)}]}(\mathcal{S})]\nonumber \\&<\frac{1}{l}\sum \limits_{i=1}^{l}\Phi[(\gamma^{-1}\circ\beta\circ\gamma\circ\nabla)_{[x_{\sigma(i)}]}(\mathcal{S})]\nonumber \\&+\frac{1}{l}\sum \limits_{\substack{1\leq i\leq l\\\mathcal{S}_t\in \mathrm{Diag}[(\gamma^{-1}\circ\beta\circ\gamma\circ\nabla)_{\otimes_{i=1}^{l}[x_{\sigma(i)}]}(\mathcal{S})]}}\mathcal{O}[(\gamma^{-1}\circ\beta\circ\gamma\circ\nabla)_{[x_{\sigma(i)}]}(\mathcal{S}_t)]\nonumber
\end{align}
and the claim follows by further controlling the two sums on the right hand-side of the inequality.
\end{proof}
\bigskip

Here, we deduce the relationship between the notion hybridization and the notion of diagonalization of a mixed expansion.

\begin{proposition}\label{hybrid diagonalization connection}
Let $\mathcal{F}=\{\mathcal{S}_i\}_{i=1}^{\infty}$ be a collection of tuples of polynomials in the ring $\mathbb{R}[x_1,x_2,\ldots,x_n]$. If
\begin{align}
(\gamma^{-1}\circ\beta\circ\gamma\circ\nabla)^k_{[x_{i}]}(\mathcal{S}_a)\Join (\gamma^{-1}\circ\beta\circ\gamma\circ\nabla)^t_{[x_{j}]}(\mathcal{S}_b)\nonumber
\end{align}
then the mixed expansion $(\gamma^{-1}\circ\beta\circ\gamma\circ\nabla)_{[x_{i}]}\circ (\gamma^{-1}\circ\beta\circ\gamma\circ\nabla)^t_{[x_{j}]}(\mathcal{S}_b)$ respectively $(\gamma^{-1}\circ\beta\circ\gamma\circ\nabla)_{[x_{j}]}\circ (\gamma^{-1}\circ\beta\circ\gamma\circ\nabla)^k_{[x_{i}]}(\mathcal{S}_a)$ is diagonalizable at the spots $\mathcal{S}_a$ of order $k+1$ respectively $\mathcal{S}_b$ with order $t+1$.
\end{proposition}

\begin{proof}
Suppose that
\begin{align}
(\gamma^{-1}\circ\beta\circ\gamma\circ\nabla)^k_{[x_{i}]}(\mathcal{S}_a)\Join (\gamma^{-1}\circ\beta\circ\gamma\circ\nabla)^t_{[x_{j}]}(\mathcal{S}_b).\nonumber
\end{align}
We deduce
\begin{align}
(\gamma^{-1}\circ\beta\circ\gamma\circ\nabla)^k_{[x_{i}]}(\mathcal{S}_a)=(\gamma^{-1}\circ\beta\circ\gamma\circ\nabla)^t_{[x_{j}]}(\mathcal{S}_b)\nonumber
\end{align}
so that by applying a copy of the expansion $(\gamma^{-1}\circ\beta\circ\gamma\circ\nabla)_{[x_{i}]}$ respectively $(\gamma^{-1}\circ\beta\circ\gamma\circ \nabla)_{[x_{j}]}$ on both sides, we have the following relations 
\begin{align}
(\gamma^{-1}\circ\beta\circ\gamma\circ\nabla)_{[x_{i}]}\circ (\gamma^{-1}\circ\beta\circ\gamma\circ\nabla)^t_{[x_{j}]}(\mathcal{S}_b)&=(\gamma^{-1}\circ\beta\circ\gamma\circ\nabla)^{k+1}_{[x_{i}]}(\mathcal{S}_a)\nonumber 
\end{align}
and 
\begin{align}
(\gamma^{-1}\circ\beta\circ\gamma\circ\nabla)_{[x_{j}]}\circ (\gamma^{-1}\circ\beta\circ\gamma\circ\nabla)^k_{[x_{i}]}(\mathcal{S}_a)&=(\gamma^{-1}\circ\beta\circ\gamma\circ\nabla)^{t+1}_{[x_{j}]}(\mathcal{S}_b)\nonumber
\end{align}
and the claim follows immediately from these two relations.
\end{proof}
\bigskip

\section{Applications of the totient inequality}

In this section, we explore some applications of the theory. We obtain an inequality that will be useful for the study of the Pierce-Birkhoff conjecture. We first make the following precise.

\begin{definition}\label{index}
Let $f_k\in \mathbb{R}[x_1,x_2,\ldots,x_n]$ be a polynomial. By the index of $x_i$ for $1\leq i\leq n$ relative to $f_k$, denoted by $\mathrm{Ind}_{f_k}(x_i)$, we mean the largest power of $x_i$ in the polynomial $f_k$.
\end{definition}
\bigskip

\begin{lemma}\label{index-totient}
Let $\mathcal{S}=(f_1,f_2,\ldots,f_s)$ be a tuple of polynomials such that $f_i\in \mathbb{R}[x_1,x_2,\ldots,x_n]$ for $1\leq i \leq s$. For any $1\leq j\leq n$, we have
\begin{align}
\Phi[(\gamma^{-1}\circ\beta\circ\gamma\circ\nabla)_{[x_{j}]}(\mathcal{S})]=\mathrm{max}\{\mathrm{Ind}_{f_i}(x_j)\}_{i=1}^{s}+1.\nonumber
\end{align}
\end{lemma}
\bigskip

\begin{proposition}
Let $f_1,f_2,\ldots,f_s\in \mathbb{R}[x_1,x_2,\ldots.x_n]$ be polynomials. There exist some $\mathcal{J}:=\mathcal{J}(l)\geq 0$ such that 
\begin{align}
\mathrm{min}\{\mathrm{max}\{\mathrm{Ind}_{f_k}(x_{\sigma(i)})\}_{k=1}^{s}+1\}_{i=1}^{l}&<\frac{1}{l}\sum \limits_{i=1}^{l}\mathrm{max}\{\mathrm{Ind}_{f_k}(x_{\sigma(i)})\}_{k=1}^{s}+2+\mathcal{J}.\nonumber
\end{align}
\end{proposition}

\begin{proof}
Consider the tuple $\mathcal{S}=(f_1,f_2,\ldots,f_s)$. We break the proof into two special cases: The case were each of the expansions $(\gamma^{-1}\circ\beta\circ\gamma\circ\nabla)_{[x_{\sigma(i)}]}(\mathcal{S})$ for $1\leq i \leq l$ does not admit and the case at least one admits a Doppler effect from the source 
\begin{align}
(\gamma^{-1}\circ\beta\circ\gamma\circ\nabla)_{\otimes_{i=1}^{l}[x_{\sigma(i)}]}(\mathcal{S}).\nonumber
\end{align}
In the case each of the expansions admit no Doppler effect from the underlying source, we get by using Proposition \ref{mixed expansion dropler effect property} and Lemma \ref{index-totient}
\begin{align}
\Phi[(\gamma^{-1}\circ\beta\circ\gamma\circ\nabla)_{\otimes_{i=1}^{l}[x_{\sigma(i)}]}(\mathcal{S})]&=\mathrm{min}\{\Phi[(\gamma^{-1}\circ\beta\circ\gamma\circ\nabla)_{[x_{\sigma(i)}]}(\mathcal{S})]\}_{i=1}^{l}\nonumber \\&=\mathrm{min}\{\mathrm{max}\{\mathrm{Ind}_{f_k}(x_{\sigma(i)})\}_{k=1}^{s}+1\}_{i=1}^{l}.\nonumber 
\end{align}
By the lemma \ref{index-totient}, we can write 
\begin{align}
\Phi[(\gamma^{-1}\circ\beta\circ\gamma\circ \nabla)_{[x_{\sigma(i)}]}(\mathcal{S})]&=\mathrm{max}\{\mathrm{Ind}_{f_i}(x_j)\}_{i=1}^{s}+1.\nonumber
\end{align}
By Corollary \ref{actual inequality}, we get
\begin{align}
\mathrm{min}\{\mathrm{max}\{\mathrm{Ind}_{f_k}(x_{\sigma(i)})\}_{k=1}^{s}+1\}_{i=1}^{l}&<\frac{1}{l}\sum \limits_{i=1}^{l}\mathrm{max}\{\mathrm{Ind}_{f_k}(x_{\sigma(i)})\}_{k=1}^{s}+2.\nonumber
\end{align}
We now turn to the case where at least one of the expansions $(\gamma^{-1}\circ\beta\circ\gamma\circ\nabla)_{[x_{\sigma(i)}]}(\mathcal{S})$ for $1\leq i\leq l$ admits a Doppler effect. In this case, Proposition \ref{mixed expansion dropler effect property} gives
\begin{align}
\Phi[(\gamma^{-1}\circ\beta\circ\gamma\circ\nabla)_{\otimes_{i=1}^{l}[x_{\sigma(i)}]}(\mathcal{S})]&=\mathrm{min}\{(\gamma^{-1}\circ\beta\circ\gamma\circ\nabla)_{[x_{\sigma(i)}]}(\mathcal{S})\}_{i=1}^{l}-\mathcal{J}\nonumber \\&=\mathrm{min}\{\mathrm{max}\{\mathrm{Ind}_{f_k}(x_{\sigma(i)})\}_{k=1}^{s}+1\}_{i=1}^{l}-\mathcal{J}\nonumber 
\end{align}
for some $\mathcal{J}:=\mathcal{J}(l)>0$. The right hand side expression is not impacted in this case. Combining both cases, the claim inequality follows as a consequence.  
\end{proof}
\bigskip

\section{Exact expansion}

In this section, we introduce the notion of an \emph{exact} expansion.

\begin{definition}\label{exact}
Let $\mathcal{F}=\{\mathcal{S}_i\}_{i=1}^{\infty}$ be a collection of tuples of polynomials in the ring $\mathbb{R}[x_1,x_2,\ldots,x_n]$. We say that expansion $(\gamma^{-1}\circ\beta\circ\gamma\circ\nabla)_{[x_k]}(\mathcal{S})$ is \emph{exact} in directions $[x_{\sigma(1)}],\ldots,[x_{\sigma(l)}]$ each with multiplicity $1$ for $1\leq l\leq n$ and $\sigma:\{1,2,\ldots, n\}\longrightarrow \{1,2,\ldots,n\}$ at the spot $\mathcal{S}_1$ if there exists a number $s\in \mathbb{N}$, called the \emph{degree} of exactness, such that 
\begin{align}
  (\gamma^{-1}\circ\beta\circ\gamma\circ\nabla)^s_{[x_k]}(\mathcal{S})=(\gamma^{-1}\circ\beta\circ\gamma\circ\nabla)_{\otimes_{i=1}^{l}[x_{\sigma(i)}]}(\mathcal{S}_1).\nonumber  
\end{align}
In general, we say that the expansion $(\gamma^{-1}\circ\beta\circ\gamma\circ\nabla)_{[x_k]}(\mathcal{S})$ is \emph{exact} in the directions $[x_{\sigma(1)}],\ldots,[x_{\sigma(l)}]$ each with multiplicity $k_1,\ldots,k_l\in \mathbb{N}$ for $1\leq l\leq n$ with degree $s$ of exactness if 
\begin{align}
   (\gamma^{-1}\circ\beta\circ\gamma\circ\nabla)^s_{[x_k]}(\mathcal{S})=(\gamma^{-1}\circ\beta\circ\gamma\circ \nabla)_{\otimes_{i=1}^{l}[x_{\sigma(i)}]^{k_i}}(\mathcal{S}_1)\nonumber  
\end{align}
where $[x_{\sigma(i)}]^{k_i}=[x_{\sigma(i)}] \otimes [x_{\sigma(i)}]\cdots \otimes [x_{\sigma(i)}]~(k_i~times)$.
\end{definition}
\bigskip

The following web shows the commutative diagram of a typical exact expansion
\begin{tikzcd}
(\gamma^{-1}\circ\beta\circ\gamma\circ\nabla)_{[x_{\sigma(1)}]}(\mathcal{S}_1) \arrow{r}{\phi_2}
& (\gamma^{-1}\circ\beta\circ\gamma\circ\nabla)_{ \otimes_{i=1}^{2}[x_{\sigma(i)}]} (\mathcal{S}_1)\arrow{d}{\phi_3} \\
(\gamma^{-1}\circ\beta\circ\gamma\circ\nabla)_{[x_k]}(\mathcal{S}) \arrow{r}{\eta^2_k}
& (\gamma^{-1}\circ\beta\circ\gamma\circ\nabla)_{\otimes_{i=1}^{3}[x_{\sigma(i)}]}(\mathcal{S}_1)
\end{tikzcd}
 \\with degree $3$ of exactness, where
 $$
 \phi_l=(\gamma^{-1}\circ\beta\circ\gamma\circ\nabla)_{[x_{\sigma(l)}]}
 $$ 
 and 
 $$
 \phi_l\circ (\gamma^{-1}\circ\beta\circ\gamma\circ\nabla)_{[x_k]}(\mathcal{S})=(\gamma^{-1}\circ\beta\circ\gamma\circ\nabla)_{[x_k]\otimes [x_{\sigma(l)}]}(\mathcal{S})
 $$ 
 and 
 $$
 \eta_k^l=(\gamma^{-1}\circ\beta\circ\gamma\circ\nabla)^l_{[x_k]}
 $$ 
 for $1\leq l\leq n$. One can also construct more expanded commutative diagrams for exact expansion with arbitrarily large degrees. The notion of an exact expansion provides alternative paths for modeling an expansion in a specific direction. This type of expansion could conceivably be difficult and often delicate, so that a little distortion in the choice of directions may not lead to the targeted expansion.
 
 \begin{proposition}\label{diagonalization vs exactness}
 The expansion $(\gamma^{-1}\circ\beta\circ\gamma\circ\nabla)_{[x_k]}(\mathcal{S})$ is exact in the directions $[x_{\sigma(1)}],\ldots,[x_{\sigma(l)}]$ for $1\leq l\leq n$ and $\sigma:\{1,2,\ldots, n\}\longrightarrow \{1,2,\ldots,n\}$ at the spot $\mathcal{S}_1$ with degree $s\in \mathbb{N}$ if and only if the expansion $(\gamma^{-1}\circ\beta\circ\gamma\circ\nabla)_{\otimes_{i=1}^{l}[x_{\sigma(i)}]}(\mathcal{S}_1)$ is diagonalizable in the direction $[x_k]$ at the spot $\mathcal{S}$ of order $s$.
 \end{proposition}
 \bigskip
 
 The proposition \ref{diagonalization vs exactness} expresses the relationship between the notion of an exactness of an expansion and the diagonalization of an expansion. These two notions are similar except that the notion of exactness is applied to expansions in a specific direction, while the notion of diagonalization is appropriate for expansions in mixed directions. Whichever way one perceives these notions as different, they can both be considered as notions orthogonal to each other. We will show that the notion of exactness in directions can be extended to other directions.
 
 \begin{proposition}
 If the expansion $(\gamma^{-1}\circ\beta\circ\gamma\circ\nabla)_{[x_k]}(\mathcal{S})$ is exact in the directions $[x_{\sigma(1)}],\ldots,[x_{\sigma(l)}]$ for $1\leq l\leq n$ and $\sigma:\{1,2,\ldots, n\}\longrightarrow \{1,2,\ldots,n\}$ at the spot $\mathcal{S}_1$ of degree $s\in \mathbb{N}$, then it is also exact in the directions  $[x_{\sigma(1)}],\ldots,[x_{\sigma(l)}],[x_k]$ at the spot $\mathcal{S}_1$ of degree $s+1$. 
 \end{proposition}
 
 \begin{proof}
 By the definition \ref{exact}, we can write 
\begin{align}
  (\gamma^{-1}\circ\beta\circ\gamma\circ\nabla)^s_{[x_k]}(\mathcal{S})=(\gamma^{-1}\circ\beta\circ\gamma\circ\nabla)_{\otimes_{i=1}^{l}[x_{\sigma(i)}]}(\mathcal{S}_1).\nonumber  
\end{align}
The claim follows by applying an extra copy of the expansion operator $(\gamma^{-1}\circ\beta\circ\gamma\circ\nabla)_{[x_k]}$ on both sides of the equation.
 \end{proof}
\bigskip

 Here, we show that the notion of exactness can be extended to proper sub-expansions of an expansion.
 
 \begin{proposition}\label{exactness vs sub-expansion}
 Let $(\gamma^{-1}\circ\beta\circ\gamma\circ\nabla)_{[x_k]}(\mathcal{S})<(\gamma^{-1}\circ\beta\circ\gamma\circ\nabla)_{[x_k]}(\mathcal{S}_a)$ be a sub-expansion of the expansion. If $(\gamma^{-1}\circ\beta\circ\gamma\circ\nabla)_{[x_k]}(\mathcal{S})$ is exact in the directions $[x_{\sigma(1)}],\ldots,[x_{\sigma(l)}]$ for $1\leq l\leq n$ and $\sigma:\{1,2,\ldots, n\}\longrightarrow \{1,2,\ldots,n\}$ at the spot $\mathcal{S}_1$ of degree $s\in \mathbb{N}$, then there exists some positive integer $m\in \mathbb{N}$ such that the expansion $(\gamma^{-1}\circ\beta\circ\gamma\circ\nabla)_{[x_k]}(\mathcal{S}_a)$ is exact of degree $s+m-1$ in the directions $[x_{\sigma(1)}],\ldots,[x_{\sigma(l)}]$ for $1\leq l\leq n$ and $\sigma:\{1,2,\ldots, n\}\longrightarrow \{1,2,\ldots,n\}$ at the spot $\mathcal{S}_1$.
 \end{proposition}
 
 \begin{proof}
 Under condition $(\gamma^{-1}\circ\beta\circ\gamma\circ\nabla)_{[x_k]}(\mathcal{S})<(\gamma^{-1}\circ\beta\circ\gamma\circ\nabla)_{[x_k]}(\mathcal{S}_a)$, there exists some fixed $m\in \mathbb{N}$ such that 
 \begin{align}
     (\gamma^{-1}\circ\beta\circ\gamma\circ\nabla)_{[x_k]}(\mathcal{S})=(\gamma^{-1}\circ\beta\circ\gamma\circ\nabla)^m_{[x_k]}(\mathcal{S}_a)\nonumber 
 \end{align}
 so that by applying $(s-1)$ copies of the expansion operator $(\gamma^{-1}\circ\beta\circ\gamma\circ\nabla)_{[x_k]}$ on both sides of the equation, we have 
 \begin{align}
     (\gamma^{-1}\circ\beta\circ\gamma\circ\nabla)^s_{[x_k]}(\mathcal{S})=(\gamma^{-1}\circ\beta\circ\gamma\circ\nabla)^{s+m-1}_{[x_k]}(\mathcal{S}_a).\nonumber
 \end{align}
 The claim follows because the expansion $(\gamma^{-1}\circ\beta\circ\gamma\circ\nabla)_{[x_k]}(\mathcal{S})$ is exact of degree $s$ in the directions $[x_{\sigma(1)}],\ldots,[x_{\sigma(l)}]$ for $1\leq l\leq n$ at the spot $\mathcal{S}_1$.
 \end{proof}
\bigskip

Although it is fairly easy to pass the notion of exactness of a sub-expansion to an expansion, the converse is actually difficult. However, we can establish this converse under certain conditions on an expansion and their sub-expansion. The follow-up result underscores this discussion.

\begin{proposition}
 Let $(\gamma^{-1}\circ\beta\circ\gamma\circ\nabla)_{[x_k]}(\mathcal{S})<(\gamma^{-1}\circ\beta\circ\gamma\circ\nabla)_{[x_k]}(\mathcal{S}_a)$ be a proper sub-expansion of the expansion. If $(\gamma^{-1}\circ\beta\circ\gamma\circ\nabla)_{[x_k]}(\mathcal{S}_a)$ is exact in the directions $[x_{\sigma(1)}],\ldots,[x_{\sigma(l)}]$ for $1\leq l\leq n$ and $\sigma:\{1,2,\ldots, n\}\longrightarrow \{1,2,\ldots,n\}$ at the spot $\mathcal{S}_1$ of degree $s\in \mathbb{N}$ and $(\gamma^{-1}\circ\beta\circ\gamma\circ\nabla)^s_{[x_k]}(\mathcal{S}_a)<(\gamma^{-1}\circ\beta\circ\gamma\circ\nabla)_{[x_k]}(\mathcal{S})$, then there exists some $j \in \mathbb{N}$ such that the proper sub-expansion $(\gamma^{-1}\circ\beta\circ\gamma\circ\nabla)_{[x_k]}(\mathcal{S})$ is exact of degree $j<s$ in the directions $[x_{\sigma(1)}],\ldots,[x_{\sigma(l)}]$ for $1\leq l\leq n$ and $\sigma:\{1,2,\ldots, n\}\longrightarrow \{1,2,\ldots,n\}$ at the spot $\mathcal{S}_1$.
\end{proposition}

\begin{proof}
Suppose that $(\gamma^{-1}\circ\beta\circ\gamma\circ\nabla)_{[x_k]}(\mathcal{S}_a)$ is exact in the directions $[x_{\sigma(1)}],\ldots,[x_{\sigma(l)}]$ for $1\leq l\leq n$ and $\sigma:\{1,2,\ldots, n\}\longrightarrow \{1,2,\ldots,n\}$ at the spot $\mathcal{S}_1$ of degree $s\in \mathbb{N}$. We deduce
\begin{align}
   (\gamma^{-1}\circ\beta\circ\gamma\circ\nabla)^s_{[x_k]}(\mathcal{S}_a)=(\gamma^{-1}\circ\beta\circ\gamma\circ\nabla)_{\otimes_{i=1}^{l}[x_{\sigma(i)}]}(\mathcal{S}_1)\nonumber  
\end{align}
so that under the requirement $(\gamma^{-1}\circ\beta\circ\gamma\circ\nabla)^s_{[x_k]}(\mathcal{S}_a)<(\gamma^{-1}\circ\beta\circ\gamma\circ\nabla)_{[x_k]}(\mathcal{S})$ there exists some $j\in \mathbb{N}$ such that 
\begin{align}
   (\gamma^{-1}\circ\beta\circ\gamma\circ\nabla)^s_{[x_k]}(\mathcal{S}_a)= (\gamma^{-1}\circ\beta\circ\gamma\circ\nabla)^j_{[x_k]}(\mathcal{S}).\nonumber
\end{align}
Since $(\gamma^{-1}\circ\beta\circ\gamma\circ\nabla)_{[x_k]}(\mathcal{S})<(\gamma^{-1}\circ\beta\circ\gamma\circ \nabla)_{[x_k]}(\mathcal{S}_a)$ is a proper sub-expansion of the expansion, there exists some $m\in \mathbb{N}$ such that
\begin{align}
   (\gamma^{-1}\circ\beta\circ\gamma\circ\nabla)_{[x_k]}(\mathcal{S})=(\gamma^{-1}\circ\beta\circ\gamma\circ\nabla)^m_{[x_k]}(\mathcal{S}_a).\nonumber
\end{align}
Combining both equations, we deduce
\begin{align}
    (\gamma^{-1}\circ\beta\circ\gamma\circ\nabla)^s_{[x_k]}(\mathcal{S}_a)=(\gamma^{-1}\circ\beta\circ\gamma\circ\nabla)^{j-1+m}_{[x_k]}(\mathcal{S}_a)\nonumber
\end{align}
so that $j<s$ and the claim follows from this assertion.
\end{proof}
\bigskip

\section{Sequences of an exact expansion}

In this section, we examine the structure and the commutative diagram of an exact expansion.

\begin{definition}
Let $(\gamma^{-1}\circ\beta\circ\gamma\circ\nabla)_{[x_k]}(\mathcal{S})$ be an exact expansion in the directions $[x_{\sigma(1)}],\ldots,[x_{\sigma(l)}]$ for $1\leq l\leq n$ and $\sigma:\{1,2,\ldots, n\}\longrightarrow \{1,2,\ldots,n\}$ of degree $s$. We call the chain
$$\begin{tikzcd}
(\gamma^{-1}\circ\beta\circ\gamma\circ\nabla)_{[x_{\sigma(1)}]}(\mathcal{S}_1) \arrow[r]
& (\gamma^{-1}\circ\beta\circ\gamma\circ\nabla)_{[x_{\sigma(1)}]\otimes [[x_{\sigma(2)}]}(\mathcal{S}_1) \arrow[r]
\arrow[d, phantom, ""{coordinate, name=Z}]
& \cdots \arrow[dll,
rounded corners,
to path={ -- ([xshift=2ex]\tikztostart.east)
|- (Z) [near end]\tikztonodes
-| ([xshift=-2ex]\tikztotarget.west)
-- (\tikztotarget)}] \\
(\gamma^{-1}\circ\beta\circ\gamma\circ\nabla)_{{ \otimes_{i=1}^{l-1}[x_{\sigma(i)}]}}(\mathcal{S}_1) \arrow[r]
& (\gamma^{-1}\circ\beta\circ\gamma\circ\nabla)_{{ \otimes_{i=1}^{l}[x_{\sigma(i)}]}}(\mathcal{S}_1)
\end{tikzcd}
$$
an exact sequence of the exact expansion $(\gamma^{-1}\circ\beta\circ\gamma\circ\nabla)_{[x_k]}(\mathcal{S})$--respectively, $(\gamma^{-1}\circ\beta\circ\gamma\circ\nabla)_{[x_k]}(\mathcal{S})$ is an exact expansion of the exact sequence--where $$
\phi_l=(\gamma^{-1}\circ\beta\circ\gamma\circ\nabla)_{[x_{\sigma(l)}]}
$$ 
and 
$$
\phi_l\circ (\gamma^{-1}\circ\beta\circ\gamma\circ\nabla)_{[x_k]}(\mathcal{S})=(\gamma^{-1}\circ\beta\circ\gamma\circ\nabla)_{[x_k]\otimes [x_{\sigma(l)}]}(\mathcal{S})
$$ 
for $1\leq l\leq n$.
\end{definition}
\bigskip

\begin{definition}
We say that the exact sequence 
$$
\begin{tikzcd}
(\gamma^{-1}\circ\beta\circ\gamma\circ\nabla)_{[x_{\sigma(1)}]}(\mathcal{S}_1) \arrow[r]
& (\gamma^{-1}\circ\beta\circ\gamma\circ\nabla)_{[x_{\sigma(1)}]\otimes [[x_{\sigma(2)}]}(\mathcal{S}_1) \arrow[r]
\arrow[d, phantom, ""{coordinate, name=Z}]
& \cdots \arrow[dll,
rounded corners,
to path={ -- ([xshift=2ex]\tikztostart.east)
|- (Z) [near end]\tikztonodes
-| ([xshift=-2ex]\tikztotarget.west)
-- (\tikztotarget)}] \\
(\gamma^{-1}\circ\beta\circ\gamma\circ\nabla)_{{ \otimes_{i=1}^{l-1}[x_{\sigma(i)}]}}(\mathcal{S}_1) \arrow[r]
& (\gamma^{-1}\circ\beta\circ\gamma\circ\nabla)_{{ \otimes_{i=1}^{l}[x_{\sigma(i)}]}}(\mathcal{S}_1)
\end{tikzcd}
$$ 
is a sub-sequence of the exact sequence 
$$\begin{tikzcd}
(\gamma^{-1}\circ\beta\circ\gamma\circ\nabla)_{[x_{\lambda(1)}]}(\mathcal{S}_2) \arrow[r]
& (\gamma^{-1}\circ\beta\circ\gamma\circ\nabla)_{[x_{\lambda(1)}]\otimes [[x_{\lambda(2)}]}(\mathcal{S}_2) \arrow[r]
\arrow[d, phantom, ""{coordinate, name=Z}]
& \cdots \arrow[dll,
rounded corners,
to path={ -- ([xshift=2ex]\tikztostart.east)
|- (Z) [near end]\tikztonodes
-| ([xshift=-2ex]\tikztotarget.west)
-- (\tikztotarget)}] \\
(\gamma^{-1}\circ\beta\circ\gamma\circ\nabla)_{{ \otimes_{i=1}^{r-1}[x_{\lambda(i)}]}}(\mathcal{S}_2) \arrow[r]
& (\gamma^{-1}\circ\beta\circ\gamma\circ\nabla)_{{ \otimes_{i=1}^{r}[x_{\lambda(i)}]}}(\mathcal{S}_2)
\end{tikzcd}
$$ 
if the first chain is contained in the second chain.
\end{definition}

\begin{definition}
We say that the expansion $(\gamma^{-1}\circ\beta\circ\gamma\circ\nabla)_{[x_k]}(\mathcal{S}_a)$ is a sub-expansion of the expansion $(\gamma^{-1}\circ\beta\circ\gamma\circ\nabla)_{[x_j]}(\mathcal{S}_b)$ along the directions $[x_{\sigma(1)}],\ldots,[x_{\sigma(l)}]$ each with multiplicity $k_i$ for $1\leq i \leq l\leq n$, where $\sigma:\{1,2,\ldots,n\}\longrightarrow \{1,2,\ldots,n\}$ if and only if
\begin{align}
    (\gamma^{-1}\circ\beta\circ\gamma\circ\nabla)_{[x_k]}(\mathcal{S}_a)=(\gamma^{-1}\circ\beta\circ\gamma\circ\nabla)_{{ \otimes_{i=1}^{r}[x_{\sigma(i)}]^{k_i}}}\circ (\gamma^{-1}\circ\beta\circ\gamma\circ\nabla)_{[x_j]}(\mathcal{S}_b).\nonumber
\end{align}
We denote this sub-expansion by 
\begin{align}
    (\gamma^{-1}\circ\beta\circ\gamma\circ\nabla)_{[x_k]}(\mathcal{S}_a)\leq_{[x_\sigma(1)],\ldots,[x_\sigma(l)]}~ (\gamma^{-1}\circ\beta\circ\gamma\circ\nabla)_{[x_j]}(\mathcal{S}_b).\nonumber
\end{align}
\end{definition}
\bigskip

We may view this as an extended notion of sub-expansions of an expansion. Indeed, the intuition remains that a sub-expansion of an expansion is an outcome of several expansions on the mother expansion. This also provides some flexibility to the manner in which sub-expansions can be obtained from their mother expansion in the framework of an expansion in mixed directions.

\begin{theorem}
If the exact sequence 
$$
\begin{tikzcd}
(\gamma^{-1}\circ\beta\circ\gamma\circ\nabla)_{[x_{\sigma(1)}]}(\mathcal{S}_1) \arrow[r]
& (\gamma^{-1}\circ\beta\circ\gamma\circ\nabla)_{[x_{\sigma(1)}]\otimes [[x_{\sigma(2)}]}(\mathcal{S}_1) \arrow[r]
\arrow[d, phantom, ""{coordinate, name=Z}]
& \cdots \arrow[dll,
rounded corners,
to path={ -- ([xshift=2ex]\tikztostart.east)
|- (Z) [near end]\tikztonodes
-| ([xshift=-2ex]\tikztotarget.west)
-- (\tikztotarget)}] \\
(\gamma^{-1}\circ\beta\circ\gamma\circ\nabla)_{{ \otimes_{i=1}^{l-1}[x_{\sigma(i)}]}}(\mathcal{S}_1) \arrow[r]
& (\gamma^{-1}\circ\beta\circ\gamma\circ\nabla)_{{ \otimes_{i=1}^{l}[x_{\sigma(i)}]}}(\mathcal{S}_1)
\end{tikzcd}
$$ 
of the exact expansion $(\gamma^{-1}\circ\beta\circ\gamma\circ\nabla)_{[x_k]}(\mathcal{S}_a)$ with degree $u$ of exactness is a sub-sequence of the exact sequence 
$$
\begin{tikzcd}
(\gamma^{-1}\circ\beta\circ\gamma\circ\nabla)_{[x_{\sigma(1)}]}(\mathcal{S}_2) \arrow[r]
& (\gamma^{-1}\circ\beta\circ\gamma\circ\nabla)_{[x_{\sigma(1)}]\otimes [[x_{\sigma(2)}]}(\mathcal{S}_2) \arrow[r]
\arrow[d, phantom, ""{coordinate, name=Z}]
& \cdots \arrow[dll,
rounded corners,
to path={ -- ([xshift=2ex]\tikztostart.east)
|- (Z) [near end]\tikztonodes
-| ([xshift=-2ex]\tikztotarget.west)
-- (\tikztotarget)}] \\
(\gamma^{-1}\circ\beta\circ\gamma\circ\nabla)_{{ \otimes_{i=1}^{r-1}[x_{\sigma(i)}]}}(\mathcal{S}_2) \arrow[r]
& (\gamma^{-1}\circ\beta\circ\gamma\circ\nabla)_{{ \otimes_{i=1}^{r}[x_{\sigma(i)}]}}(\mathcal{S}_2)
\end{tikzcd}
$$ 
of the exact expansion  $(\gamma^{-1}\circ\beta\circ\gamma\circ\nabla)_{[x_j]}(\mathcal{S}_b)$ with degree $v$ of exactness then
\begin{align}
    (\gamma^{-1}\circ\beta\circ\gamma circ\nabla)^u_{[x_k]}(\mathcal{S}_a)\leq_{[x_{\sigma(l+1)}],[x_{\sigma(l+2)}],\ldots,[x_{\sigma(r)}],\ldots, [x_{\lambda(s)}]} ~ (\gamma^{-1}\circ\beta\circ\gamma\circ\nabla)^v_{[x_j]}(\mathcal{S}_b).\nonumber
\end{align}
for some $s\in \mathbb{N}$ with $s\leq n$ and where $\lambda:\{1,2,\ldots,n\}\longrightarrow \{1,2,\ldots, n\}$.
\end{theorem}

\begin{proof}
Under the main assumption, we can embed the chain 
$$
\begin{tikzcd}
(\gamma^{-1}\circ\beta\circ\gamma\circ\nabla)_{[x_{\sigma(1)}]}(\mathcal{S}_1) \arrow[r]
& (\gamma^{-1}\circ\beta\circ\gamma\circ\nabla)_{[x_{\sigma(1)}]\otimes [[x_{\sigma(2)}]}(\mathcal{S}_1) \arrow[r]
\arrow[d, phantom, ""{coordinate, name=Z}]
& \cdots \arrow[dll,
rounded corners,
to path={ -- ([xshift=2ex]\tikztostart.east)
|- (Z) [near end]\tikztonodes
-| ([xshift=-2ex]\tikztotarget.west)
-- (\tikztotarget)}] \\
(\gamma^{-1}\circ\beta\circ\gamma\circ\nabla)_{{ \otimes_{i=1}^{l-1}[x_{\sigma(i)}]}}(\mathcal{S}_1) \arrow[r]
& (\gamma^{-1}\circ\beta\circ\gamma\circ\nabla)_{{ \otimes_{i=1}^{l}[x_{\sigma(i)}]}}(\mathcal{S}_1)
\end{tikzcd}
$$ 
in the chain 
$$
\begin{tikzcd}
(\gamma^{-1}\circ\beta\circ\gamma\circ\nabla)_{[x_{\sigma(1)}]}(\mathcal{S}_2) \arrow[r]
& (\gamma^{-1}\circ\beta\circ\gamma\circ\nabla)_{[x_{\sigma(1)}]\otimes [[x_{\sigma(2)}]}(\mathcal{S}_2) \arrow[r]
\arrow[d, phantom, ""{coordinate, name=Z}]
& \cdots \arrow[dll,
rounded corners,
to path={ -- ([xshift=2ex]\tikztostart.east)
|- (Z) [near end]\tikztonodes
-| ([xshift=-2ex]\tikztotarget.west)
-- (\tikztotarget)}] \\
(\gamma^{-1}\circ\beta\circ\gamma\circ\nabla)_{{ \otimes_{i=1}^{r-1}[x_{\sigma(i)}]}}(\mathcal{S}_2) \arrow[r]
& (\gamma^{-1}\circ\beta\circ\gamma\circ\nabla)_{{ \otimes_{i=1}^{r}[x_{\sigma(i)}]}}(\mathcal{S}_2)
\end{tikzcd}
$$ 
so that by the commutative property of an expansion, we can write 
\begin{align}
    (\gamma^{-1}\circ\beta\circ\gamma\circ\nabla)_{{ \otimes_{i=1}^{r}[x_{\sigma(i)}]}}(\mathcal{S}_2)&=(\gamma^{-1}\circ\beta\circ\gamma\circ\nabla)_{{ \otimes_{j=1}^{s}[x_{\lambda(j)}]}} \circ (\gamma^{-1}\circ\beta\circ\gamma\circ\nabla)_{{ \otimes_{i=1}^{l}[x_{\sigma(i)}]}}\nonumber \\& \circ (\gamma^{-1}\circ\beta\circ\gamma\circ\nabla)_{{ \otimes_{i=l+1}^{r}[x_{\sigma(i)}]}}(\mathcal{S}_1)\nonumber \\&=(\gamma^{-1}\circ\beta\circ\gamma\circ\nabla)_{{ \otimes_{i=l+1}^{r}[x_{\sigma(i)}]}}\circ (\gamma^{-1}\circ\beta\circ\gamma\circ\nabla)_{{ \otimes_{j=1}^{s}[x_{\lambda(j)}]}}\nonumber \\& \circ (\gamma^{-1}\circ\beta\circ\gamma\circ\nabla)_{{ \otimes_{i=1}^{l}[x_{\sigma(i)}]}}(\mathcal{S}_1)\nonumber \\&=(\gamma^{-1}\circ\beta\circ\gamma\circ\nabla)_{{ \otimes_{i=l+1}^{r}[x_{\sigma(i)}]}}\circ (\gamma^{-1}\circ\beta\circ\gamma\circ\nabla)_{{ \otimes_{j=1}^{s}[x_{\lambda(j)}]}}\nonumber \\& \circ  (\gamma^{-1}\circ\beta\circ\gamma\circ\nabla)^u_{[x_k]}(\mathcal{S}_a)\nonumber
\end{align}
since $(\gamma^{-1}\circ\beta\circ\gamma\circ \nabla)_{{ \otimes_{i=1}^{l}[x_{\sigma(i)}]}}(\mathcal{S}_1)= (\gamma^{-1}\circ\beta\circ\gamma\circ\nabla)^u_{[x_k]}(\mathcal{S}_a)$. Under the exactness condition 
$$
(\gamma^{-1}\circ \beta \circ \gamma \circ \nabla)^v_{[x_j]}(\mathcal{S}_b)=(\gamma^{-1}\circ \beta \circ \gamma \circ\nabla)_{{ \otimes_{i=1}^{r}[x_{\sigma(i)}]}}(\mathcal{S}_2)
$$
we obtain 
\begin{align}
    (\gamma^{-1}\circ\beta\circ\gamma\circ\nabla)^v_{[x_j]}(\mathcal{S}_b)&=(\gamma^{-1}\circ\beta\circ\gamma\circ\nabla)_{{ \otimes_{i=l+1}^{r}[x_{\sigma(i)}]}}\circ (\gamma^{-1}\circ\beta\circ\gamma\circ\nabla)_{{ \otimes_{j=1}^{s}[x_{\lambda(j)}]}}\nonumber \\& \circ  (\gamma^{-1}\circ\beta\circ\gamma\circ\nabla)^u_{[x_k]}(\mathcal{S}_a).\nonumber
\end{align}
We deduce 
\begin{align}
    (\gamma^{-1}\circ\beta\circ\gamma\circ\nabla)^u_{[x_k]}(\mathcal{S}_a)\leq_{[x_{\sigma(l+1)}],[x_{\sigma(l+2)}],\ldots,[x_{\sigma(r)}],\ldots, [x_{\lambda(s)}]} ~ (\gamma^{-1}\circ\beta\circ\gamma\circ\nabla)^v_{[x_j]}(\mathcal{S}_b).\nonumber
\end{align}
\end{proof}
\bigskip

\section{Sub-expansion and indices of expansion}
In this section, we introduce the notion of the \emph{index} of sub-expansions of an expansion.

\begin{definition}
Let $(\gamma^{-1}\circ\beta\circ\gamma\circ\nabla)_{[x_{j}]}(\mathcal{S}_z)$ and $(\gamma^{-1}\circ\beta\circ\gamma\circ\nabla)_{[x_{j}]}(\mathcal{S}_t)$ be expansions. By the \emph{index} of expansion $(\gamma^{-1}\circ\beta\circ\gamma\circ\nabla)_{[x_{j}]}(\mathcal{S}_z)$ in the expansion $(\gamma^{-1}\circ\beta\circ\gamma\circ\nabla)_{[x_{j}]}(\mathcal{S}_t)$, denoted by 
$$
\bigg[(\gamma^{-1}\circ\beta\circ\gamma\circ\nabla)_{[x_{j}]}(\mathcal{S}_t):(\gamma^{-1}\circ\beta\circ\gamma\circ\nabla)_{[x_{j}]}(\mathcal{S}_z)\bigg]
$$ 
we mean the value of $r\in \mathbb{N}$ such that 
\begin{align}
    (\gamma^{-1}\circ\beta\circ\gamma\circ\nabla)_{[x_{j}]}(\mathcal{S}_z)=(\gamma^{-1}\circ\beta\circ\gamma\circ\nabla)^r_{[x_{j}]}(\mathcal{S}_t).\nonumber
\end{align}
We write
\begin{align}
    \bigg[(\gamma^{-1}\circ\beta\circ\gamma\circ\nabla)_{[x_{j}]}(\mathcal{S}_t):(\gamma^{-1}\circ\beta\circ\gamma\circ\nabla)_{[x_{j}]}(\mathcal{S}_z)\bigg]=r.\nonumber
\end{align}
to denote the index. We say that the index is finite if and only if it exists and we write 
\begin{align}
    \bigg[(\gamma^{-1}\circ\beta\circ\gamma\circ\nabla)_{[x_{j}]}(\mathcal{S}_t):(\gamma^{-1}\circ\beta\circ\gamma\circ\nabla)_{[x_{j}]}(\mathcal{S}_z)\bigg]<\infty.\nonumber
\end{align}
On the other hand, if no such value exists, then we say that the index is infinite and we write 
\begin{align}
    \bigg[(\gamma^{-1}\circ\beta\circ\gamma\circ\nabla)_{[x_{j}]}(\mathcal{S}_t):(\gamma^{-1}\circ\beta\circ\gamma\circ\nabla)_{[x_{j}]}(\mathcal{S}_z)\bigg]=\infty.\nonumber
\end{align}
\end{definition}
\bigskip

\begin{proposition}\label{finite index-sub-expansion}
Let $(\gamma^{-1}\circ\beta\circ\gamma\circ\nabla)_{[x_{j}]}(\mathcal{S}_z)$ and $(\gamma^{-1}\circ\beta\circ\gamma\circ\nabla)_{[x_{j}]}(\mathcal{S}_t)$ be expansions. 
We have
\begin{align}
     \bigg[(\gamma^{-1}\circ\beta\circ\gamma\circ \nabla)_{[x_{j}]}(\mathcal{S}_t):(\gamma^{-1}\circ\beta\circ\gamma\circ\nabla)_{[x_{j}]}(\mathcal{S}_z)\bigg]<\infty\nonumber
\end{align}
if and only if $(\gamma^{-1}\circ\beta\circ\gamma\circ\nabla)_{[x_{j}]}(\mathcal{S}_z) \leq (\gamma^{-1}\circ\beta\circ\gamma\circ\nabla)_{[x_{j}]}(\mathcal{S}_t)$.
\end{proposition}

\begin{proof}
This is a simple consequence of the notion sub-expansions of an expansion and the index of an expansion. 
\end{proof}
\bigskip

\begin{proposition}\label{index transitivity}
Let $(\gamma^{-1}\circ\beta\circ\gamma\circ\nabla)_{[x_{j}]}(\mathcal{S}_1)$, $(\gamma^{-1}\circ\beta\circ\gamma\circ\nabla)_{[x_{j}]}(\mathcal{S}_2)$ and $(\gamma^{-1}\circ\beta\circ\gamma\circ\nabla)_{[x_{j}]}(\mathcal{S}_3)$ be expansions. If $\bigg[(\gamma^{-1}\circ\beta\circ\gamma\circ\nabla)_{[x_{j}]}(\mathcal{S}_3):(\gamma^{-1}\circ\beta\circ\gamma\circ\nabla)_{[x_{j}]}(\mathcal{S}_2)\bigg]<\infty$ and $\bigg[(\gamma^{-1}\circ\beta\circ\gamma\circ\nabla)_{[x_{j}]}(\mathcal{S}_2):(\gamma^{-1}\circ\beta\circ\gamma\circ\nabla)_{[x_{j}]}(\mathcal{S}_1)\bigg]<\infty$ then 
\begin{align}
    \bigg[(\gamma^{-1}\circ\beta\circ\gamma\circ\nabla)_{[x_{j}]}(\mathcal{S}_3):(\gamma^{-1}\circ\beta\circ\gamma\circ\nabla)_{[x_{j}]}(\mathcal{S}_1)\bigg]<\infty.\nonumber
\end{align}
\end{proposition}

\begin{proof}
Suppose that $\bigg[(\gamma^{-1}\circ\beta\circ\gamma\circ\nabla)_{[x_{j}]}(\mathcal{S}_3):(\gamma^{-1}\circ\beta\circ\gamma\circ\nabla)_{[x_{j}]}(\mathcal{S}_2)\bigg]<\infty$ and $\bigg[(\gamma^{-1}\circ\beta\circ\gamma\circ\nabla)_{[x_{j}]}(\mathcal{S}_2):(\gamma^{-1}\circ\beta\circ\gamma\circ\nabla)_{[x_{j}]}(\mathcal{S}_1)\bigg]<\infty$. There exist some $r,s\in \mathbb{N}$ such that
\begin{align}
   (\gamma^{-1}\circ\beta\circ\gamma\circ\nabla)_{[x_{j}]}(\mathcal{S}_2)=(\gamma^{-1}\circ\beta\circ\gamma\circ\nabla)^r_{[x_{j}]}(\mathcal{S}_3)\nonumber
\end{align}
and 
\begin{align}
    (\gamma^{-1}\circ\beta\circ\gamma\circ\nabla)_{[x_{j}]}(\mathcal{S}_1)=(\gamma^{-1}\circ\beta\circ\gamma\circ\nabla)^s_{[x_{j}]}(\mathcal{S}_2).\nonumber
\end{align}
It follows that
\begin{align}
    (\gamma^{-1}\circ\beta\circ\gamma\circ\nabla)_{[x_{j}]}(\mathcal{S}_1)&=(\gamma^{-1}\circ\beta\circ\gamma\circ\nabla)^s_{[x_{j}]}(\mathcal{S}_2\nonumber \\&=(\gamma^{-1}\circ\beta\circ\gamma\circ\nabla)^{r+s-1}_{[x_{j}]}(\mathcal{S}_3)\nonumber
\end{align}
so that $\bigg[(\gamma^{-1}\circ\beta\circ\gamma\circ \nabla)_{[x_{j}]}(\mathcal{S}_3):(\gamma^{-1}\circ\beta\circ\gamma\circ\nabla)_{[x_{j}]}(\mathcal{S}_1)\bigg]<\infty.$
\end{proof}
\bigskip

Here, we show that the index of a sub-expansion in an expansion decreases with further expansions.

\begin{proposition}
Let $(\gamma^{-1}\circ\beta\circ\gamma\circ\nabla)_{[x_{j}]}(\mathcal{S}_1)\leq  (\gamma^{-1}\circ\beta\circ\gamma\circ\nabla)_{[x_{j}]}(\mathcal{S}_2)$. If there exists some $l\in \mathbb{N}$ such that $(\gamma^{-1}\circ\beta\circ\gamma\circ\nabla)_{[x_{j}]}(\mathcal{S}_1)\leq  (\gamma^{-1}\circ\beta\circ\gamma\circ\nabla)^l_{[x_{j}]}(\mathcal{S}_2)$, then 
\begin{align}
    \bigg[(\gamma^{-1}\circ\beta\circ\gamma\circ\nabla)^l_{[x_{j}]}(\mathcal{S}_2):(\gamma^{-1}\circ\beta\circ\gamma\circ\nabla)_{[x_{j}]}(\mathcal{S}_1)\bigg]&<\bigg[(\gamma^{-1}\circ\beta\circ\gamma\circ\nabla)_{[x_{j}]}(\mathcal{S}_2)\nonumber \\&:(\gamma^{-1}\circ\beta\circ\gamma\circ\nabla)_{[x_{j}]}(\mathcal{S}_1)\bigg].\nonumber
\end{align}
\end{proposition}

\begin{proof}
Suppose that $(\gamma^{-1}\circ\beta\circ\gamma\circ\nabla)_{[x_{j}]}(\mathcal{S}_1)\leq  (\gamma^{-1}\circ\beta\circ\gamma\circ\nabla)_{[x_{j}]}(\mathcal{S}_2)$. There exists some $s\in \mathbb{N}$ such that 
\begin{align}
    (\gamma^{-1}\circ\beta\circ\gamma\circ\nabla)_{[x_{j}]}(\mathcal{S}_1)=(\gamma^{-1}\circ\beta\circ\gamma\circ\nabla)^s_{[x_{j}]}(\mathcal{S}_2).\nonumber
\end{align}
Under the regularity condition $(\gamma^{-1}\circ\beta\circ\gamma\circ\nabla)_{[x_{j}]}(\mathcal{S}_1)\leq  (\gamma^{-1}\circ\beta\circ\gamma\circ\nabla)^l_{[x_{j}]}(\mathcal{S}_2)$ there exists some $u\in \mathbb{N}$ such that 
\begin{align}
   (\gamma^{-1}\circ\beta\circ\gamma\circ\nabla)_{[x_{j}]}(\mathcal{S}_1)=(\gamma^{-1}\circ\beta\circ\gamma\circ\nabla)^{l+u}_{[x_{j}]}(\mathcal{S}_2)\nonumber 
\end{align}
so that 
\begin{align}
   (\gamma^{-1}\circ\beta\circ\gamma\circ\nabla)^s_{[x_{j}]}(\mathcal{S}_2)=(\gamma^{-1}\circ\beta\circ\gamma\circ\nabla)^{l+u}_{[x_{j}]}(\mathcal{S}_2)\nonumber 
\end{align}
and $u<u+l=s$. We deduce the claimed inequality by making the substitutions $\bigg[(\gamma^{-1}\circ\beta\circ\gamma\circ\nabla)^l_{[x_{j}]}(\mathcal{S}_2):(\gamma^{-1}\circ\beta\circ\gamma\circ\nabla)_{[x_{j}]}(\mathcal{S}_1)\bigg]=u$ and $\bigg[(\gamma^{-1}\circ\beta\circ\gamma\circ\nabla)_{[x_{j}]}(\mathcal{S}_2):(\gamma^{-1}\circ\beta\circ\gamma\circ\nabla)_{[x_{j}]}(\mathcal{S}_1)\bigg]=s$.
\end{proof}
\bigskip

Here, we relate the index of the smallest sub-expansion in a collection of chains of sub-expansion in their mother expansion to the index of other sub-expansions in other sub-expansion.

\begin{theorem}\label{identity}
Let $(\gamma^{-1}\circ\beta\circ\gamma\circ\nabla)_{[x_{j}]}(\mathcal{S}_1)\leq  (\gamma^{-1}\circ\beta\circ\gamma\circ\nabla)_{[x_{j}]}(\mathcal{S}_2)\leq \cdots \leq  (\gamma^{-1}\circ\beta\circ\gamma\circ\nabla)_{[x_{j}]}(\mathcal{S}_n)$ be a chain of sub-expansions of the expansion $(\gamma^{-1}\circ\beta\circ\gamma\circ \nabla)_{[x_{j}]}(\mathcal{S}_n)$. We have
\begin{align}
   \bigg[(\gamma^{-1}\circ\beta\circ\gamma\circ\nabla)_{[x_{j}]}(\mathcal{S}_n):(\gamma^{-1}\circ\beta\circ\gamma\circ\nabla)_{[x_{j}]}(\mathcal{S}_1)\bigg]&=\sum \limits_{i=1}^{n-1}\bigg[(\gamma^{-1}\circ\beta\circ\gamma\circ\nabla)_{[x_{j}]}(\mathcal{S}_{i+1})\nonumber \\&:(\gamma^{-1}\circ\beta\circ\gamma\circ\nabla)_{[x_{j}]}(\mathcal{S}_i)\bigg]-(n-2).\nonumber 
\end{align}
\end{theorem}

\begin{proof}
By the proposition \ref{finite index-sub-expansion}, we write $\bigg[(\gamma^{-1}\circ\beta\circ\gamma\circ\nabla)_{[x_{j}]}(\mathcal{S}_{i+1}):(\gamma^{-1}\circ\beta\circ\gamma\circ\nabla)_{[x_{j}]}(\mathcal{S}_i)\bigg]<\infty$ for all $1\leq i\leq n-1$.  There exist some $r_1\in \mathbb{N}$ such that
\begin{align}
   (\gamma^{-1}\circ\beta\circ\gamma\circ\nabla)_{[x_{j}]}(\mathcal{S}_{n-1})=(\gamma^{-1}\circ\beta\circ\gamma\circ\nabla)^{r_1}_{[x_{j}]}(\mathcal{S}_n).\nonumber 
\end{align}
Again there exists some $r_2\in \mathbb{N}$ such that 
\begin{align}
    (\gamma^{-1}\circ\beta\circ\gamma\circ\nabla)_{[x_{j}]}(\mathcal{S}_{n-2})=(\gamma^{-1}\circ\beta\circ\gamma\circ\nabla)^{r_2}_{[x_{j}]}(\mathcal{S}_{n-1})\nonumber
\end{align}
so that 
\begin{align}
   (\gamma^{-1}\circ\beta\circ\gamma\circ\nabla)_{[x_{j}]}(\mathcal{S}_{n-2})&= (\gamma^{-1}\circ\beta\circ\gamma\circ\nabla)^{r_2}_{[x_{j}]}(\mathcal{S}_{n-1})\nonumber \\&=(\gamma^{-1}\circ\beta\circ\gamma\circ\nabla)^{r_1+r_2-1}_{[x_{j}]}(\mathcal{S}_n).\nonumber 
\end{align}
Similarly, there exists some $r_3\in \mathbb{N}$ such that 
\begin{align}
   (\gamma^{-1}\circ\beta\circ\gamma\circ\nabla)_{[x_{j}]}(\mathcal{S}_{n-3})=(\gamma^{-1}\circ\beta\circ\gamma\circ\nabla)^{r_3}_{[x_{j}]}(\mathcal{S}_{n-2})\nonumber 
\end{align}
so that 
\begin{align}
    (\gamma^{-1}\circ\beta\circ\gamma\circ\nabla)_{[x_{j}]}(\mathcal{S}_{n-3})&=(\gamma^{-1}\circ\beta\circ\gamma\circ\nabla)^{r_3}_{[x_{j}]}(\mathcal{S}_{n-2})\nonumber \\&=(\gamma^{-1}\circ\beta\circ\gamma\circ\nabla)^{r_1+r_2+r_3-2}_{[x_{j}]}(\mathcal{S}_n).\nonumber  
\end{align}
Repeating this argument and using 
$$
\bigg[(\gamma^{-1}\circ\beta\circ\gamma\circ\nabla)_{[x_{j}]}(\mathcal{S}_{i+1}):(\gamma^{-1}\circ\beta\circ\gamma\circ\nabla)_{[x_{j}]}(\mathcal{S}_i)\bigg]<\infty
$$ 
for all $1\leq i\leq n-1$, we obtain 
\begin{align}
   (\gamma^{-1}\circ\beta\circ\gamma\circ\nabla)_{[x_{j}]}(\mathcal{S}_{1})&=(\gamma^{-1}\circ\beta\circ\gamma\circ\nabla)^{r_1+r_2+r_3+\cdots+r_{n-1}-(n-2)}_{[x_{j}]}(\mathcal{S}_n).\nonumber  
\end{align}
We deduce 
\begin{align}
    \bigg[(\gamma^{-1}\circ\beta\circ\gamma\circ\nabla)_{[x_{j}]}(\mathcal{S}_n):(\gamma^{-1}\circ\beta\circ\gamma\circ\nabla)_{[x_{j}]}(\mathcal{S}_1)\bigg]&=\sum \limits_{i=1}^{n-1}r_{n-i}-(n-2).\nonumber 
\end{align}
The claimed assertion follows because 
$$
\bigg[(\gamma^{-1}\circ\beta\circ\gamma\circ\nabla)_{[x_{j}]}(\mathcal{S}_{i+1}):(\gamma^{-1}\circ\beta\circ\gamma\circ\nabla)_{[x_{j}]}(\mathcal{S}_i)\bigg]=r_{n-i}
$$ 
for $1\leq i\leq n-1$.
\end{proof}
\bigskip

We now obtain an important inequality as a consequence of Theorem \ref{identity} that relates the index of the smallest sub-expansion in their mother expansion to local indices in each sub-expansion of the sub-expansions in the chain. 

\begin{corollary}[The index inequality]
Let $(\gamma^{-1}\circ\beta\circ\gamma\circ\nabla)_{[x_{j}]}(\mathcal{S}_1)\leq  (\gamma^{-1}\circ\beta\circ\gamma\circ\nabla)_{[x_{j}]}(\mathcal{S}_2)\leq \cdots \leq  (\gamma^{-1}\circ\beta\circ\gamma\circ\nabla)_{[x_{j}]}(\mathcal{S}_n)$--a chain of sub-expansions of the expansion $(\gamma^{-1}\circ\beta\circ\gamma\circ\nabla)_{[x_{j}]}(\mathcal{S}_n)$. We have
\begin{align}
   \bigg[(\gamma^{-1}\circ\beta\circ\gamma\circ\nabla)_{[x_{j}]}(\mathcal{S}_n):(\gamma^{-1}\circ\beta\circ\gamma\circ\nabla)_{[x_{j}]}(\mathcal{S}_1)\bigg]&<\sum \limits_{i=1}^{n-1}\bigg[(\gamma^{-1}\circ\beta\circ\gamma\circ\nabla)_{[x_{j}]}(\mathcal{S}_{i+1})\nonumber \\&:(\gamma^{-1}\circ\beta\circ\gamma\circ\nabla)_{[x_{j}]}(\mathcal{S}_i)\bigg].\nonumber 
\end{align}
\end{corollary}

\begin{theorem}
Let $(\gamma^{-1}\circ\beta\circ\gamma\circ\nabla)_{[x_{j}]}(\mathcal{S}_1)\leq  (\gamma^{-1}\circ\beta\circ\gamma\circ\nabla)_{[x_{j}]}(\mathcal{S}_2)$--a sub-expansion of the expansion.  If there exists some $s\in \mathbb{N}$ such that $ (\gamma^{-1}\circ\beta\circ\gamma\circ\nabla)^s_{[x_{j}]}(\mathcal{S}_2)\leq (\gamma^{-1}\circ\beta\circ\gamma\circ\nabla)_{[x_{j}]}(\mathcal{S}_1)$, then
\begin{align}
s+1&=\bigg[(\gamma^{-1}\circ\beta\circ\gamma\circ\nabla)_{[x_{j}]}(\mathcal{S}_{2}):(\gamma^{-1}\circ\beta\circ\gamma\circ\nabla)_{[x_{j}]}(\mathcal{S}_1)\bigg]\nonumber \\&+\bigg[(\gamma^{-1}\circ\beta\circ\gamma\circ\nabla)_{[x_{j}]}(\mathcal{S}_{1}):(\gamma^{-1}\circ\beta\circ\gamma\circ\nabla)^s_{[x_{j}]}(\mathcal{S}_2)\bigg].\nonumber
\end{align}
\end{theorem}

\begin{proof}
Under the condition $(\gamma^{-1}\circ\beta\circ\gamma\circ\nabla)_{[x_{j}]}(\mathcal{S}_1)\leq  (\gamma^{-1}\circ\beta\circ\gamma\circ\nabla)_{[x_{j}]}(\mathcal{S}_2)$, there exists some $l\in \mathbb{N}$ such that 
\begin{align}
   (\gamma^{-1}\circ\beta\circ\gamma\circ\nabla)_{[x_{j}]}(\mathcal{S}_1)=(\gamma^{-1}\circ\beta\circ\gamma\circ\nabla)^l_{[x_{j}]}(\mathcal{S}_2) \nonumber
\end{align}
so that $\bigg[(\gamma^{-1}\circ\beta\circ\gamma\circ\nabla)_{[x_{j}]}(\mathcal{S}_{2}):(\gamma^{-1}\circ\beta\circ\gamma\circ\nabla)_{[x_{j}]}(\mathcal{S}_1)\bigg]=l$. 
Again 
$$
(\gamma^{-1}\circ\beta\circ\gamma\circ\nabla)^s_{[x_{j}]}(\mathcal{S}_2)\leq (\gamma^{-1}\circ\beta\circ\gamma\circ\nabla)_{[x_{j}]}(\mathcal{S}_1)
$$ 
for some $s\in \mathbb{N}$ implies that there exist some $r\in \mathbb{N}$ such that
\begin{align}
   (\gamma^{-1}\circ\beta\circ\gamma\circ\nabla)^s_{[x_{j}]}(\mathcal{S}_2)=(\gamma^{-1}\circ\beta\circ\gamma\circ\nabla)^r_{[x_{j}]}(\mathcal{S}_1)\nonumber 
\end{align}
so that $\bigg[(\gamma^{-1}\circ\beta\circ\gamma\circ\nabla)_{[x_{j}]}(\mathcal{S}_{1}):(\gamma^{-1}\circ\beta\circ\gamma\circ\nabla)^s_{[x_{j}]}(\mathcal{S}_2)\bigg]=r$. We deduce 
\begin{align}
   (\gamma^{-1}\circ\beta\circ\gamma\circ\nabla)^s_{[x_{j}]}(\mathcal{S}_2)&=(\gamma^{-1}\circ\beta\circ\gamma\circ\nabla)^r_{[x_{j}]}(\mathcal{S}_1)\nonumber \\&=(\gamma^{-1}\circ\beta\circ\gamma\circ\nabla)^{r+l-1}_{[x_{j}]}(\mathcal{S}_2)\nonumber 
\end{align}
and write $s+1=r+l$. The claimed assertion is deduced by making the substitutions 
$$
\bigg[(\gamma^{-1}\circ\beta\circ\gamma\circ \nabla)_{[x_{j}]}(\mathcal{S}_{2}):(\gamma^{-1}\circ\beta\circ\gamma\circ\nabla)_{[x_{j}]}(\mathcal{S}_1)\bigg]=l
$$ 
and 
$$
\bigg[(\gamma^{-1}\circ\beta\circ\gamma\circ\nabla)_{[x_{j}]}(\mathcal{S}_{1}):(\gamma^{-1}\circ\beta\circ\gamma\circ\nabla)^s_{[x_{j}]}(\mathcal{S}_2)\bigg]=r.
$$
\end{proof}

\subsection{An applications to additive number theory}

We apply the framework to the theory of partitions in additive number theory.

\begin{corollary}\label{application}
Let  $(\gamma^{-1}\circ\beta\circ\gamma\circ\nabla)_{[x_{j}]}(\mathcal{S}_1)\leq  (\gamma^{-1}\circ\beta\circ\gamma\circ\nabla)_{[x_{j}]}(\mathcal{S}_2)$ be such that $(\gamma^{-1}\circ\beta\circ\gamma\circ\nabla)^s_{[x_{j}]}(\mathcal{S}_2)\leq (\gamma^{-1}\circ\beta\circ\gamma\circ\nabla)_{[x_{j}]}(\mathcal{S}_1)$. If $\bigg[(\gamma^{-1}\circ\beta\circ\gamma\circ\nabla)_{[x_{j}]}(\mathcal{S}_{2}):(\gamma^{-1}\circ\beta\circ\gamma\circ\nabla)_{[x_{j}]}(\mathcal{S}_1)\bigg]$ and $\bigg[(\gamma^{-1}\circ\beta\circ\gamma\circ\nabla)_{[x_{j}]}(\mathcal{S}_{1}):(\gamma^{-1}\circ\beta\circ\gamma\circ\nabla)^s_{[x_{j}]}(\mathcal{S}_2)\bigg]$ are both prime numbers, then $s+1$ can be written as sum of two prime numbers.
\end{corollary}
\bigskip

\section{Dominating expansions}

In this section, we introduce the notion of \emph{dominating} expansions and their corresponding numbers.

\begin{definition}
Let $(\gamma^{-1}\circ\beta\circ\gamma\circ\nabla)_{[x_{j}]}(\mathcal{S}_z)\leq (\gamma^{-1}\circ\beta\circ\gamma\circ\nabla)_{[x_{j}]}(\mathcal{S}_t)$. We call the smallest number $s\in \mathbb{N}$ such that $(\gamma^{-1}\circ\beta\circ\gamma\circ\nabla)^s_{[x_{j}]}(\mathcal{S}_t)\leq (\gamma^{-1}\circ\beta\circ\gamma\circ\nabla)_{[x_{j}]}(\mathcal{S}_z)$ the \emph{dominating number} of the expansion $(\gamma^{-1}\circ\beta\circ\gamma\circ\nabla)_{[x_{j}]}(\mathcal{S}_z)$ relative to the expansion $(\gamma^{-1}\circ\beta\circ\gamma\circ\nabla)_{[x_{j}]}(\mathcal{S}_t)$. We say that the expansion $(\gamma^{-1}\circ\beta\circ\gamma\circ\nabla)_{[x_{j}]}(\mathcal{S}_z)$ dominates expansion $(\gamma^{-1}\circ\beta\circ\gamma\circ\nabla)_{[x_{j}]}(\mathcal{S}_t)$ with \emph{dominating number} $s$. We denote the dominating number by
\begin{align}
    \mathbb{D}[(\gamma^{-1}\circ\beta\circ\gamma\circ\nabla)_{[x_{j}]}(\mathcal{S}_z)~|~(\gamma^{-1}\circ\beta\circ\gamma\circ\nabla)_{[x_{j}]}(\mathcal{S}_t)]=s.\nonumber
\end{align}
\end{definition}
\bigskip

Here, we relate the notion of the dominating expansion to the notion of sub-expansion of an expansion.

\begin{proposition}
Let $(\gamma^{-1}\circ\beta\circ\gamma\circ\nabla)_{[x_{j}]}(\mathcal{S}_1)\leq (\gamma^{-1}\circ\beta\circ\gamma\circ\nabla)_{[x_{j}]}(\mathcal{S}_n)$ and $(\gamma^{-1}\circ\beta\circ\gamma\circ\nabla)_{[x_{j}]}(\mathcal{S}_2)\leq (\gamma^{-1}\circ\beta\circ\gamma\circ\nabla)_{[x_{j}]}(\mathcal{S}_n)$. We have
\begin{align}
 \mathbb{D}[(\gamma^{-1}\circ\beta\circ\gamma\circ\nabla)_{[x_{j}]}(\mathcal{S}_1)~|~(\gamma^{-1}\circ\beta\circ\gamma\circ\nabla)_{[x_{j}]}(\mathcal{S}_n)]&<\mathbb{D}[(\gamma^{-1}\circ\beta\circ\gamma\circ\nabla)_{[x_{j}]}(\mathcal{S}_2)~|\nonumber \\&~(\gamma^{-1}\circ\beta\circ\gamma\circ\nabla)_{[x_{j}]}(\mathcal{S}_n)]\nonumber   
\end{align}
if and only if $(\gamma^{-1}\circ\beta\circ\gamma\circ\nabla)_{[x_{j}]}(\mathcal{S}_2)\leq (\gamma^{-1}\circ\beta\circ\gamma\circ\nabla)_{[x_{j}]}(\mathcal{S}_1)$.
\end{proposition}

\begin{proof}
Suppose that 
\begin{align}
  \mathbb{D}[(\gamma^{-1}\circ\beta\circ\gamma\circ\nabla)_{[x_{j}]}(\mathcal{S}_1)~|~(\gamma^{-1}\circ\beta\circ\gamma\circ\nabla)_{[x_{j}]}(\mathcal{S}_n)]&<\mathbb{D}[(\gamma^{-1}\circ\beta\circ\gamma\circ\nabla)_{[x_{j}]}(\mathcal{S}_2)~|\nonumber \\&~(\gamma^{-1}\circ\beta\circ\gamma\circ\nabla)_{[x_{j}]}(\mathcal{S}_n)]\nonumber   
\end{align}
and let $\mathbb{D}[(\gamma^{-1}\circ\beta\circ\gamma\circ\nabla)_{[x_{j}]}(\mathcal{S}_1)~|~(\gamma^{-1}\circ\beta\circ\gamma\circ\nabla)_{[x_{j}]}(\mathcal{S}_n)]=u$ and $\mathbb{D}[(\gamma^{-1}\circ\beta\circ\gamma\circ\nabla)_{[x_{j}]}(\mathcal{S}_2)~|~(\gamma^{-1}\circ\beta\circ\gamma\circ\nabla)_{[x_{j}]}(\mathcal{S}_n)]=v$. We have $(\gamma^{-1}\circ\beta\circ\gamma\circ\nabla)^v_{[x_{j}]}(\mathcal{S}_n)\leq (\gamma^{-1}\circ\beta\circ\gamma\circ\nabla)_{[x_{j}]}(\mathcal{S}_2)$ and $(\gamma^{-1}\circ\beta\circ\gamma\circ\nabla)^u_{[x_{j}]}(\mathcal{S}_n)\leq (\gamma^{-1}\circ\beta\circ\gamma\circ\nabla)_{[x_{j}]}(\mathcal{S}_1)$. Since $v$ and $u$ are the smallest such numbers and $v>u$, we obtain the chain of sub-expansions
\begin{align}
  (\gamma^{-1}\circ\beta\circ\gamma\circ\nabla)^v_{[x_{j}]}(\mathcal{S}_n)&\leq (\gamma^{-1}\circ\beta\circ\gamma\circ\nabla)_{[x_{j}]}(\mathcal{S}_2)\leq (\gamma^{-1}\circ\beta\circ\gamma\circ\nabla)^u_{[x_{j}]}(\mathcal{S}_n)\nonumber \\& \leq (\gamma^{-1}\circ\beta\circ\gamma\circ\nabla)_{[x_{j}]}(\mathcal{S}_1)\nonumber  
\end{align}
so that 
$$
(\gamma^{-1}\circ\beta\circ\gamma\circ\nabla)_{[x_{j}]}(\mathcal{S}_2)\leq (\gamma^{-1}\circ\beta\circ\gamma\circ\nabla)_{[x_{j}]}(\mathcal{S}_1).
$$ 
Conversely, suppose that $(\gamma^{-1}\circ\beta\circ\gamma\circ\nabla)_{[x_{j}]}(\mathcal{S}_2)\leq (\gamma^{-1}\circ\beta\circ\gamma\circ\nabla)_{[x_{j}]}(\mathcal{S}_1)$. We deduce
\begin{align}
   \mathbb{D}[(\gamma^{-1}\circ\beta\circ\gamma\circ\nabla)_{[x_{j}]}(\mathcal{S}_1)~|~(\gamma^{-1}\circ\beta\circ\gamma\circ\nabla)_{[x_{j}]}(\mathcal{S}_n)]&<\mathbb{D}[(\gamma^{-1}\circ\beta\circ\gamma\circ\nabla)_{[x_{j}]}(\mathcal{S}_2)~|\nonumber \\&~(\gamma^{-1}\circ\beta\circ\gamma\circ\nabla)_{[x_{j}]}(\mathcal{S}_n)]\nonumber   
\end{align}
since 
$$
\mathbb{D}[(\gamma^{-1}\circ\beta\circ\gamma\circ\nabla)_{[x_{j}]}(\mathcal{S}_1)~|~(\gamma^{-1}\circ\beta\circ\gamma\circ\nabla)_{[x_{j}]}(\mathcal{S}_n)]
$$ 
and 
$$
\mathbb{D}[(\gamma^{-1}\circ\beta\circ\gamma\circ\nabla)_{[x_{j}]}(\mathcal{S}_2)~|~(\gamma^{-1}\circ\beta\circ\gamma\circ\nabla)_{[x_{j}]}(\mathcal{S}_n)]
$$ 
are the smallest such numbers.
\end{proof}
\bigskip

The notion of the totient introduced in previous sections provides a specific time range for an expansion to run into extinction. In other words, the totient is the time taken for an expansion to come to a complete halt. The next result controls the total dominating number of any chain of sub-expansion of an expansion by an expression involving the totient.

\begin{proposition}
Let $(\gamma^{-1}\circ\beta\circ\gamma\circ\nabla)_{[x_{j}]}(\mathcal{S}_1)\leq (\gamma^{-1}\circ\beta\circ\gamma\circ\nabla)_{[x_{j}]}(\mathcal{S}_2)\leq \cdots \leq (\gamma^{-1}\circ\beta\circ\gamma\circ\nabla)_{[x_{j}]}(\mathcal{S}_n)$ be a chain of sub-expansions of the expansion $(\gamma^{-1}\circ\beta\circ\gamma\circ\nabla)_{[x_{j}]}(\mathcal{S}_n)$. We have
\begin{align}
    &\sum \limits_{i=1}^{n}\mathbb{D}[(\gamma^{-1}\circ\beta\circ\gamma\circ\nabla)_{[x_{j}]}(\mathcal{S}_i)~|~(\gamma^{-1}\circ\beta\circ\gamma\circ\nabla)_{[x_{j}]}(\mathcal{S}_{n})]\nonumber \\& \leq \frac{\Phi[(\gamma^{-1}\circ\beta\circ\gamma\circ\nabla)_{[x_{j}]}(\mathcal{S}_n)]-1}{2}\nonumber \times (\Phi[(\gamma^{-1}\circ\beta\circ\gamma\circ\nabla)_{[x_{j}]}(\mathcal{S}_n)]-2).\nonumber
\end{align}
\end{proposition}

\begin{proof}
We insert the chain $(\gamma^{-1}\circ\beta\circ\gamma\circ\nabla)_{[x_{j}]}(\mathcal{S}_1)\leq (\gamma^{-1}\circ\beta\circ\gamma\circ\nabla)_{[x_{j}]}(\mathcal{S}_2)\leq \cdots \leq (\gamma^{-1}\circ\beta\circ\gamma\circ\nabla)_{[x_{j}]}(\mathcal{S}_n)$ of sub-expansions of the expansion $(\gamma^{-1}\circ\beta\circ\gamma\circ\nabla)_{[x_{j}]}(\mathcal{S}_n)$ into the full chain of sub-expansion of the expansion $(\gamma^{-1}\circ\beta\circ\gamma\circ\nabla)_{[x_{j}]}(\mathcal{S}_n)$. The sum on the left side is bounded by 
\begin{align}
    \sum \limits_{i=2}^{\Phi[(\gamma^{-1}\circ\beta\circ\gamma\circ\nabla)_{[x_{j}]}(\mathcal{S}_n)]-1}\bigg(\Phi[(\gamma^{-1}\circ\beta\circ\gamma\circ\nabla)_{[x_{j}]}(\mathcal{S}_n)]-i\bigg).\nonumber
\end{align}
\end{proof}
\bigskip

Here, we compare the dominating number to the index of an expansion.

\begin{proposition}\label{dominating-index}
If $(\gamma^{-1}\circ\beta\circ\gamma\circ\nabla)_{[x_{j}]}(\mathcal{S}_1)\leq (\gamma^{-1}\circ\beta\circ\gamma\circ\nabla)_{[x_{j}]}(\mathcal{S}_2)$, then 
\begin{align}
   \bigg[(\gamma^{-1}\circ\beta\circ\gamma\circ\nabla)_{[x_{j}]}(\mathcal{S}_2):(\gamma^{-1}\circ\beta\circ\gamma\circ\nabla)_{[x_{j}]}(\mathcal{S}_1)\bigg]&\leq \mathbb{D}[(\gamma^{-1}\circ\beta\circ\gamma\circ\nabla)_{[x_{j}]}(\mathcal{S}_1)~|\nonumber \\&~(\gamma^{-1}\circ\beta\circ\gamma\circ\nabla)_{[x_{j}]}(\mathcal{S}_2)].\nonumber 
\end{align}
\end{proposition}

\begin{proof}
Let 
$$
v=\bigg[(\gamma^{-1}\circ\beta\circ\gamma\circ\nabla)_{[x_{j}]}(\mathcal{S}_2):(\gamma^{-1}\circ\beta\circ\gamma\circ\nabla)_{[x_{j}]}(\mathcal{S}_1)\bigg]
$$ 
and 
$$u=\mathbb{D}[(\gamma^{-1}\circ\beta\circ\gamma\circ\nabla)_{[x_{j}]}(\mathcal{S}_1)~|~(\gamma^{-1}\circ\beta\circ\gamma\circ\nabla)_{[x_{j}]}(\mathcal{S}_2)].
$$
We can write 
$$
(\gamma^{-1}\circ\beta\circ\gamma\circ\nabla)_{[x_{j}]}(\mathcal{S}_1)=(\gamma^{-1}\circ\beta\circ\gamma\circ\nabla)^v_{[x_{j}]}(\mathcal{S}_2)
$$ 
and 
$$
(\gamma^{-1}\circ\beta\circ\gamma\circ\nabla)^u_{[x_{j}]}(\mathcal{S}_2)\leq (\gamma^{-1}\circ\beta\circ\gamma\circ\nabla)_{[x_{j}]}(\mathcal{S}_1)=(\gamma^{-1}\circ\beta\circ\gamma\circ\nabla)^v_{[x_{j}]}(\mathcal{S}_2).
$$
It implies that $(\gamma^{-1}\circ\beta\circ\gamma\circ\nabla)^u_{[x_{j}]}(\mathcal{S}_2)$ is a sub-expansion of the expansion $(\gamma^{-1}\circ\beta\circ\gamma\circ\nabla)^v_{[x_{j}]}(\mathcal{S}_2)$ so that there exist some $t\geq 0$ such that 
\begin{align}
    (\gamma^{-1}\circ\beta\circ\gamma\circ\nabla)^u_{[x_{j}]}(\mathcal{S}_2)=(\gamma^{-1}\circ\beta\circ\gamma\circ\nabla)^{v+t}_{[x_{j}]}(\mathcal{S}_2).\nonumber
\end{align}
It follows that $\mathbb{D}[(\gamma^{-1}\circ\beta\circ\gamma\circ\nabla)_{[x_{j}]}(\mathcal{S}_1)~|~(\gamma^{-1}\circ\beta\circ\gamma\circ\nabla)_{[x_{j}]}(\mathcal{S}_2)]=u=v+t\geq v=\bigg[(\gamma^{-1}\circ\beta\circ\gamma\circ\nabla)_{[x_{j}]}(\mathcal{S}_2):(\gamma^{-1}\circ\beta\circ\gamma\circ\nabla)_{[x_{j}]}(\mathcal{S}_1)\bigg]$.
\end{proof}
\bigskip

Here, we generalize the inequality in proposition \ref{dominating-index} to arbitrary sub-expansions in a chain.

\begin{theorem}\label{main}
Let $(\gamma^{-1}\circ\beta\circ\gamma\circ\nabla)_{[x_{j}]}(\mathcal{S}_1)\leq (\gamma^{-1}\circ\beta\circ\gamma\circ\nabla)_{[x_{j}]}(\mathcal{S}_2)\leq \cdots \leq (\gamma^{-1}\circ\beta\circ\gamma\circ\nabla)_{[x_{j}]}(\mathcal{S}_n)$ be a chain of sub-expansions of the expansion $(\gamma^{-1}\circ\beta\circ\gamma\circ\nabla)_{[x_{j}]}(\mathcal{S}_n)$. We have
\begin{align}
   &\sum \limits_{i=k+1}^{n}\bigg[(\gamma^{-1}\circ\beta\circ\gamma\circ\nabla)_{[x_{j}]}(\mathcal{S}_i):(\gamma^{-1}\circ\beta\circ\gamma\circ\nabla)_{[x_{j}]}(\mathcal{S}_{i-1})\bigg]\nonumber \\&<\mathbb{D}[(\gamma^{-1}\circ\beta\circ\gamma\circ\nabla)_{[x_{j}]}(\mathcal{S}_k)~|~(\gamma^{-1}\circ\beta\circ\gamma\circ\nabla)_{[x_{j}]}(\mathcal{S}_n)]+(n-k).\nonumber
\end{align}
\end{theorem}
\bigskip

\begin{proof}
Let 
$$
\mathbb{D}[(\gamma^{-1}\circ\beta\circ\gamma\circ\nabla)_{[x_{j}]}(\mathcal{S}_k)~|~(\gamma^{-1}\circ\beta\circ\gamma\circ\nabla)_{[x_{j}]}(\mathcal{S}_n)]=u.
$$
We obtain
\begin{align}
   (\gamma^{-1}\circ\beta\circ\gamma\circ\nabla)^u_{[x_{j}]}(\mathcal{S}_n)\leq (\gamma^{-1}\circ\beta\circ\gamma\circ\nabla)_{[x_{j}]}(\mathcal{S}_k)\nonumber 
\end{align} 
so that there exists some $r_k\in \mathbb{N}$ such that 
\begin{align}
   (\gamma^{-1}\circ\beta\circ\gamma\circ\nabla)^u_{[x_{j}]}(\mathcal{S}_n)=(\gamma^{-1}\circ\beta\circ\gamma\circ\nabla)^{r_k}_{[x_{j}]}(\mathcal{S}_k).\nonumber  
\end{align}
Using the chain of sub-expansion 
$$
(\gamma^{-1}\circ\beta\circ\gamma\circ\nabla)_{[x_{j}]}(\mathcal{S}_k)\leq (\gamma^{-1}\circ\beta\circ\gamma\circ\nabla)_{[x_{j}]}(\mathcal{S}_{k+1})
$$ 
we find that there exists some $r_{k+1}\in \mathbb{N}$ such that 
\begin{align}
   (\gamma^{-1}\circ\beta\circ\gamma\circ\nabla)_{[x_{j}]}(\mathcal{S}_k)=(\gamma^{-1}\circ\beta\circ\gamma\circ\nabla)^{r_{k+1}}_{[x_{j}]}(\mathcal{S}_{k+1})\nonumber
\end{align}
so that 
\begin{align}
    \gamma^{-1}\circ\beta\circ\gamma\circ\nabla)^u_{[x_{j}]}(\mathcal{S}_n)&=(\gamma^{-1}\circ\beta\circ\gamma\circ\nabla)^{r_k}_{[x_{j}]}(\mathcal{S}_k)\nonumber \\&=(\gamma^{-1}\circ\beta\circ\gamma\circ\nabla)^{r_k+r_{k+1}-1}_{[x_{j}]}(\mathcal{S}_{k+1}).\nonumber
\end{align}
 Further of the chain $(\gamma^{-1}\circ\beta\circ\gamma\circ\nabla)_{[x_{j}]}(\mathcal{S}_{k+1})\leq (\gamma^{-1}\circ\beta\circ\gamma\circ\nabla)_{[x_{j}]}(\mathcal{S}_{k+2})$ implies that there exists some $r_{k+2}\in \mathbb{N}$ such that 
\begin{align}
   (\gamma^{-1}\circ\beta\circ\gamma\circ\nabla)_{[x_{j}]}(\mathcal{S}_{k+1})=(\gamma^{-1}\circ\beta\circ\gamma\circ\nabla)^{r_{k+2}}_{[x_{j}]}(\mathcal{S}_{k+2})\nonumber  
\end{align}
so that
\begin{align}
   (\gamma^{-1}\circ\beta\circ\gamma\circ\nabla)^u_{[x_{j}]}(\mathcal{S}_n)&=(\gamma^{-1}\circ\beta\circ\gamma\circ\nabla)^{r_k}_{[x_{j}]}(\mathcal{S}_k)\nonumber \\&=(\gamma^{-1}\circ\beta\circ\gamma\circ\nabla)^{r_k+r_{k+1}-1}_{[x_{j}]}(\mathcal{S}_{k+1})\nonumber \\&=(\gamma^{-1}\circ\beta\circ\gamma\circ\nabla)^{r_k+r_{k+1}+r_{k+2}-2}_{[x_{j}]}(\mathcal{S}_{k+2}).\nonumber
\end{align}
Continuing the argument, we obtain 
\begin{align}
   (\gamma^{-1}\circ\beta\circ\gamma\circ\nabla)^u_{[x_{j}]}(\mathcal{S}_n)&= (\gamma^{-1}\circ\beta\circ\gamma\circ\nabla)^{r_k+r_{k+1}+r_{k+2}+\cdots+r_n-(n-k)}_{[x_{j}]}(\mathcal{S}_{n}).\nonumber
\end{align}
We deduce
\begin{align}
   u&=r_k+\sum \limits_{i=k+1}^{n}r_i-(n-k)\nonumber \\& \geq \sum \limits_{i=k+1}^{n}r_i-(n-k)\nonumber
\end{align}
with 
$$
r_i=\bigg[(\gamma^{-1}\circ\beta\circ\gamma\circ\nabla)_{[x_{j}]}(\mathcal{S}_i):(\gamma^{-1}\circ\beta\circ\gamma\circ\nabla)_{[x_{j}]}(\mathcal{S}_{i-1})\bigg]
$$ 
and 
$$\mathbb{D}[(\gamma^{-1}\circ\beta\circ\gamma\circ\nabla)_{[x_{j}]}(\mathcal{S}_k)~|~(\gamma^{-1}\circ\beta\circ\gamma\circ\nabla)_{[x_{j}]}(\mathcal{S}_n)]=u
$$ 
establishing the asserted inequality.
\end{proof}

\begin{corollary}
Let $(\gamma^{-1}\circ\beta\circ\gamma\circ\nabla)_{[x_{j}]}(\mathcal{S}_1)\leq (\gamma^{-1}\circ\beta\circ\gamma\circ\nabla)_{[x_{j}]}(\mathcal{S}_2)\leq \cdots \leq (\gamma^{-1}\circ\beta\circ\gamma\circ\nabla)_{[x_{j}]}(\mathcal{S}_n)$ be a chain of sub-expansions of the expansion $(\gamma^{-1}\circ\beta\circ\gamma\circ\nabla)_{[x_{j}]}(\mathcal{S}_n)$. We have
\begin{align}
   \sum \limits_{i=k+1}^{n}\bigg[(\gamma^{-1}\circ\beta\circ\gamma\circ\nabla)_{[x_{j}]}(\mathcal{S}_i):(\gamma^{-1}\circ\beta\circ\gamma\circ\nabla)_{[x_{j}]}(\mathcal{S}_{i-1})\bigg]>n-k.\nonumber 
\end{align}
\end{corollary}
\bigskip

\section{Analytic expansions and application to function theory}

In this section, we introduce the notion of \emph{singularity}, the \emph{kernel} and a\emph{analytic expansions}. We provide an application to the existence of singularities of solutions to certain polynomial equations of several variables.

\subsection{Some notions from single variable expansivity theory}

In this section, we review some notions in single variable expansivity theory developed. 

\begin{definition}\label{rank}
Let $\mathcal{F}=\{\mathcal{S}_m\}^{\infty}_{m=1}$ be a family of tuples of polynomials in the ring $\mathbb{R}[x]$ each having at least two entries with distinct degrees. The value of $n$ such that the expansion $(\gamma^{-1}\circ\beta\circ\gamma\circ\nabla)^{n}(\mathcal{S})\neq \mathcal{S}_{0}$ and $(\gamma^{-1}\circ\beta\circ\gamma\circ\nabla)^{n+1}(\mathcal{S})=\mathcal{S}_{0}$ where $\mathcal{S}_{0}=(0,0,\ldots,0)$ is called the \emph{degree} of expansion. We call $(\gamma^{-1}\circ\beta\circ\gamma\circ\nabla)^{n}(\mathcal{S})$ is the \emph{rank} of expansion, denoted by $\mathcal{R}(\mathcal{S})$. 
\end{definition}

\begin{theorem}\label{rank2}
Let $\mathcal{S}_{1}$ and $\mathcal{S}_{2}$ be tuples of polynomials in the ring $\mathbb{R}[x]$ with $\mathrm{deg}(\mathcal{S}_1)>\mathrm{deg}(\mathcal{S}_{2})$ and satisfying certain initial conditions at each phase of expansion. If $\mathcal{R}(\mathcal{S}_{1})=\mathcal{R}(\mathcal{S}_{2})$, then there exist some $j$ satisfying $1\leq j <\mathrm{deg}(\mathcal{S}_{1})$ such that $\mathcal{S}_{1}^{j}=\mathcal{S}_{2}$.
\end{theorem}

\begin{proof}
Suppose that $\mathcal{S}_{1}$ and $\mathcal{S}_{2}$ are tuples of polynomials in the ring $\mathbb{R}[x]$. Assume that $\mathrm{deg}(\mathcal{S}_{1})=k_1$ and $\mathrm{deg}(\mathcal{S}_{2})=k_2$. By definition \ref{rank}, we can write $\mathcal{R}(\mathcal{S}_{1})=(\gamma^{-1}\circ\beta\circ\gamma\circ\nabla)^{k_1}(\mathcal{S}_1)$ and $\mathcal{R}(\mathcal{S}_{2})=(\gamma^{-1}\circ\beta\circ\gamma\circ\nabla)^{k_2}(\mathcal{S}_2)$. Under the assumption $\mathcal{R}(\mathcal{S}_{1})=\mathcal{R}(\mathcal{S}_{2})$, we get $(\gamma^{-1}\circ\beta\circ\gamma\circ\nabla)^{k_2}(\mathcal{S}_2)=(\gamma^{-1}\circ\beta\circ\gamma\circ\nabla)^{k_1}(\mathcal{S}_1)$ if and only if $(\gamma^{-1}\circ\beta\circ\gamma\circ\nabla)^{k_1-k_2}(\mathcal{S}_1)=\mathcal{S}_{2}$. Since $1\leq k_1-k_2<k_1=\mathrm{deg}(\mathcal{S}_{1})$, we deduce the asserted claim.
\end{proof}

\begin{definition}\label{limit}
Let $\{\mathcal{S}^{m}\}^{\infty}_{m=1}$ be a family of expanded tuples of $\mathcal{S}$ having at least two entries of distinct degrees. The limit of expansion of $\mathcal{S}$ is the first expanded tuple $\mathcal{S}^{j}=(g_1,g_2,\ldots,g_n)$ such that $\mathrm{deg}(g_1)=\cdots=\mathrm{deg}(g_n)$ for $n\geq 3$ and $1\leq j\leq m$. We denote the limit by 
\begin{align}
\lim (\mathcal{S}^{n})=\mathcal{S}^{j}.\nonumber
\end{align}
\end{definition} 

\begin{theorem}\label{exist}
Let $\{\mathcal{S}^{m}\}^{\infty}_{m=1}$ be a family of expansions of the tuple $\mathcal{S}$ of polynomials in the ring $\mathbb{R}[x]$ such that at least two entries are of distinct degree. The limit of expansions $\lim (\mathcal{S}^{n})$ of $\mathcal{S}$ exists.
\end{theorem}

\begin{proof}
Let $\{\mathcal{S}^{m}\}^{\infty}_{m=1}$ be a family of expansions of the tuple $\mathcal{S}$ of polynomials in the ring $\mathbb{R}[x]$ having at least two entries of distinct degree. Suppose that the limit of expansion does not exist, and let $\mathcal{S}^{1}=(f_1,f_2,\ldots,f_n)$ be the first phase expansion of $\mathcal{S}$. It implies $\mathrm{deg}(f_i)\neq \mathrm{deg}(f_j)$ for some $1\leq i,j\leq n$ with $i\neq j$. It follows in particular that $\mathcal{S}^{1}\neq \mathcal{R}(\mathcal{S})$ and $\mathcal{S}^{1}\neq \mathcal{S}_{0}$. Thus, the second phase expansion exists. Let $\mathcal{S}^{2}=(g_1,g_2,\ldots,g_n)$ be the second phase expanded tuple. By hypothesis, we have $\mathrm{deg}(g_i)\neq \mathrm{deg}(g_{j})$ for some $1\leq i,j\leq n$ with $i\neq j$. It implies $\mathcal{S}^{2}\neq \mathcal{R}(\mathcal{S})$ and $\mathcal{S}^{2}\neq \mathcal{S}_{0}$. Thus, the third phase expansion exist. By induction, we conclude that the tuple $\mathcal{S}$ of $\mathbb{R}[x]$ admits infinite number of expansions. This violates the finite degrees of the entries.
\end{proof}

\begin{theorem}\label{lemma1}
Let $\{\mathcal{S}^{n}\}_{n=1}^{\infty}$ be a family of expanded tuples of the tuple $\mathcal{S}$ of polynomials in the ring $\mathbb{R}[x]$ such that at least two entries are of distinct degrees and satisfies certain initial conditions at each phase of expansion. There exist some number $k$ called the \emph{dimension} of expansion ($\mathrm{dim}(\mathcal{S})$) such that 
$$
\lim (\mathcal{S}^{n})=(\Delta\circ\gamma^{-1}\circ\beta^{-1}\circ\gamma)^{k}(\mathcal{R}(\mathcal{S}))
$$ 
for some $k<\mathrm{deg}(\mathcal{S})$. 
\end{theorem}

\begin{proof}
Let $\mathcal{S}$ be any tuple of polynomials in the ring $\mathbb{R}[x]$ that can be expanded with at least two entries of distinct degree. The limit exists by Theorem \ref{exist}. Since an expansion can only be finitely applied and the map $\Delta\circ\gamma^{-1}\circ\beta^{-1}\circ\gamma$ is a recovery map which exists, there exist such number $k$ so that  $\lim (\mathcal{S}^{n})=(\Delta\circ\gamma^{-1}\circ\beta^{-1}\circ\gamma)^{k}(\mathcal{R}(\mathcal{S}))$. We only need to show that $k$ lies in the stated range. Suppose that $\lim (\mathcal{S}^{n})=(\Delta\circ\gamma^{-1}\circ\beta^{-1}\circ\gamma)^{k}(\mathcal{R}(\mathcal{S}))$ for any $k\geq \mathrm{deg}(\mathcal{S})$. Since  the map is a bijection, it follows that $(\gamma^{-1}\circ\beta\circ\gamma\circ\nabla)^{k}(\lim (\mathcal{S}^{n}))=\mathcal{R}(\mathcal{S})$. We deduce 
$$
(\gamma^{-1}\circ\beta\circ\gamma\circ\nabla)^{k}(\lim (\mathcal{S}^{n})=\mathcal{S}_{0}
$$ 
which implies $\mathcal{R}(\mathcal{S})=\mathcal{S}_{0}$ and so the rank of an expansion is null. This violate the definition \ref{rank}.
\end{proof}

\begin{definition}
Let $\mathcal{S}$ be a tuple of polynomial in the ring $\mathbb{R}[x]$ and $\{\mathcal{S}^{m}\}_{m=1}^{\infty}$ be the family of expanded tuple of $\mathcal{S}$. By the \emph{local number} of expansion, denoted by $\mathcal{L}(\mathcal{S})$, we mean the value of $n$ such that 
$$
(\gamma^{-1}\circ\beta\circ\gamma\circ\nabla)^{n}(\mathcal{S})=\lim (\mathcal{S}^{m}).
$$
\end{definition}
\bigskip

Using Theorem \ref{lemma1} and the assumption that any tuple of polynomial in the ring $\mathbb{R}[x]$ satisfying certain initial conditions at each phase of expansion, we get
\begin{align}
(\gamma^{-1}\circ\beta\circ\gamma\circ\nabla)^{n}(\mathcal{S})=(\Delta\circ\gamma^{-1}\circ\beta^{-1}\circ\gamma)^{k}(\mathcal{R}(\mathcal{S}))\nonumber
\end{align}
if and only if 
\begin{align}
(\gamma^{-1}\circ\beta\circ\gamma\circ\nabla)^{n+k}(\mathcal{S})=\mathcal{R}(\mathcal{S}).\nonumber
\end{align}
By the definition of the rank, we get
\begin{align}
n+k=\mathrm{deg}(\mathcal{S})\nonumber
\end{align}
which we call the principal equation and where $\mathcal{L}(\mathcal{S})=n$, $\mathrm{dim}(\mathcal{S})=k$ and $\mathrm{deg}(\mathcal{S})$ are the local number, the dimension and the degree of expansion, respectively, of an expansion on $\mathcal{S}$. It is interesting to recognize that the value of the local number $\mathcal{L}(\mathcal{S})$ is bounded. 

\begin{lemma}\label{local number}
Let $\mathcal{S}$ be a tuple of polynomials in the ring $\mathbb{R}[x]$, satisfying certain initial conditions at each phase with $deg(\mathcal{S})\geq 4$.  If $dim(\mathcal{S})>2$, then the local number $\mathcal{L}(\mathcal{S})$ must satisfy the inequality 
\begin{align}
0 \leq \mathcal{L}(\mathcal{S})\leq 2. \nonumber
\end{align}
\end{lemma}

\subsection{The kernel of an expansion}

In this section, we introduce the notion of the \emph{kernel} of an expansion. It is similar to the notion of the boundary points of an expansion under the single variable theory. This choice of terminology is appropriate for this context, since we are no longer considering real tuples as solutions to our tuple equation but tuples consisting of solutions to certain partial differential equation. 

\begin{definition}\label{kernel}
Let $\mathcal{F}=\{\mathcal{S}_i\}_{i=1}^{\infty}$ be a collection of $l$-tuples of polynomials in the ring $\mathbb{R}[x_1,x_2,\ldots,x_n]$. By the \emph{kernel} of expansion $(\gamma^{-1}\circ\beta\circ\gamma\circ\nabla)^k_{[x_{i}]}(\mathcal{S})$, denoted by $\mathrm{Ker}[(\gamma^{-1}\circ\beta\circ\gamma\circ\nabla)^k_{[x_{i}]}(\mathcal{S})]$, we mean 
\begin{align}
\mathrm{Ker}[(\gamma^{-1}\circ\beta\circ\gamma\circ\nabla)^k_{[x_{i}]}(\mathcal{S})]&=\bigg \{ (f_1,f_2,\ldots,f_l)~|~\nonumber \\& f_r \in \mathbb{F}_\mathbb{C}[x_1,\ldots,x_{i-1},x_{i+1},\ldots,x_n],\nonumber \\&~(1\leq r \leq l),\mathrm{Id}_r[(\gamma^{-1}\circ\beta\circ\gamma\circ\nabla)^k_{[x_{i}](f_r)}(\mathcal{S})]=0\bigg\} \nonumber
\end{align}
where $\mathbb{F}_\mathbb{C}$ is a function field with complex number $\mathbb{C}$ base space. We call each tuple in the kernel an \emph{annihilator} of the given expansion.
\end{definition}
\bigskip

\begin{definition}\label{decomposition}
Let $\mathcal{F}=\{\mathcal{S}_i\}_{i=1}^{\infty}$ be a collection of $l$-tuples of polynomials in the ring $\mathbb{C}[x_1,x_2,\ldots,x_n]$. We denote by 
\begin{align}
\mathrm{Id}_r[(\gamma^{-1}\circ\beta\circ\gamma\circ\nabla)^k_{[x_{i}]}(\mathcal{S})](f_r)_{x_i}\nonumber
\end{align}
the value of 
\begin{align}
\mathrm{Id}_r[(\gamma^{-1}\circ\beta\circ\gamma\circ\nabla)^k_{[x_{i}]}(\mathcal{S})]\nonumber
\end{align}
at $x_i=f_r$.
\end{definition}
\bigskip

\begin{proposition}\label{kernel uniqueness}
Let $\mathcal{F}=\{\mathcal{S}_i\}_{i=1}^{\infty}$ be a collection of $l$-tuples of polynomials in the ring $\mathbb{R}[x_1,x_2,\ldots,x_n]$. If for $\mathcal{S}_1,\mathcal{S}_2\in \mathcal{F}$ 
\begin{align}
\mathrm{Ker}[(\gamma^{-1}\circ\beta\circ\gamma\circ\nabla)^k_{[x_{i}]}(\mathcal{S}_1)]&=\mathrm{Ker}[(\gamma^{-1}\circ\beta\circ\gamma\circ\nabla)^k_{[x_{i}]}(\mathcal{S}_2)]\nonumber
\end{align}
then $\mathcal{S}_1=\mathcal{S}_2+\mathcal{S}_{\mathbb{C}[x_1,\ldots,x_{i-1},x_{i+1},\ldots,x_n]}$ and where $\mathcal{S}_{\mathbb{C}[x_1,\ldots,x_{i-1},x_{i+1},\ldots,x_n]}$ is an $l$ tuple of polynomials in the ring $\mathbb{C}[x_1,\ldots,x_{i-1},x_{i+1},\ldots,x_n]$.
\end{proposition}

\begin{proof}
Suppose that $\mathcal{S}_1,\mathcal{S}_2\in \mathcal{F}$ and 
\begin{align}
\mathrm{Ker}[(\gamma^{-1}\circ\beta\circ\gamma\circ\nabla)^k_{[x_{i}]}(\mathcal{S}_1)]&=\mathrm{Ker}[(\gamma^{-1}\circ\beta\circ\gamma\circ\nabla)^k_{[x_{i}]}(\mathcal{S}_2)].\nonumber
\end{align}
For any $(f_1,f_2,\ldots,f_l)\in \mathrm{Ker}[(\gamma^{-1}\circ\beta\circ\gamma\circ\nabla)^k_{[x_{i}]}(\mathcal{S}_1)]$ we get
$$
(f_1,f_2,\ldots,f_l)\in \mathrm{Ker}[(\gamma^{-1}\circ\beta\circ\gamma\circ\nabla)^k_{[x_{i}]}(\mathcal{S}_2)]
$$ 
so that we can write
\begin{align}
\mathrm{Id}_r[(\gamma^{-1}\circ\beta\circ\gamma\circ\nabla)^k_{[x_{i}](f_r)}(\mathcal{S}_1)]&=\mathrm{Id}_r[(\gamma^{-1}\circ\beta\circ\gamma\circ\nabla)^k_{[x_{i}](f_r)}(\mathcal{S}_2)]=0\nonumber
\end{align}
for $1\leq r \leq l$. Using the definition \ref{decomposition}, we write
\begin{align}
\mathrm{Id}_r[(\gamma^{-1}\circ\beta\circ\gamma\circ\nabla)^k_{[x_{i}]}(\mathcal{S}_1)](f_r)_{x_i}&=\mathrm{Id}_r[(\gamma^{-1}\circ\beta\circ\gamma\circ\nabla)^k_{[x_{i}](f_r)}(\mathcal{S}_1)]\nonumber \\&=0\nonumber \\&=\mathrm{Id}_r[(\gamma^{-1}\circ\beta\circ\gamma\circ\nabla)^k_{[x_{i}](f_r)}(\mathcal{S}_2)]\nonumber \\&=\mathrm{Id}_r[(\gamma^{-1}\circ\beta\circ\gamma\circ\nabla)^k_{[x_{i}]}(\mathcal{S}_2)](f_r)_{x_i}\nonumber
\end{align}
for $1\leq r \leq l$. We deduce
\begin{align}
(\gamma^{-1}\circ\beta\circ\gamma\circ\nabla)^k_{[x_{i}]}(\mathcal{S}_1)=(\gamma^{-1}\circ\beta\circ\gamma\circ\nabla)^k_{[x_{i}]}(\mathcal{S}_2)\nonumber
\end{align}
in 
\begin{align}
\mathrm{Ker}[(\gamma^{-1}\circ\beta\circ\gamma\circ\nabla)^k_{[x_{i}]}(\mathcal{S}_1)]&=\mathrm{Ker}[(\gamma^{-1}\circ\beta\circ\gamma\circ\nabla)^k_{[x_{i}]}(\mathcal{S}_2)].\nonumber
\end{align}
By the linearity of an expansion in a specific direction, we get 
\begin{align}
\mathcal{S}_1=\mathcal{S}_2+\mathcal{S}_{\mathbb{R}[x_1,\ldots,x_{i-1},x_{i+1},\ldots,x_n]}\nonumber
\end{align}
since an expansion in a specific direction is uniquely determined by their kernel.
\end{proof}
\bigskip

It may be possible that distinct non-equivalent expansions in separate directions at spots can have the same kernel. In the following, we show that all hybrid expansions have the same kernel of their expansions.

\begin{proposition}\label{kernel-diagonal}
Let $\mathcal{F}=\{\mathcal{S}_i\}_{i=1}^{\infty}$ be a collection of tuples of polynomials in the ring $\mathbb{R}[x_1,x_2,\ldots,x_n]$. If $k\neq j$ with $1\leq j,k\leq l$, then there exists some $\mathcal{S}_t,\mathcal{S}_r\in \mathcal{F}$ with $\mathcal{S}_t-\mathcal{S}_r \neq \mathcal{S}_{\mathbb{R}}$ and some $u,v\geq 1$ such that
\begin{align}
\mathrm{Ker}[(\gamma^{-1}\circ\beta\circ\gamma\circ\nabla)^u_{[x_{\sigma(i)}]}(\mathcal{S}_t)]&=\mathrm{Ker}[(\gamma^{-1}\circ\beta\circ\gamma\circ\nabla)^v_{[x_{\sigma(j)}]}(\mathcal{S}_r)].\nonumber
\end{align}
\end{proposition}

\begin{proof}
Using the lemma \ref{exist diagonalization} and the assumption $k\neq j$ with $1\leq j,k\leq l$, there exists some $u,v\geq 1$ and $\mathcal{S}_t,\mathcal{S}_r\in \mathcal{F}$ with $\mathcal{S}_t-\mathcal{S}_r\neq \mathcal{S}_{\mathbb{R}}$ such that we can write for $\mathcal{S}\in \mathcal{F}$
\begin{align}
(\gamma^{-1}\circ\beta\circ\gamma\circ\nabla)_{\otimes_{i=1}^{l}[x_{\sigma(i)}]}(\mathcal{S})=(\gamma^{-1}\circ\beta\circ\gamma\circ\nabla)^u_{[x_{\sigma(k)}]}(\mathcal{S}_t)\nonumber
\end{align}
and 
\begin{align}
(\gamma^{-1}\circ\beta\circ\gamma\circ\nabla)_{\otimes_{i=1}^{l}[x_{\sigma(i)}]}(\mathcal{S})&=(\gamma^{-1}\circ\beta\circ\gamma\circ\nabla)^v_{[x_{\sigma(j)}]}(\mathcal{S}_r)\nonumber 
\end{align}
so that 
\begin{align}
(\gamma^{-1}\circ\beta\circ\gamma\circ\nabla)^u_{[x_{\sigma(k)}]}(\mathcal{S}_t)&=(\gamma^{-1}\circ\beta\circ\gamma\circ\nabla)^v_{[x_{\sigma(j)}]}(\mathcal{S}_r).\nonumber
\end{align}
This establishes the claim.
\end{proof}
\bigskip

\subsection{Singularity and singular points of an expansion}

In this section, we introduce the notion of \emph{singularity} and associated \emph{singular} points of an expansion in a specific direction.

\begin{definition}\label{singularity}
Let $\mathcal{F}=\{\mathcal{S}_i\}_{i=1}^{\infty}$ be a collection of $l$-tuples of polynomials in the ring $\mathbb{R}[x_1,x_2,\ldots,x_n]$ and $\mathrm{Ker}[(\gamma^{-1}\circ\beta\circ\gamma\circ\nabla)^k_{[x_{i}]}(\mathcal{S})]$ be the \emph{kernel} of expansion $(\gamma^{-1}\circ\beta\circ\gamma\circ\nabla)^k_{[x_{i}]}(\mathcal{S})$. By a \emph{singular} point of the expansion, we mean a tuple $\mathcal{S}=(a_1,\ldots a_{i-1},a_{i+1},\ldots,a_n)$ with $a_j\in \mathbb{C}$ such that for some annihilator
\begin{align}
(f_1,f_2,\ldots,f_l)\in \mathrm{Ker}[(\gamma^{-1}\circ\beta\circ\gamma\circ\nabla)^k_{[x_{i}]}(\mathcal{S})]\nonumber
\end{align}
we have 
\begin{align}
f_i[(a_1,\ldots a_{k-1},a_{k+1},\ldots,a_n)]=\infty \nonumber
\end{align}
for some $1\leq i\leq l$. We call the collection of all such points the \emph{singularity} of the expansion. We denote the singularity by
\begin{align}
\mathrm{Sing}[(\gamma^{-1}\circ\beta\circ\gamma\circ\nabla)^k_{[x_{i}]}(\mathcal{S})]&=\bigg \{(a_1,\ldots a_{i-1},a_{i+1},\ldots,a_n) \in \mathbb{C}^{n-1}|\nonumber \\&~f_i[(a_1,\ldots a_{k-1},a_{k+1},\ldots,a_n)]=\infty \bigg \}.\nonumber
\end{align}
\end{definition}

\begin{definition}
Let $\mathcal{F}=\{\mathcal{S}_i\}_{i=1}^{\infty}$ be a collection of tuples of polynomials in the ring $\mathbb{R}[x_1,x_2,\ldots,x_n]$ and $\mathcal{D}\subset \mathbb{C}^{n-1}$. We say that the expansion $(\gamma^{-1}\circ\beta\circ\gamma\circ\nabla)^k_{[x_{i}]}(\mathcal{S})$ is \emph{analytic} in $\mathcal{D}$ if 
\begin{align}
\mathcal{D}\cap \mathrm{Sing}[(\gamma^{-1}\circ\beta\circ\gamma\circ\nabla)^k_{[x_{i}]}(\mathcal{S})]=\emptyset.\nonumber
\end{align}
If the expansion is analytic in the entire $\mathcal{C}^{n-1}$, then we say it is \emph{analytic}.
\end{definition} 

\begin{definition}
Let $f_k\in \mathbb{R}[x_1,x_2,\ldots,x_n]$ be a polynomial. By the \emph{index} of $x_i$ for $1\leq i\leq n$ relative to $f_k$, denoted by $\mathrm{Ind}_{f_k}(x_i)$, we mean the largest power (exponent) of $x_i$ in the polynomial $f_k$.
\end{definition}
\bigskip

\begin{lemma}
Let $\mathcal{S}=(f_1,f_2,\ldots,f_s)$ be a tuple of polynomials such that $f_i\in \mathbb{R}[x_1,x_2,\ldots,x_n]$ for $1\leq i \leq s$. We have
\begin{align}
\Phi[(\gamma^{-1}\circ\beta\circ\gamma\circ\nabla)_{[x_{j}]}(\mathcal{S})]=\mathrm{max}\{\mathrm{Ind}_{f_i}(x_j)\}_{i=1}^{s}+1\nonumber
\end{align}
for any $1\leq j\leq n$.
\end{lemma}

\begin{definition}\label{unionization stage}
Let $\mathcal{F}=\{\mathcal{S}_i\}_{i=1}^{\infty}$ be a collection of tuples of polynomials in the ring $\mathbb{R}[x_1,x_2,\ldots,x_n]$ and $(\gamma^{-1}\circ\beta\circ\gamma\circ\nabla)_{[x_{i}]}(\mathcal{S})$ be an expansion. By the \emph{unionization} stage of the expansion, we mean the least of value of $j$ such that 
\begin{align}
(\gamma^{-1}\circ\beta\circ\gamma\circ\nabla)^j_{[x_{i}](0)}(\mathcal{S})=\mathcal{S}_0.\nonumber
\end{align}
\end{definition}
\bigskip

\begin{definition}\label{fibre}
Let $\mathcal{F}=\{\mathcal{S}_i\}_{i=1}^{\infty}$ be a collection of $l$-tuples of polynomials in the ring $\mathbb{R}[x_1,x_2,\ldots,x_n]$ and $(\gamma^{-1}\circ\beta\circ\gamma\circ\nabla)_{[x_{i}]}(\mathcal{S})$ be an expansion. By the \emph{normalization} stage of the expansion, denoted by $\varrho[(\gamma^{-1}\circ\beta\circ\gamma\circ\nabla)_{[x_{i}]}(\mathcal{S})]$, we mean the smallest value of $k$ such that 
\begin{align}
\mathrm{Ind}_{\mathrm{Id}_r[(\gamma^{-1}\circ\beta\circ\gamma\circ\nabla)^k_{[x_{i}]}(\mathcal{S})]}(x_i)=\mathrm{Ind}_{\mathrm{Id}_s[(\gamma^{-1}\circ\beta\circ\gamma\circ\nabla)^k_{[x_{i}]}(\mathcal{S})]}(x_i)\nonumber
\end{align}
for all $1\leq r,s\leq l$ with $r\neq s$. We call the corresponding expansion 
\begin{align}
(\gamma^{-1}\circ\beta\circ\gamma\circ\nabla)^k_{[x_{i}]}(\mathcal{S})\nonumber
\end{align}
the \emph{fiber} of the expansion $(\gamma^{-1}\circ\beta\circ\gamma\circ\nabla)_{[x_{i}]}(\mathcal{S})$.  
\end{definition}
\bigskip

We will assume the normalization stage of an expansion satisfies $\varrho[(\gamma^{-1}\circ\beta\circ\gamma\circ\nabla)_{[x_{i}]}(\mathcal{S})]>0$ by working with tuples of multivariate polynomials with at least two entries of distinct degree of the underlying direction. We will also implicitly assume that any two entries $f_i,f_j\in \mathbb{R}[x_1,x_2,\ldots,x_n]$ of the tuple $\mathcal{S}$ must satisfy $f_j\neq g\cdot f_i$ for any $g\in \mathbb{R}[x_1,x_2,\ldots,x_n]$.

\begin{proposition}\label{unionization stage proof}
Let $\mathcal{F}=\{\mathcal{S}_i\}_{i=1}^{\infty}$ be a collection of tuples of polynomials in the ring $\mathbb{R}[x_1,x_2,\ldots,x_n]$ and $(\gamma^{-1}\circ\beta\circ\gamma\circ\nabla)_{[x_{i}]}(\mathcal{S})$ be an expansion. The \emph{unionization} stage of the expansion satisfies the inequality
\begin{align}
j\geq \left \lfloor \frac{\Phi[(\gamma^{-1}\circ\beta\circ\gamma\circ\nabla)_{[x_{i}]}(\mathcal{S})]}{\varrho[(\gamma^{-1}\circ\beta\circ\gamma\circ\nabla)_{[x_{i}]}(\mathcal{S})]}\right \rfloor.\nonumber
\end{align}
\end{proposition}

\begin{proof}
The normalization stage of the expansion $(\gamma^{-1}\circ\beta\circ\gamma\circ\nabla)_{[x_{i}]}(\mathcal{S})$ is  $\varrho[(\gamma^{-1}\circ\beta\circ\gamma\circ\nabla)_{[x_{i}]}(\mathcal{S})]$ so that the unionization stage is the index of the normalization stage of the expansion 
\begin{align}
j\geq \left \lfloor \frac{\Phi[(\gamma^{-1}\circ\beta\circ\gamma\circ\nabla)_{[x_{i}]}(\mathcal{S})]}{\varrho[(\gamma^{-1}\circ\beta\circ\gamma\circ\nabla)_{[x_{i}]}(\mathcal{S})]}\right \rfloor.\nonumber
\end{align}
\end{proof}
\bigskip

Here, the \emph{normalization} stage is analogous to and parallels the notion of the limit of an expansion under the single variable theory. Similarly, the local number also parallels the notion of the \emph{unionization} stage in multivariate expansivity theory.

\begin{proposition}\label{isomorphism}
Let $\{(\gamma^{-1}\circ\beta\circ\gamma\circ\nabla)_{[x_{i}]}(\mathcal{S}_i)\}_{i=1}^{\infty}$ be a collection of expansions in the direction $x_i$ of $l$-tuples of polynomials in the ring $\mathbb{R}[x_1,x_2,\ldots,x_n]$ and $\{(\gamma^{-1}\circ\beta\circ\gamma\circ\nabla)(\mathcal{S}'_i)\}_{i=1}^{\infty}$ be a collection of expansions of tuples of polynomials in the ring $\mathbb{R}[x]$. The map 
\begin{align}
\chi_{(a_1,a_2,\ldots,a_{i-1},a_{i+1},\ldots,a_n)}:\{(\gamma^{-1}\circ\beta\circ\gamma\circ\nabla)_{[x_{i}]}(\mathcal{S}_i)\}_{i=1}^{\infty}\longrightarrow \{(\gamma^{-1}\circ\beta\circ\gamma\circ\nabla)(\mathcal{S}'_i)\}_{i=1}^{\infty}\nonumber
\end{align}
for a fixed $(a_1,a_2,\ldots,a_{i-1},a_{i+1},\ldots,a_n)\in \mathbb{R}^{n-1}$ such that for any $(\gamma^{-1}\circ\beta\circ\gamma\circ\nabla)_{[x_{i}]}(\mathcal{S})\in \{(\gamma^{-1}\circ\beta\circ\gamma\circ\nabla)_{[x_{i}]}(\mathcal{S}_i)\}_{i=1}^{\infty}$ then 
\begin{align}
\chi_{(a_1,a_2,\ldots,a_{i-1},a_{i+1},\ldots,a_n)}(\gamma^{-1}\circ\beta\circ\gamma\circ\nabla)_{[x_{i}]}(\mathcal{S})&=(\gamma^{-1}\circ\beta\circ\gamma\circ\nabla)(\mathcal{S})(a_1,a_2,\ldots, a_{i-1},\nonumber \\& a_{i+1},\ldots,a_n)\in \{(\gamma^{-1}\circ\beta\circ\gamma\circ\nabla)(\mathcal{S}'_i)\}_{i=1}^{\infty}\nonumber
\end{align}
is an isomorphism. We denote the isomorphism by 
\begin{align}
(\gamma^{-1}\circ\beta\circ\gamma\circ\nabla)_{[x_{i}]}(\mathcal{S}_i)\simeq (\gamma^{-1}\circ\beta\circ\gamma\circ\nabla)(\mathcal{S}'_i).\nonumber
\end{align}
\end{proposition}
 \bigskip
 
\begin{proposition}\label{upper bound for normalization stage}
Let $\mathcal{F}=\{\mathcal{S}_i\}_{i=1}^{\infty}$ be a collection of $l$-tuples of polynomials in the ring $\mathbb{R}[x_1,x_2,\ldots,x_n]$ and $(\gamma^{-1}\circ\beta\circ\gamma\circ\nabla)_{[x_{i}]}(\mathcal{S})$ be an expansion. We have
\begin{align}
\varrho[(\gamma^{-1}\circ\beta\circ\gamma\circ\nabla)_{[x_{i}]}(\mathcal{S})]\leq 2.\nonumber
\end{align}
\end{proposition}
\bigskip

\begin{proof}
Let $\mathcal{S}\in \mathcal{F}=\{\mathcal{S}_i\}_{i=1}^{\infty}$ then we fix all other directions $x_j$ for all $j\neq i$ so that the expansion 
\begin{align}
(\gamma^{-1}\circ\beta\circ\gamma\circ\nabla)_{[x_{i}]}(\mathcal{S})\simeq (\gamma^{-1}\circ\beta\circ\gamma\circ\nabla)(\mathcal{S})(a_1,a_2,\ldots, a_{i-1},a_{i+1},\ldots,a_n)\nonumber
\end{align}
where $(\mathcal{S})(a_1,a_2,\ldots, a_{i-1},a_{i+1},\ldots,a_n)$ is a tuple of polynomials in the ring $\mathbb{R}[x_i]$. Using the lemma \ref{local number}, we obtain the inequality 
\begin{align}
\mathcal{L}\bigg[(\gamma^{-1}\circ\beta\circ\gamma\circ\nabla)(\mathcal{S})(a_1,a_2,\ldots, a_{i-1},a_{i+1},\ldots,a_n)\bigg]\leq 2\nonumber
\end{align}
and by appealing to Proposition \ref{isomorphism}, we recover the following inequality
\begin{align}
\varrho[(\gamma^{-1}\circ\beta\circ\gamma\circ\nabla)_{[x_{i}]}(\mathcal{S})]\leq 2.\nonumber
\end{align}
\end{proof}
\bigskip

Here, we show that the unionization stage of a typical expansion cannot be too big. In other words, it cannot possibly be the case that the totient of an expansion in a specific direction coincides with the unionization stage.

\begin{theorem}[Analytic range]\label{analytic exitence law}
Let $\mathcal{F}=\{\mathcal{S}_i\}_{i=1}^{\infty}$ be a collection of tuples of polynomials in the ring $\mathbb{R}[x_1,x_2,\ldots,x_n]$ and $(\gamma^{-1}\circ\beta\circ\gamma\circ\nabla)_{[x_{i}]}(\mathcal{S})$ be an expansion. The expansion is analytic in the range
\begin{align}
j\geq \left \lfloor \frac{\Phi[(\gamma^{-1}\circ\beta\circ\gamma\circ\nabla)_{[x_{i}]}(\mathcal{S})]}{2}\right \rfloor.\nonumber
\end{align}
\end{theorem}
\bigskip

\begin{proof}
Every expansion $(\gamma^{-1}\circ\beta\circ\gamma\circ\nabla)_{[x_{i}]}(\mathcal{S})$ is always analytic at the unionization. Using the proposition \ref{unionization stage proof}, we get
\begin{align}
j&\geq\left\lfloor \frac{\Phi[(\gamma^{-1}\circ\beta\circ\gamma\circ\nabla)_{[x_{i}]}(\mathcal{S})]}{\varrho[(\gamma^{-1}\circ\beta\circ\gamma\circ\nabla)_{[x_{i}]}(\mathcal{S})]}\right\rfloor \nonumber \\& \geq \left \lfloor \frac{\Phi[(\gamma^{-1}\circ\beta\circ\gamma\circ\nabla)_{[x_{i}]}(\mathcal{S})]}{2}\right\rfloor \nonumber
\end{align}
using proposition \ref{upper bound for normalization stage}. The expansion $(\gamma^{-1}\circ\beta\circ\gamma\circ\nabla)^j_{[x_{i}]}(\mathcal{S})$ is analytic in the range 
\begin{align}
j\geq \left\lfloor\frac{\Phi[(\gamma^{-1}\circ\beta\circ\gamma\circ\nabla)_{[x_{i}]}(\mathcal{S})]}{2}\right \rfloor.\nonumber
\end{align}
\end{proof}

\subsection{An application to function theory}

\begin{corollary}\label{existence of singularity}
Let $\mathcal{F}=\{\mathcal{S}_i\}_{i=1}^{\infty}$ be a collection of tuples of polynomials in the ring $\mathbb{R}[x_1,x_2,\ldots,x_n]$ and $(\gamma^{-1}\circ\beta\circ\gamma\circ\nabla)^j_{[x_{i}]}(\mathcal{S})$ be an expansion. For
\begin{align}
j<\left \lfloor \frac{\Phi[(\gamma^{-1}\circ\beta\circ\gamma\circ\nabla)_{[x_{i}]}(\mathcal{S})]}{2}\right \rfloor \nonumber
\end{align}
there exist some $(f_1,f_2,\ldots,f_l)\in \mathrm{Ker}[(\gamma^{-1}\circ\beta\circ\gamma\circ\nabla)^j_{[x_{i}]}(\mathcal{S})]$ 
and 
\begin{align}
(a_1,a_2,\ldots,a_{k-1},a_{k+1},\ldots,a_n)\in\mathbb{C}^{n-1}\nonumber
\end{align} 
such that 
\begin{align}
f_i[(a_1,\ldots a_{k-1},a_{k+1},\ldots,a_n)]=\infty \nonumber
\end{align}
for some $1\leq i\leq l$.
\end{corollary}
\bigskip


\begin{figure}[htbp]
\centering
\begin{tikzpicture}[
    font=\small,
    >=Latex,
    node distance=12mm and 18mm,
    box/.style={draw, rounded corners, align=center, minimum width=38mm, minimum height=11mm, inner sep=4pt},
    bigbox/.style={draw, rounded corners, align=center, minimum width=50mm, minimum height=13mm, inner sep=5pt},
    smallbox/.style={draw, rounded corners, align=center, minimum width=28mm, minimum height=9mm, inner sep=3pt}
]

\node[bigbox] (S) {$\mathcal S=(f_1,\dots,f_n)$\\Tuple of polynomials in $\mathbb R[x_1,\dots,x_n]$};
\node[box, right=28mm of S] (dir) {$T_{[x_i]}(\mathcal S)$\\Directional expansion};
\node[box, below=18mm of dir] (mix) {$T_{[x_{\sigma(1)}]}\otimes\cdots\otimes T_{[x_{\sigma(l)}]}(\mathcal S)$\\Mixed expansion};

\node[box, right=28mm of mix] (spec) {$x_j=a_j$ for $j\neq i$\\Specialization to one variable};
\node[bigbox, right=28mm of spec] (sv) {$\mathbb R[x]$\\Single-variable expansivity theory};

\node[smallbox, above left=10mm and -2mm of mix] (phi) {$\Phi$\\Totient};
\node[smallbox, above=10mm of mix] (dop) {$\mathcal I$\\Doppler intensity};
\node[smallbox, above right=10mm and -2mm of mix] (exa) {Exactness};
\node[smallbox, below left=10mm and -2mm of mix] (ker) {Kernel};
\node[smallbox, below=10mm of mix] (sing) {Singularity};
\node[smallbox, below right=10mm and -2mm of mix] (dom) {Dominating number};

\draw[->, thick] (S) -- node[above] {apply $T_{[x_i]}$} (dir);
\draw[->, thick] (dir) -- node[right] {iterate / mix directions} (mix);

\draw[->, thick] (mix) -- (phi);
\draw[->, thick] (mix) -- (dop);
\draw[->, thick] (mix) -- (exa);
\draw[->, thick] (mix) -- (ker);
\draw[->, thick] (mix) -- (sing);
\draw[->, thick] (mix) -- (dom);

\draw[->, thick] (mix) -- node[above] {fix all other variables} (spec);
\draw[->, thick] (spec) -- node[above] {recover the 1-variable theory} (sv);

\node[align=center, below=16mm of S, text width=48mm] (leftnote)
{The multivariate theory tracks how a tuple changes under repeated directional expansion.};

\node[align=center, below=16mm of sv, text width=48mm] (rightnote)
{Specialization collapses the multivariate picture to the one-variable setting.};

\end{tikzpicture}
\caption{Schematic intuition for multivariate expansivity theory. The theory starts with a tuple of polynomials, applies directional and mixed expansions, and then attaches invariants that measure extinction, interaction, stabilization, and singular behavior. Specializing the other variables recovers the single-variable theory.}
\label{fig:intuition-expansivity}
\end{figure}

\bibliographystyle{amsplain}

\begin{thebibliography}{10}

\bibitem{Mahé1984} L.~Mahé, \textit{On the Pierce-Birkhoff conjecture}, Rocky Mountain J. Math., vol. 14:4, JSTOR, 1984, 983--986.\\

\bibitem{Delzell1989} C.~N.~Delzell, \textit{On the Pierce-Birkhoff conjecture over ordered fields}, Rocky Mountain J. Math., vol. 19:3, JSTOR, 1989, 651--668.\\

\bibitem{Marshall1992} M.~Marshall, \textit{The Pierce-Birkhoff conjecture for curves}, Canad. J. Math., vol. 44:6, Cambridge University Press, 1992, 1262--1271.\\

\bibitem{Wagner2010} S.~Wagner, \textit{On the Pierce-Birkhoff conjecture for smooth affine surfaces over real closed fields},
Ann. Fac. Sci. Toulouse Math. vol. 19:S1, 2010, 221--242.\\

\bibitem{LucasSchaubSpivakovsky2012} F.~Lucas, D.~Schaub, and M.~Spivakovsky, \textit{On the Pierce-Birkhoff conjecture}, Journal of Algebra, vol. 435, Elsevier, 2015, 124--158.\\

\bibitem{HenriksenIsbell1962} M.~Henriksen and J.~R.~Isbell, \textit{Lattice-ordered rings and function rings}, Pacific J. Math., vol. 12, 1962, 533--565.\\

\bibitem{BirkhoffPierce1956} G.~Birkhoff and R.~S.~Pierce, \textit{Lattice-ordered rings}, Anais Acad. Bras. Ci., vol. 28: 41-69, 1956, 420.\\

\end{thebibliography}

\end{document}